\documentclass{amsart}
\calclayout
\usepackage{amssymb,amsmath,amsfonts,epsfig,latexsym,tikz}
\usepackage{tikz-cd}
\usepackage[hidelinks]{hyperref}
\usepackage{enumerate}
\usepackage{mathtools}
\usepackage{verbatim}
\usepackage{cleveref}
\usepackage[shortlabels]{enumitem}
\usepackage{subcaption}
\usepackage{stmaryrd}

\usetikzlibrary{positioning}
\usetikzlibrary{matrix}
\usetikzlibrary{decorations}
\usetikzlibrary{decorations.pathreplacing, decorations.pathmorphing, angles,quotes}
\usetikzlibrary {arrows.meta} 
\usetikzlibrary{patterns} 

\newtheorem{theorem}{Theorem}[subsection]
\newtheorem{definition}[theorem]{Definition}
\newtheorem{proposition}[theorem]{Proposition}
\newtheorem{lemma}[theorem]{Lemma}
\newtheorem{corollary}[theorem]{Corollary}

\newtheorem*{LMMC}{Local Motivic Monodromy Conjecture}
\newtheorem*{LPMC}{Local $\p$-adic Monodromy Conjecture}
\newtheorem*{LTMC}{Local Topological Monodromy Conjecture}

\def\A{\mathbb{A}}
\def\C{\mathbb{C}}
\def\F{\mathbb{F}}
\def\N{\mathbb{N}}
\def\P{\mathbb{P}}
\def\Q{\mathbb{Q}} 
\def\R{\mathbb{R}}
\def\T{\mathcal{T}}
\def\Z{\mathbb{Z}}
\def\L{\mathbb{L}}

\def\p{\mathfrak{p}}

\def\t{\mathrm{top}}
\def\mot{\mathrm{mot}}
\def\for{\mathrm{for}}
\def\mh{{\hat{\mu}}}

\def\M{\mathcal{M}^\mh}
\def\tE{\widetilde{E}}

\def\1{\mathbf{1}}

\def\cF{\mathcal{F}}

\def\mot{\mathrm{mot}}

\def\UB{U\hspace{-1 pt}B_1}
\def\k{\Bbbk}

\def\cP{\mathcal P}

\def\<{\langle}
\def\>{\rangle}

\DeclareMathOperator{\bbox}{box}
\DeclareMathOperator{\BBox}{Box}

\DeclareMathOperator{\codim}{codim}
\DeclareMathOperator{\Span}{span}
\DeclareMathOperator{\spn}{span}

\DeclareMathOperator{\Hom}{Hom}

\DeclareMathOperator{\lk}{lk}

\DeclareMathOperator{\Star}{Star}
\DeclareMathOperator{\supp}{supp}
\DeclareMathOperator{\Sym}{Sym}

\DeclareMathOperator{\Vol}{Vol}

\DeclareMathOperator{\One}{\mathbf{1}}
\DeclareMathOperator{\pr}{pr}

\DeclareMathOperator{\Contrib}{Contrib}

\DeclareMathOperator{\rec}{rec}
\DeclareMathOperator{\Newt}{Newt}
\DeclareMathOperator{\ver}{Vert}
\DeclareMathOperator{\Spec}{Spec}
\DeclareMathOperator{\Ima}{im}
\DeclareMathOperator{\csp}{sp}

\DeclareMathOperator{\gens}{Gen}
\DeclareMathOperator{\Unb}{Unb}
\DeclareMathOperator{\aff}{aff}

\newcommand{\Conv}[1]{\operatorname{Conv}\left\{{#1}\right\}}

\theoremstyle{definition}
\newtheorem{remark}[theorem]{Remark}
\newtheorem{example}[theorem]{Example}

\begin{document}

\title[The local motivic monodromy conjecture for simplicial nondegenerate singularities]{The local motivic monodromy conjecture for simplicial nondegenerate singularities}          
\date{\today}
\author{Matt Larson, Sam Payne, and Alan Stapledon}

\address{Princeton University Department of Mathematics, Fine Hall, Washington Road, Princeton, NJ 08540}
\email{mattlarson@princeton.edu}

\address{Department of Mathematics, University of Michigan, 530 Church St., Ann Arbor, MI 48109}
\email{sdpayne@umich.edu}

\address{Sydney Mathematics Research Institute, L4.42, Quadrangle A14, University of Sydney, NSW 2006, Australia}
\email{astapldn@gmail.com}

\begin{abstract}
We prove the local motivic monodromy conjecture for singularities that are nondegenerate with respect to a simplicial Newton polyhedron. It follows that all poles of the local topological zeta functions of such singularities correspond to eigenvalues of monodromy acting on the cohomology of the Milnor fiber of some nearby point, as do the poles of Igusa's local $p$-adic zeta functions for large primes $p$.  
\end{abstract}

\maketitle

\vspace{-20 pt}

\setcounter{tocdepth}{1}
\tableofcontents

\section{Introduction}

Singularities are a central and pervasive phenomenon in geometry, and one prototypical setting in which they arise is the fiber of a smooth function over an isolated critical value. Monodromy studies how the nearby smooth fibers vary as one moves around such a critical value. Since Milnor’s foundational work on complex hypersurface singularities \cite{Milnor68}, monodromy has served as a bridge between topological, geometric, and analytic perspectives. It reveals subtle invariants of singularities, particularly through the linear action on the cohomology of the Milnor fiber, a topological invariant introduced by Milnor. The monodromy conjectures predict that certain analytic invariants of singularities, encoded in local zeta functions, detect eigenvalues of this monodromy action. These conjectures have guided extensive work over the past four decades on the interaction between resolution data, Newton polyhedra, and monodromy.

The conjectures appear in several closely related forms. In the archimedean setting, a theorem of Malgrange shows that poles of certain oscillatory integrals give rise to monodromy eigenvalues \cite{Malgrange74}. In the nonarchimedean setting, Igusa introduced $p$-adic local zeta functions to study congruence counting problems \cite{Igusa75}, and it was observed by Igusa and Denef that the real parts of their poles appear to be governed by monodromy eigenvalues \cite{Denef85, Igusa88}. This observation led to the local $p$-adic monodromy conjecture \cite{Denef91b}. A topological variant, formulated by Denef and Loeser using Euler characteristics of resolution strata, is known as the local topological monodromy conjecture \cite{DenefLoeser92}. All of these conjectures remain wide open in general.

Motivic integration provides a unifying framework for these observed and predicted phenomena. By replacing numerical invariants such as volumes or Euler characteristics with classes in the Grothendieck ring of varieties, Denef and Loeser introduced the local motivic zeta function, a refinement that simultaneously encodes arithmetic, topological, and geometric information \cite{DenefLoeser98}. The corresponding local motivic monodromy conjecture predicts that poles of this motivic zeta function give rise to eigenvalues of monodromy. By specialization from motivic zeta functions to $p$-adic and topological zeta functions, the motivic conjecture implies both the topological conjecture and the $p$-adic conjecture in cases of good reduction.

A useful way to view the monodromy conjectures—especially the $p$-adic monodromy conjecture—is as a local analogue of the Riemann hypothesis for varieties over finite fields (i.e., the last of the Weil conjectures), famously proved by Deligne \cite{Deligne74}. In both settings, one begins with a generating function defined by counting solutions to polynomial equations. In the Weil conjectures, the Hasse-Weil zeta function records the number of points of a variety over algebraic extensions of a finite field, while in the $p$-adic setting, Igusa’s local zeta function records the number of solutions to a polynomial congruence modulo $p^n$, for all $n$. The expectation in each case is that the analytic behavior of this generating function—specifically, the location of its poles—is governed by underlying cohomological data. For smooth projective varieties, the Riemann hypothesis proved by Deligne says that the zeta function is controlled by the action of Frobenius on \'etale cohomology. Analogously, the $p$-adic monodromy conjecture predicts that the poles of Igusa’s local zeta function reflect the action of geometric monodromy on the cohomology of Milnor fibers. 

The motivic zeta function takes this generating function perspective one step further, and may be viewed as a universal cohomological zeta function: it interpolates between arithmetic and topological realizations and expresses the principle that local counting problems, even at the level of congruences modulo powers of a prime, are governed by the geometry and topology of singularities.

Despite the conceptual elegance of the motivic monodromy conjecture and the close analogy between the $p$-adic monodromy conjecture and the Weil conjectures, which were proved more than fifty years ago, both of these conjectures remain wide open; very few cases are known. One basic obstacle on the motivic side is that motivic zeta functions take values in Grothendieck rings, which have zero divisors, making even the notion of a pole subtle. Progress on these conjectures has relied on identifying classes of singularities for which the zeta functions and monodromy can be analyzed explicitly.

One such class is the Newton nondegenerate singularities. The condition of nondegeneracy is generic, i.e., an open and dense subset of the polynomials with fixed Newton polytope are Newton nondegenerate.  Following the work of Kouchnirenko and Varchenko \cite{Kouchnirenko76, Varchenko76}, nondegeneracy ensures that many invariants of a singularity are governed by the combinatorial geometry of its Newton polyhedron. In this setting, local zeta functions admit explicit descriptions in terms of faces of the polyhedron, and monodromy eigenvalues can often be studied using toric geometry and combinatorial topology.

In this paper, we prove the local motivic monodromy conjecture for nondegenerate singularities whose Newton polyhedra are simplicial. This additional assumption on the Newton polyhedron enables us to incorporate techniques from Ehrhart theory and the theory of local $h$-polynomials. Our main result establishes that one can choose a set of candidate poles for the local motivic zeta function such that each candidate pole corresponds to a nearby eigenvalue of monodromy (Theorem~\ref{thm:lmmsimplicial}).

More precisely, if $f$ is a nondegenerate polynomial whose Newton polyhedron is simplicial, then every candidate pole of its local motivic zeta function gives rise to an eigenvalue of monodromy at the origin or at a nearby point on the hypersurface. As immediate consequences, we recover the local topological monodromy conjecture for this class of singularities (Theorem~\ref{thm:ltmsimplicial}), and the local $p$-adic monodromy conjecture for almost all primes (Theorem~\ref{thm:lpmsimplicial}).

The proof has two complementary components. First, we establish existence results for monodromy eigenvalues using a nonnegative, combinatorial formula for alternating sums of reduced cohomology along coordinate strata (Sections~\ref{sec:nearby}-\ref{sec:vanishing}). This formula is expressed in terms of Ehrhart data and local $h$-polynomials, and avoids the cancellation phenomena that complicate earlier approaches based on monodromy zeta functions. Second, we identify and eliminate fake poles of the motivic zeta function by introducing a local formal zeta function, whose algebraic properties make it possible to intersect sets of candidate poles and systematically remove those not forced by monodromy (Sections~\ref{sec:poles}-\ref{sec:fakepoles}).

Although our main theorems are stated for simplicial Newton polyhedra, many of the arguments apply more broadly. In particular, we obtain additional cases of the motivic monodromy conjecture in dimension three, and we isolate precise combinatorial obstructions that arise in higher dimensions (Section~\ref{sec:beyondsimplicial}). These results clarify the role of simpliciality and highlight the remaining challenges in extending the conjecture to arbitrary nondegenerate singularities.

\subsection{Main results}

Throughout, let $\k$ be a field of characteristic 0, and let $f \in \k[x_1, \ldots, x_n]$ be a regular function whose vanishing locus $X_f$ contains $0 \in \A^n$.  The coefficients of $f$ are contained in a finitely generated subfield $\k' \subset \k$, so we may choose an embedding $\k' \subset \C$, view $f$ as a holomorphic function on $\C^n$, and consider the Milnor fiber $\cF_{x}$, with its monodromy action, for any geometric point $x \in X_f$.  The characteristic polynomial of the induced action on $H^*(\cF_{x}, \C)$ is independent of all choices and its zeros are the \emph{eigenvalues of monodromy} of $f$ at $x$.  The monodromy is quasi-unipotent, so all such eigenvalues of monodromy are roots of unity. 
We say that $\exp(2 \pi i \alpha)$ is a \emph{nearby eigenvalue of monodromy} of $f$ if  $0$ lies in the Zariski closure of the locus of points $x \in X_f$ such that $\exp (2 \pi i \alpha)$ is an eigenvalue of monodromy of $f$ at $x$.

The local motivic zeta function is a subtle invariant of the singularity of $f$ at $0$, introduced by Denef and Loeser \cite{DenefLoeser98}. Let $K^\mh$ be the Grothendieck ring of $\k$-varieties with good $\mh$-action, where $\mh = \varprojlim \mu_m$ is the inverse limit of the groups of $m$th roots of unity, and let $\M := K^\mh[ \L^{-1}]$ be the associated motivic ring obtained by inverting $\L := [\A^1]$. Then the local motivic zeta function $Z_\mot(T) \in \M \llbracket T \rrbracket$ is expressible non-uniquely as the formal power series expansion of a rational function in $\M\big[T, \frac{1}{1-\L^a T^b}\big]_{(a,b) \in \Z \times \Z_{>0}, a/b \in \cP}$, for some finite $\cP \subset \Q$.  Any such $\cP$ is a \emph{set of candidate poles} for $Z_\mot(T)$, as defined in \cite{BoriesVeys16,BultotNicaise20}.  

\begin{LMMC}
There is a set of candidate poles $\cP \subset \Q$ for $Z_\mot(T)$ such that, for every $\alpha \in \cP$, $\exp(2\pi i \alpha)$ is a 
nearby eigenvalue of monodromy. 
\end{LMMC}

\noindent Note that the notion of poles is subtle in this context because $K^\mh$ is not known to be 
an integral domain; in particular, it is unclear whether the intersection of two sets of candidate poles for $Z_\mot(T)$ is necessarily a set of candidate poles.  Our first main result (Theorem~\ref{thm:lmmsimplicial}) confirms the local motivic monodromy conjecture for singularities that are nondegenerate with respect to a simplicial Newton polyhedron.  

For $u = (u_1, \ldots, u_n)$ in $\Z_{\geq 0}^n$, let $x^u := x_1^{u_1} \cdots x_n^{u_n}$, and write  $f = \sum_u a_u x^u$. 
 The Newton polyhedron of $f$, denoted $\Newt(f)$, is the Minkowski sum
$\mathrm{conv} \{ u : a_u \neq 0 \} + \R_{\geq 0}^n.$ 
For each face $F$ of $\Newt(f)$, we consider $f|_F := \sum_{u \in F} a_u x^u.$ Then $f$ is \emph{nondegenerate} if, for all compact faces $F$, the vanishing locus of $f|_F$ has no singularities in the complement of the coordinate hyperplanes in $\A^n$. 

For any face $F$ of $\partial \Newt(f)$ that meets the interior of the orthant $\R_{>0}^n$, let $C_F := \overline {\R_{\geq 0} F}$ be the closure of the cone spanned by $F$. The set of all faces of such cones forms a fan $\Delta$
 whose support is the positive orthant $\R_{\geq 0}^n$.  We say that $\Newt(f)$ is \emph{simplicial} if $\Delta$ is a simplicial fan.

\begin{theorem} \label{thm:lmmsimplicial}
Suppose that $\Newt(f)$ is simplicial and $f$ is nondegenerate.  Then there is a set of candidate poles $\cP \subset \Q$ for $Z_\mot(T)$ such that, for every $\alpha \in \cP$, $\exp(2\pi i \alpha)$ is 
a nearby eigenvalue of monodromy.
\end{theorem}

\noindent In other words, the local motivic monodromy conjecture is true for any nondegenerate singularity with a simplicial Newton polyhedron.  This was known previously for $n = 2$ \cite{BultotNicaise20}. 
Our definition of simplicial Newton polyhedron agrees with that in \cite{JKYS19}.  A \emph{convenient} Newton polyhedron, i.e., one that intersects each of the coordinate axes \cite{Kouchnirenko76}, is simplicial if and only if each of its compact faces is a simplex.

\begin{remark}
The methods used in the proof of Theorem~\ref{thm:lmmsimplicial} are discussed in Section~\ref{sec:methods}. Roughly speaking, we have one collection of arguments, presented in Sections~\ref{sec:nearby}-\ref{sec:vanishing} that proves the existence of eigenvalues corresponding to candidate poles associated to many facets of $\Newt(f)$.  Another collection of arguments, presented in Sections~\ref{sec:poles}-\ref{sec:fakepoles}, shows that certain such candidate poles are \emph{fake} and can be removed to give a smaller set of candidate poles.  Each of these arguments is carried out not only for simplicial Newton polyhedra, but in somewhat greater generality. As a result, we are able to prove a range of cases of the local motivic monodromy conjecture where $f$ is nondegenerate and $\Newt(f)$ is not necessarily simplicial, including all such cases for $n = 3$. See Section~\ref{sec:beyondsimplicial} for details.  
\end{remark}

The local motivic monodromy conjecture is a motivic analogue of the local $p$-adic and topological monodromy conjectures, and the following cases of the latter conjectures are consequences of Theorem~\ref{thm:lmmsimplicial}. 

The local motivic zeta function specializes to the local topological zeta function $Z_{\t}(s) \in \Q(s)$ by expanding $Z_{\mot}(T)$ as a power series in $\L - 1$ and setting $T \mapsto \L^{-s}$ and $[Y] \mapsto \chi(Y/\mh)$ \cite[Section~2.3]{DenefLoeser98}. 
It follows that the poles of $Z_{\t}(s)$ are contained in every set of candidate poles for $Z_{\mot}(T)$.

\begin{theorem} \label{thm:ltmsimplicial}
Suppose $\Newt(f)$ is simplicial and $f$ is nondegenerate. If $\alpha$ is a pole of $Z_{\t}(s)$, then $\exp(2\pi i \alpha)$ is a nearby eigenvalue of monodromy.
\end{theorem}

\noindent This confirms the local topological monodromy conjecture \cite[Conjecture~3.3.2]{DenefLoeser92} for singularities that are nondegenerate with respect to a simplicial Newton polyhedron.  

If $f \in \Z_p[x_1, \ldots, x_n]$ has good reduction mod $p$, i.e., if $\overline f \in \F_p[x_1, \ldots, x_n]$ is nondegenerate with $\Newt(\overline f) = \Newt(f)$, then $Z_{\mot}(T)$ also specializes to the Igusa local $p$-adic zeta function $Z_{(p)}(s) \in \Q(p^s)$, which is viewed as a global meromorphic function in the complex variable $s$.  In this case, the real part of any pole of $Z_{(p)}(s)$ is contained in every set of candidate poles for $Z_{\mot}(T)$.

\begin{theorem} \label{thm:lpmsimplicial}
Suppose $f \in \Z_p[x_1, \ldots, x_n]$, $\Newt(f)$ is simplicial, and $f$ is nondegenerate with good reduction mod $p$. If $\alpha$ is a pole of $Z_{(p)}(s)$, then $\exp(2\pi i \Re(\alpha))$ is 
a nearby eigenvalue of monodromy.
\end{theorem}

\noindent If $f \in \Z[x_1, \ldots, x_n]$ is nondegenerate, then $f$ has good reduction mod $p$ for all but finitely many primes $p$. In this sense, Theorem~\ref{thm:lpmsimplicial} implies that the local $p$-adic  monodromy conjecture holds for nondegenerate singularities with simplicial Newton polyhedra.

\subsection{Background and motivation}

We now discuss the background and motivation for the local monodromy conjectures in more detail. In particular, we recall the definitions of the local motivic, $p$-adic, and topological zeta functions and how they relate to the geometry of embedded log resolutions. We also recall A'Campo's formula for the zeta function of monodromy at the origin.

\subsubsection{Archimedean zeta functions}
The motivation for the local monodromy conjectures comes from a theorem of Malgrange concerning the following archimedean analogues of local zeta functions.  Suppose $\k = \R$ or $\C$, and let $\Phi$ be a smooth function supported on a compact set that does not contain any critical points of $f$ other than $0$.  Consider the function 
\[
Z_\Phi(s) := \int_{\k^n} \Phi(x) | f(x)|^{\delta s} dx,
\]
where $\delta = 1$ if $\k = \R$ and $\delta = 2$ if $\k = \C$.  This integral converges for $s \in \C$ with $\Re(s) > 0$.

\begin{theorem}[\cite{Malgrange74}] \label{thm:Malgrange}
The function $Z_\Phi(s)$ extends to a meromorphic function on $\C$ whose poles are rational numbers. Moreover, if $\alpha$ is a pole of $Z_\Phi(s)$, then $\exp(2 \pi i \alpha)$ is an eigenvalue of monodromy.
\end{theorem}

\noindent  Furthermore, every eigenvalue of monodromy of $f$ at the origin corresponds to a pole of $Z_\Phi(s)$ for some $\Phi$. 

\subsubsection{Log resolutions and zeta functions of monodromy}

Let $h  \colon \cF_x \to \cF_x$ denote the monodromy action on the Milnor fiber of $f$ at $x \in X_f$.  The zeta function of monodromy of $f$ at $x$ is then
\begin{equation} \label{eq:HWFrob}
\zeta_x(t) := \frac{\det\big( 1-t h^* \, | \, H^{\mathrm{even}}(\cF_x, \C) \big)}{\det\big( 1-t h^*  \, | \, H^{\mathrm{odd}}(\cF_x, \C) \big)}.
\end{equation}
The zeta function of monodromy at $0$ may be expressed in terms of the numerical data of a log resolution and the topological Euler characteristics of the strata in the fiber, as follows.

Let $\pi \colon Y \to \A^n$ be a proper morphism that is an isomorphism away from $X_f$, such that the support of $D := \pi^{-1}(X_f)$ is a divisor with simple normal crossings. Let $D_1, \ldots, D_r$ be the irreducible components of $D$. The associated \emph{numerical data} of this log resolution are the pairs of integers $(N_i, \nu_i)$, where $N_i$ and $\nu_i -1$ are the orders of vanishing of $\pi^*(f)$ and $\pi^* (dx_1 \wedge \cdots \wedge dx_n)$, respectively, along $D_i$.

For $I \subset \{1, \ldots, r\}$, let $D_I := \bigcap_{i \in I} D_i$ and $D_I^\circ := D_I \smallsetminus \bigcup_{j \not \in I} D_{I \cup \{ j \}}.$ We then define $$E_I := D_I \cap \pi^{-1}(0) \mbox{ and } E_I^\circ = D_I^\circ \cap \pi^{-1}(0)$$ for the corresponding closed and locally closed strata in the fiber over $0$.

\begin{theorem}[\cite{ACampo75}]
The zeta function of monodromy acting on the cohomology of $\cF_0$ is
\begin{equation} \label{eq:ACampo}
\zeta_0(t) = \prod_{i = 1}^r (1-t^{N_i})^{\chi(E_i^\circ)},
\end{equation}
where we omit the corresponding term if $E_i^{\circ}$ is empty. 
\end{theorem}

\noindent Note that the exponents $\chi(E_i^\circ)$ can be positive or negative, and there can be a great deal of cancellation in simplifying this rational function expression for $\zeta_0(t)$ down to a quotient of two relatively prime polynomials.  In particular, it is difficult to determine from the numerical data of the log resolution whether any given root of $1- t^{N_i}$ is an eigenvalue of monodromy.

\subsubsection{Igusa's local $\p$-adic zeta functions}
Let $\k$ be a finite extension of $\Q_p$, equipped with the unique extension of the $p$-adic valuation and its associated norm. Let $R \subset \k$ be the valuation ring, with $\p \subset R$ the maximal ideal.  For instance, if $f \in \mathbb{Q}_p[x_1, \dotsc, x_n]$, then $R = \Z_p$ and $\p = p\Z_p$.  Igusa introduced and studied the local zeta function
\[
Z_{\p}(s) := \int_{\p^n} |f(x)|^s dx,
\]
in his proof of a conjecture of Borewicz and Shafarevich \cite[p.~63]{BorewiczShafarevich66} on rationality of generating functions for the number of solutions mod $p^m$ to a polynomial equation with integer coefficients. Here, $dx$ denotes the normalized Haar measure on the compact additive group $\p^n$.  Note that $Z_{\p}(s)$ is a  nonarchimedean analogue of the asymptotic integrals $Z_\Phi(s)$; the role of $\Phi$ is played by the indicator function of the compact subset $\p^n$. 
Igusa proved that $Z_{\p}(s)$ is a rational function in $q^{-s}$ \cite{Igusa75}, where $q = |R/\p|$, and that its real poles other than $-1$ are all of the form $\alpha_i := -\nu_i/N_i$, where $(N_i, \nu_i)$ is the numerical data associated to some exceptional divisor in a given log resolution \cite{Igusa78}. Denef gave a second proof of the rationality of $Z_{\p}(s)$, using $p$-adic cell decompositions \cite{Denef84}.  

\begin{remark}
Note that some sources in the literature define a local $\p$-adic zeta function by integrating with respect to the restriction of the normalized Haar measure on $R^n$; the result differs from our $Z_\p(s)$ by a factor of $q^{-n}$.  Such renormalizations do not affect the poles of the local zeta functions. 
\end{remark}

Typically, very few of the rational numbers $\alpha_i$ associated to the numerical data in a log resolution are actually poles of $Z_\p(s)$.  In the archimedean setting, this is explained by Malgrange's theorem (Theorem~\ref{thm:Malgrange}), since many rational numbers that appear in this way do not correspond to eigenvalues of monodromy.

Both Denef \cite{Denef85} and Igusa \cite{Igusa88} observed that the analogue of Malgrange's theorem seems to hold for $Z_\p(s)$; in all examples that had been computed, whenever $\alpha_j$ is a pole of $Z_\p(s)$, the corresponding root of unity $\exp(2\pi i \alpha_j)$ is an eigenvalue of monodromy. Loeser proved that this is true for $n = 2$ \cite{Loeser88} and for certain  nondegenerate singularities in higher dimensions  \cite{Loeser90}. 
By the early 1990s, the expectation that this nonarchimedean analogue of Malgrange's theorem should hold was known as the monodromy conjecture.  See, e.g., \cite[Conjecture~4.3]{Denef91b}, \cite[Conjecture~2.3.2]{Denef91}, and \cite[p.~546--547]{Veys93}.  We will follow the usual convention and call this the \emph{local $\p$-adic monodromy conjecture} to distinguish it from the topological and motivic variants that followed.  

\begin{LPMC}
Suppose $\k$ is a number field.  For all but finitely many primes $\p \subset \mathcal{O}_\k$, if $\alpha$ is a pole of $Z_\p(s)$, then $\exp(2 \pi i \Re(\alpha))$ is 
a nearby eigenvalue of monodromy.
\end{LPMC}

\noindent 
Interest in this conjecture has persisted through the decades \cite{Nicaise10, ViuSos22}.  Bories and Veys proved it for $n = 3$ when $f$ is nondegenerate 
\cite[Theorem 0.12]{BoriesVeys16}. There has been little progress in higher dimensions.  

\subsubsection{Good reduction}  The local $\p$-adic zeta function has a particularly simple expression when $X = X_f \subset \A^n$ has an embedded log resolution with good reduction mod $\p$.  Most results 
showing that poles of $Z_\p(s)$ correspond to nearby eigenvalues of monodromy
for $n \geq 3$, including those of \cite{BoriesVeys16} and our Theorem~\ref{thm:lpmsimplicial}, have a good reduction hypothesis.  

Suppose the log resolution $\pi \colon Y \to \A^n$ factors through a closed embedding $Y \hookrightarrow \P^m \times \A^n$ over $\k$.  Let $\P^m_R$ and $\A^n_R$ denote the projective and affine spaces of dimension $m$ and $n$, respectively, over $R$.  Let $X_R$ and $Y_R$ be the closures of $X$ and $Y$ in $\A^n_R$ and $\P^m_R \times \A^n_R$.  Then $\pi$ extends naturally to a projective morphism $\pi_R \colon  Y_R \to X_R$.  Let $\overline X$ and $\overline Y$ be the respective special fibers of $X_R$ and $Y_R$.  Base change to $\F_q = R/\p$ gives a projective morphism 
$
\overline \pi \colon \overline Y \to \overline X.
$ 
Let $\overline D_i$ be the special fiber of the closure of $D_i$ in $Y_R$.  
\begin{definition}
The resolution $\pi \colon Y \to \A^n$ has \emph{good reduction} mod $\p$ if 
\begin{itemize}
\item $\overline Y$ is smooth in a neighborhood of $\overline \pi^{-1}(0)$;
\item $\overline D_1, \ldots, \overline D_r$ are smooth and distinct over $\F_q$, and they meet each other transversely.
\end{itemize}
\end{definition}

\noindent 
Note that, if $f$ and $\pi$ are defined over a number field $K$, then $\pi$ has good reduction mod $\p$ for all but finitely many primes $\p$ in the ring of integers $\mathcal{O}_K$ \cite[Theorem~2.4]{Denef87}.

Any resolution with good reduction mod $\p$ gives rise to a pleasant formula for $Z_\p(s)$ in terms of the numerical data of the resolution and the number of $\F_q$-points in the strata of the fiber over $0$.  Let $$\overline E_I^\circ = \{ x \in \overline \pi^{-1}(0) : x \in \overline D_i \mbox { if and only if } i \in I \}.$$

\begin{theorem}[{\cite[Theorem~3.1]{Denef87}}]
Suppose $\pi$ has good reduction mod $\p$.  Then
\begin{equation} \label{eq:goodreduction}
Z_\p(s) = \sum_{I \subset \{1, \ldots, r\}} (q-1)^{|I|} |\overline E_I^\circ(\F_q)| \prod_{i \in I} \frac{q^{-N_i s - \nu_i}}{1-q^{-N_i s - \nu_i}}.
\end{equation}
\end{theorem}

\noindent Note that point counts over finite fields are analogous to topological Euler characteristics over $\C$; both are additive with respect to disjoint unions and multiplicative with respect to products.  The role of $|\overline E_i^\circ(\F_q)|$ in \eqref{eq:goodreduction} is analogous to that of $\chi(E_i^\circ)$ in \eqref{eq:ACampo}. When $|\overline E_i^\circ(\F_q)|$ vanishes, then a term in \eqref{eq:goodreduction} involving a pole at $\alpha_i$ vanishes, and when the Euler characteristic $\chi(E_i^\circ)$ vanishes, a term in \eqref{eq:ACampo} involving the multiplicity of the corresponding eigenvalue of monodromy vanishes.  For more explicit connections, see \cite{Denef91b}.  

\subsubsection{Local topological zeta functions}
The analogy between Euler characteristics and the point counts over finite fields that appear in formulas for the local zeta functions in cases of good reduction leads to the topological zeta functions of Denef and Loeser.  Heuristically, these are limits of local $\p$-adic zeta functions.  The local topological zeta function is defined as follows:
\[
Z_\t(s) := \sum_{I \subset \{1, \ldots, r\}} \chi(E_I^\circ) \prod_{i \in I} \frac{1}{N_i s + \nu_i}.
\]
It is independent of the choice of resolution \cite[Theorem~2.1.2]{DenefLoeser92}.

Suppose $f$ has coefficients in a number field $K$.  Then poles of $Z_{\t}(s)$ give rise to poles of most local $\p$-adic zeta functions.  More precisely, after clearing denominators, we may assume that $f$ has coefficients in the ring of integers.  In this case, if $\alpha$ is a pole of $Z_{\t}(s)$ then, for all but finitely many primes $\p$ in the ring of integers, there are infinitely many unramified extensions $\k \, | \, K_\p$ such that $\alpha$ is a pole of the local $\p$-adic zeta function of $f$ over $\k$. See \cite[Theorem~2.2]{DenefLoeser92}.  

\begin{LTMC}[{\cite[Conjecture~3.3.2]{DenefLoeser92}}]
If $\alpha$ is a pole of $Z_\t(s)$, then $\exp(2 \pi i \alpha)$ is 
a nearby eigenvalue of monodromy.
\end{LTMC}

\noindent We note that local topological zeta functions have a pleasantly simple expression for singularities that are nondegenerate \cite[Section 5]{DenefLoeser92}.  For the corresponding formula for the local $\p$-adic zeta function of a nondegenerate singularity with good reduction mod $\p$, see \cite[Theorem~4.2]{DenefHoornaert01}.

One naturally expects that the local topological monodromy conjecture should be easier to prove than the local $\p$-adic monodromy conjecture, even in the good reduction case, and experience does bear this out.  For instance, both conjectures are known in the special case when $f$ is nondegenerate and $n = 3$. However, the proof of the topological case \cite{LemahieuVanProeyen11} preceded the proof of the $\p$-adic case \cite{BoriesVeys16} by a few years and is considerably shorter. 
See \cite[Exercise~3.65]{ViuSos22} for an example of a nondegenerate hypersurface (for $n = 5$) with a real pole of its local $\p$-adic zeta functions that is not a pole of its local topological zeta function.

\subsubsection{The local motivic zeta function}

The local motivic zeta function of $f$ at $0$ is a formal power series with coefficients in a localization of the $\mh$-equivariant Grothendieck ring of varieties and can be defined in terms of an embedded log resolution, as follows.

\medskip

Let $\mu_m := \Spec \k[t]/(t^m-1)$ denote the group of $m$th roots of unity over $\k$, and let $$\mh := \varprojlim \mu_m.$$
An action of $\mh$ on a $\k$-variety $Y$ is \emph{good} if the action factors through $\mu_m$ for some $m$, and $Y$ is covered by invariant affine opens.  The \emph{Grothendieck ring} $K^\mh$ 
is additively generated by classes $[Y]$, where $Y$ is a $\k$-variety with good $\mh$-action, subject to the relations:
\begin{itemize}
\item if $Z$ is a closed $\mh$-invariant subvariety, then $[Y] = [Y \smallsetminus Z] + [Z]$;
\item if $W \to Y$ is a $\mh$-equivariant $\A^m$-bundle, then $[W] = [\A^m \times Y]$.
\end{itemize}
In the second relation, $\mh$ acts trivially on $\A^m$. Multiplication in the Grothendieck ring $K^\mh$ is given by $[Y] \cdot [Z] = [Y \times Z]$, with the diagonal $\mh$-action on $Y \times Z$. 

\medskip 

For each nonempty subset $I \subset \{1, \ldots, r\}$, let $m_I := \gcd \{ N_i : i \in I \}$.  Then $D_I^\circ$ is covered by Zariski open subsets $U \subset Y$ on which $\pi^*f$ is of the form $u g^{m_I}$, where $u$ is a unit on $U$, and $g$ is a regular function.  Consider the Galois cover $\widetilde D_I^\circ \to D_I^\circ$, with Galois group $\mu_{m_I}$ whose restriction to such an open set $D_I^\circ \cap U$ is
\[
\{ (z,y) \in \A^1 \times (D_I^\circ \cap U) : z^{m_I} = u^{-1} \}.
\]
Then $\widetilde D_I^\circ$ comes with the evident good $\mh$-action that factors through $\mu_{m_I}$ and commutes with the projection to $\A^n$. Let
\[
\widetilde E_I^\circ := \widetilde D_I^\circ \times_{\A^n} \{ 0 \} 
\]  
be the induced Galois cover of the fiber of $D_I^\circ$ over $0$, with the good $\mh$-action that it inherits from $\widetilde D_I^\circ$.

Let $\L := [\A^1]$, and set $\M = K^\mh[\L^{-1}]$.  The local motivic zeta function of $f$ at $0$ is the formal power series expansion in $\M \llbracket T \rrbracket$ of the following rational function in $\M(T)$: 
\[
Z_\mot(T) = \sum_{I \subset \{1, \ldots, r\}} (\L-1)^{|I|} [\tE_I^\circ] \prod_{i \in I} \frac{\L^{-\nu_i} T^{N_i}}{1-\L^{-\nu_i} T^{N_i}}.
\]

Note, in particular, that $Z_{\mot}(T)$ is contained in the subring of $\M \llbracket T \rrbracket$ generated over $\M$ by $T$ and $\big \{ \frac{1}{1-\L^{-\nu_i}T^{N_i}} : 1 \leq i \leq r  \big\}$. This subring depends on the choice of a log resolution, but the power series $Z_{\mot}(T)$ is independent of all choices. 

\begin{remark}\label{r:conventions}
In the literature, an additional multiplicative factor of $\L^{-n}$ sometimes appears in the definition of the local motivic zeta function. See, e.g. \cite[(0.1.2)]{RodriguesVeys03} and \cite[Theorem~3.18]{ViuSos22}. Other versions differ from ours by a factor of $\L-1$ \cite[Corollary~5.3.2]{BultotNicaise20}. These renormalizations are not relevant to the local motivic monodromy conjecture.
\end{remark}

Grothendieck rings of varieties are not integral domains \cite{Poonen02}, so care is required in defining poles of $Z_\mot(T)$.  Various notions are possible.  See, for instance, \cite[\S4]{RodriguesVeys03}.  We follow the now standard convention and state the local motivic monodromy conjecture in terms of sets of candidate poles, as in \cite{BoriesVeys16, BultotNicaise20}.  

\begin{definition}
Let $\cP$ be a finite 
set of rational numbers. Then $\cP$ is a \emph{set of candidate poles} for $Z_\mot(T)$ if $Z_\mot(T)$ is contained in $$\M\bigg[T, \frac{1}{1-\L^a T^b}\bigg]_{(a,b) \in \Z \times \Z_{>0}, a/b \in \cP}.$$
\end{definition}

\noindent Roughly speaking, if $\alpha$ satisfies any reasonable notion of being a pole of $Z_{\mot}(T)$, then it is contained in every set of candidate poles.

\begin{remark}
In practice, passing to an embedded log resolution $\pi$ is not a useful way of computing local zeta functions; this  typically introduces many exceptional divisors whose numerical data correspond neither to poles of the zeta function nor to eigenvalues of monodromy.  One obtains more efficient expressions for the local zeta functions of nondegenerate singularities by first proving that they can be computed from a log smooth partial resolution \cite{BultotNicaise20} or a stacky resolution \cite{Quek22}.
\end{remark}

\subsection{Prior results}

The local monodromy conjectures remain wide open in general, despite the persistent efforts of many mathematicians over a period of decades.  Perhaps most surprising is that the local topological monodromy conjecture remains open for isolated nondegenerate singularities, even though there are well-known and relatively simple combinatorial formulas for both the characteristic polynomial of monodromy
\cite[Theorem~4.1]{Varchenko76}
 and the local topological zeta function \cite[Theorem~5.3]{DenefLoeser92}. Nevertheless, there is a vast literature on the local monodromy conjectures, far more than can reasonably be reviewed  here. We give only a brief and largely ahistorical review of prior work closely related to our main theorems, and recommend the excellent survey articles \cite{Nicaise10, ViuSos22} for more detailed discussions and further references.

\subsubsection{Local monodromy conjectures}\label{ssec:history} For $n = 2$, Bultot and Nicaise proved the local motivic monodromy conjecture in full generality \cite[Theorem~8.2.1]{BultotNicaise20}, building on earlier work of Loeser \cite{Loeser88}.

For nondegenerate singularities when $n = 3$, Lemahieu and Van Proeyen proved the local topological monodromy conjecture \cite{LemahieuVanProeyen11}.  Bories and Veys used the same arguments for existence of eigenvalues and developed new arguments to reduce the size of sets of candidate poles, proving the  local $\p$-adic monodromy conjecture  \cite{BoriesVeys16}.  They also proved a \emph{naive} variant of the local motivic monodromy conjecture for nondegenerate singularities with $n = 3$. In the naive variant, the ring $K^\mh$ is replaced with the ordinary Grothendieck ring of varieties (without $\mh$-action); the local motivic zeta function specializes to the local naive motivic zeta function by  setting $[Y] \mapsto [Y/\mh]$. 

Esterov, Lemahieu, and Takeuchi introduced new arguments for both existence of eigenvalues and cancellation of poles for local topological zeta functions of nondegenerate singularities, especially for $n = 4$, and stated a conjecture for how these should generalize to higher dimensions \cite[Conjecture~1.3]{ELT}.  Recently, while this paper was in the final stages of preparation, Quek produced a naive motivic upgrade for some of the pole cancellation arguments from \cite{ELT}, giving a new proof of the main result of Bories and Veys for $n = 3$.  Quek also suggested a different way in which the pole cancellation statements for small $n$  might generalize to higher dimensions \cite[Question~5.1.8]{Quek22}.  Neither of these predictions is correct.  See Examples~\ref{ex:ELTconj} and \ref{ex:QuekConj}.  We also note that some of the claimed results in \cite{ELT} are incorrect already for $n = 4$. In particular, the classification of facets of Newton polyhedra in \cite[Lemma~5.18]{ELT} is incomplete (Example~\ref{ex:ELT518}) and there are counterexamples to their claimed results on existence of eigenvalues (Example~\ref{ex:ELTdim4}) and cancellation of poles (Example~\ref{ex:ELTprop}).

In higher dimensions, Budur and van der Veer recently proved the local monodromy conjectures for nondegenerate singularities whose Newton polyhedron is a large dilate of a convenient Newton polyhedron \cite[Theorem~1.10]{Budur}.  Indeed, they show that when $P = \Newt(f)$ is convenient and $k$ is sufficiently large, every candidate eigenvalue corresponding to a facet of $k P$ is an eigenvalue of monodromy. The proof is an application of Varchenko's formula \cite[Theorem~4.1]{Varchenko76}, and the bound on $k$ depends on $P$.  Here we show, by different arguments that depend on Ehrhart theory and positivity properties of local $h$-polynomials, that any $k \geq 2$ is large enough. 
We also prove a generalization of this result when $P$ is not necessarily convenient. See Theorem~\ref{thm:dilate} and 
Proposition~\ref{p:Budurextend}.

\subsubsection{Global zeta functions and strong monodromy conjectures} There are \emph{global} versions of the local motivic, $\p$-adic, and topological zeta functions and their associated monodromy conjectures. See, e.g., \cite{DenefLoeser92} for a discussion of the local and global topological zeta functions.  The difference between the local and global motivic zeta functions is illustrated by \cite[Theorems 8.3.2 and 8.3.5]{BultotNicaise20}. The global zeta functions are invariants of $X_f \subset \A^n$, while the local zeta functions are invariants of its germ at $0$.

Replacing ``local" by ``global" in each of the local monodromy conjectures gives rise to its global counterpart. There are also \emph{strong} versions of the local and global monodromy conjectures proposing that the real parts of the poles of the corresponding zeta functions are zeros of the Bernstein-Sato polynomial $b_f$.  If $\alpha$ is a zero of $b_f$ then $\exp(2\pi i \alpha)$ is an eigenvalue of monodromy, and all eigenvalues of monodromy occur in this way \cite{Malgrange74}. It is also conjectured that the orders of poles of local zeta functions are bounded by the multiplicities of zeros of $b_f$ \cite[Conjecture 3.3.1$'$]{DenefLoeser92}.

The strong local and global motivic monodromy conjectures are known for $n = 2$ \cite{BultotNicaise20}.  Loeser has given a combinatorial condition on Newton polyhedra that guarantees that each candidate pole associated to a facet is a zero of the Bernstein-Sato polynomial \cite{Loeser90}.  Nondegenerate polynomials with such Newton polyhedra therefore satisfy the strong local 
motivic monodromy conjecture. 

Aside from this, we note that if $X_f$ is smooth aside from an isolated singularity at $0$, then each local monodromy conjecture at 0 implies the corresponding global monodromy conjecture.  The Newton polyhedra whose nondegenerate singularities are isolated were classified by Kouchnirenko \cite{Kouchnirenko76}.  Furthermore, if $f$ has such a Newton polyhedron and $f|_F$ has no singularities outside the coordinate hyperplanes for \emph{all} faces $F$ of $\Newt(f)$, not just the compact faces, then $X_f$ is smooth away from $0$.  Thus the global motivic monodromy conjecture for isolated singularities with simplicial Newton polyhedra that satisfy this stronger nondegeneracy condition follows from Theorem~\ref{thm:lmmsimplicial}. 

\subsubsection{Further variants of the local zeta functions and monodromy conjectures}
There are also monodromy and holomorphy conjectures for $\p$-adic zeta functions \emph{twisted} by a character, and topological analogues of twisted $\p$-adic zeta functions.  For discussions of these variants, see, e.g., \cite{Denef91}. Another variant is the topological zeta function for a variety equipped with a holomorphic form that plays the role of $\Phi(x) dx$ in Malgrange's archimedean zeta functions \cite{Veys07}. 
Our results on eigenvalues of monodromy in Sections~\ref{sec:nearby}-\ref{sec:vanishing} are applicable to all such variants.

\subsection{Methods and structure of the paper} \label{sec:methods}

We conclude the introduction with a brief overview of our approach to the local motivic monodromy conjecture and outline the content of each section of the paper.

\subsubsection{Key definitions}
We first recall the notion of candidate poles and candidate eigenvalues.
In the literature, a candidate pole and candidate eigenvalue is associated to each facet of $\Newt(f)$. For our purposes, it is important to extend these notions to a
wider class of faces of $\Newt(f)$. To be precise, let $G$ be a proper face of $\Newt(f)$ that contains the vector $\1 = (1,...,1)$ in its linear span, denoted 
$\Span(G)$. Let $\psi_G$ be the unique linear function on $\Span(G)$ with value $1$ on $G$. 
Then $$\alpha_G := - \psi_G(\1)$$ is the \emph{candidate pole} associated to $G$, and $\exp(2 \pi i \alpha_G)$ is the corresponding \emph{candidate eigenvalue} of monodromy. 
We say that $G$ \emph{contributes} $\alpha_G$ as a candidate pole.
If $G'$ contains $G$ as a face, then $G'$ also contains $\1$ in its linear span and 
$\alpha_{G'} = \alpha_G$.

\begin{definition}
Let $\Contrib(\alpha)$ be the set of faces of $\Newt(f)$ that contribute the candidate pole $\alpha$.
\end{definition}

\noindent Then $\{ \alpha \in \mathbb{Q} : \Contrib(\alpha) \neq \emptyset \} \cup \{ -1 \}$ is a set of candidate poles for $Z_\mot(T)$ \cite[Corollary 8.3.4]{BultotNicaise20}.  This set of candidate poles is standard in the literature.  The key difference here is that we consider faces in $\Contrib(\alpha)$ of arbitrary codimension, not just facets. This change in perspective is crucial in what follows.

Let $C$ be a cone in $\Delta$, the fan over the faces of $\Newt(f)$. 
The rays of $C$ are the union of rays through vertices in $\Newt(f)$ and rays disjoint from $\Newt(f)$ that contain a coordinate vector $e_\ell$ for some $1 \le \ell \le n$. 
In particular,
for each ray of $C$, there is a corresponding distinguished generator: either the corresponding vertex of $\Newt(f)$, or the corresponding coordinate vector $e_\ell$. We let $\gens(C)$ be the set of distinguished generators of the rays of $C$.

We say that a vertex $A$ in $G$ is an \emph{apex} with \emph{base direction} $e_\ell^*$ if $\langle e_{\ell}^*, A \rangle > 0$, and
$\langle e_{\ell}^*, V \rangle = 0$ for all  $V \in \gens(C_G)$  with $V \neq A$.
In this case, $G \cap \{ V \in \R^n_{\ge 0} : \langle e_\ell^*, V \rangle = 0 \}$ is the corresponding \emph{base} of $G$. 

\begin{definition}\label{def:B1}
A face $G$ of $\Newt(f)$ is \emph{$B_1$} 
if it has an apex $A$ with base direction 
$e_\ell^*$, and $\langle e_{\ell}^*, A \rangle = 1$. 

\end{definition}
The notion of $B_1$ was introduced for simplicial facets in \cite[Definition 3]{LemahieuVanProeyen11}. For arbitrary facets, our definition agrees with \cite[Definition 1.1.7]{Quek22} but is more restrictive than \cite[Definition 3.1]{ELT}. All of these definitions of $B_1$-facets agree when $\Newt(f)$ is simplicial.  Note that the base direction $e_{\ell}^*$ determines the apex $A$. The converse is not true. A $B_1$-face may have several apices, and when the face is not a facet, each of those apices can have multiple base directions.  We introduce the following definition.

\begin{definition}\label{def:UB1}
A face $G$ of $\Newt(f)$ is \emph{$\UB$} 
if it has an apex $A$ with a unique base direction 
$e_\ell^*$, and $\langle e_{\ell}^*, A \rangle = 1$. 

\end{definition}
Theorems~\ref{t:mainsimplicialeigenvalue} and \ref{thm:nopolesimplicial} show the importance of the notion of $\UB$-faces.   Note that every $B_1$-facet is $\UB$; the distinction between $B_1$ and $\UB$ is only relevant when considering higher codimension faces.

\subsubsection{Eigenvalue multiplicities and local $h$-polynomials} The starting point for our work is the third author's nonnegative formula for the multiplicities of eigenvalues of monodromy at $0$ when $f$ is nondegenerate and $\Newt(f)$ is convenient
\cite[Section~6.3]{Stapledon17}. Specializing \cite[Theorem~6.20]{Stapledon17} from equivariant mixed Hodge numbers to equivariant 
multiplicities,
one obtains a combinatorial formula with nonnegative integer coefficients for the multiplicities of the eigenvalues of monodromy on the reduced cohomology of  $\cF_0$. 

Assume that $\Newt(f)$ is simplicial.
If we forget the lattice structure of the fan $\Delta$, we may view $\Delta$ as encoding a triangulation of a simplex, e.g., by slicing with a transverse hyperplane. 
Then the combinatorial formula for eigenvalues is a sum over 
cones $C$ in $\Delta$
of a contribution that is a product 
of two nonnegative factors, one coming from Ehrhart theory (the number of lattice points in a polyhedral set). The other factor is the evaluation of the local $h$-polynomial 
$\ell(\Delta,C;t)$ 
at $t = 1$.
These local $h$-polynomials were first introduced and studied by Stanley in the special case where $C = 0$ and later generalized by Athanasiadis, Nill, and Schepers  \cite{Athanasiadis12b, NillSchepers12}.
They have nonnegative, symmetric integer coefficients and 
naturally appear when applying the decomposition theorem to toric morphisms. See \cite[Theorem~5.2]{Stanley92}, \cite[Theorem~6.1]{KatzStapledon16} and \cite{deCataldoMiglioriniMustata18}. 

This formula for eigenvalue multiplicities in the convenient nondegenerate case offers fundamental advantages over earlier approaches to existence of eigenvalues.  Whereas the formulas of A'Campo and Varchenko for zeta functions of monodromy typically involve a great deal of cancellation, the third author's formula is a sum of nonnegative terms.  
Moreover, for each compact face $G$ in $\Contrib(\alpha)$, 
there is a canonically associated \emph{essential face $E \subset G$}. See Definition~\ref{d:essentialface}. 
Then the Ehrhart  factor in the 
summand associated to $C_E$
for the multiplicity of $\exp(2 \pi i \alpha)$ is strictly positive. 
Thus, either $\exp(2 \pi i \alpha)$ is an eigenvalue of monodromy or $\ell(\Delta, C_E;t)$ is zero.  There are a number of simple sufficient conditions for the nonvanishing of $\ell(\Delta, C_E; t)$; for instance, if $E$ meets the interior of the positive orthant, then $\ell(\Delta, C_E;0) = 1$. The general problem of classifying when local $h$-polynomials vanish was posed by Stanley \cite[Problem~4.13]{Stanley92}.  See \cite{dMGPSS20} for a classification  when $n \leq 4$ and $E = \emptyset$ and for partial results in higher dimensions.

\subsubsection{A nonnegative formula for nearby eigenvalues} In Section~\ref{sec:nearby}, we extend the results of \cite{Stapledon17} to the case where $\Newt(f)$ is simplicial but not necessarily convenient. In this setting, the singularity of $X_f$ at $0$ may not be isolated, and the Milnor fibers at $0$ and at nearby points may have cohomology in multiple positive degrees.

In this setting, we consider $\widetilde{\chi}(\cF_x) := \sum_i (-1)^i \widetilde{H}^i(\cF_x,\C)$ as a virtual representation, where $\widetilde{H}$ denotes reduced cohomology. Now $\exp(2 \pi i \alpha)$ has a multiplicity $\widetilde{m}_x(\alpha)$,  which may be positive or negative, as an eigenvalue in this virtual representation.  We consider these multiplicities not only at $0$ but also at a general point $x_I$ in each coordinate subspace $\A^I$ contained in $X_f$. The idea of studying the eigenvalues at these points was first introduced when $n = 3$ in \cite{LemahieuVanProeyen11} and  further developed in \cite{ELT}. We give a nonnegative formula for the alternating sum
\[
\sum_{\A^I \subset X_f} (-1)^{n - 1 - |I|} \widetilde{m}_{x_I}(\alpha). 
\]
See Theorem~\ref{t:nonnegativeVarchenko} for a precise statement.  

Theorem~\ref{t:nonnegativeVarchenko} implies, in particular, the remarkable fact that the corresponding alternating product of monodromy zeta functions is a polynomial, i.e., 
\begin{equation}
\prod_{\A^I \subset X_f} \left( \frac{\zeta_{x_I}(t)}{1 - t}\right)^{(-1)^{n-1-|I|}} \in \Z[t].
\end{equation}
From this perspective, the theorem provides a nonnegative formula for the vanishing order of this polynomial at $\exp(2\pi i \alpha)$.  See Remark~\ref{r:monodromyzetaformula} for the precise formula.

Just as in the convenient case, 
this nonnegative formula is 
a sum over 
cones $C$ in $\Delta$, and each of the terms is once again an Ehrhart factor times $\ell(\Delta,C;1)$.
Moreover, for each compact $G \in \Contrib(\alpha)$,
we have an essential face $E \subset G$, and 
the Ehrhart  factor in the  
$C_E$-summand
for the multiplicity of $\exp(2 \pi i \alpha)$ is strictly positive. We deduce the following corollary.
 See Corollary~\ref{c:nonvanishingeigenvalue} for an equivalent statement.

\begin{corollary}\label{c:intrononvanishingeigenvalue}
Suppose $\Newt(f)$ is simplicial and $f$ is nondegenerate.
Let $G$ be a compact face in  $\Contrib(\alpha)$ with essential face $E$. 
If 
$\ell(\Delta,C_E;t)$ is nonzero, then 
$\sum_{\A^I \subset X_f} (-1)^{n - 1 - |I|} \widetilde{m}_{x_I}(\alpha) > 0$. In particular, 
 $\exp(2 \pi i \alpha)$ is 
a nearby eigenvalue of monodromy (for reduced cohomology).

\end{corollary}

This motivates a detailed study of necessary conditions for the vanishing of $\ell(\Delta, C_E;t)$ when $E$ is the essential face 
associated to some compact
face $G \in \Contrib(\alpha)$. 
In this situation, we also have some additional structure which is crucial for our arguments.
The face $C_G \smallsetminus C_E \in \lk_{\Delta}(C_E)$ admits what we call a \emph{full partition}. See Lemma~\ref{l:existencefullpartition} and Definition~\ref{d:fullpartition}.

\subsubsection{A necessary condition for the vanishing of the local $h$-polynomial}

Motivated by the results of Section~\ref{sec:nearby}, in Section~\ref{sec:vanishing} we undertake a detailed investigation of the conditions under which $\ell(\Delta, C';t)$ vanishes, where $C'$ is a cone in $\Delta$ that is contained in a cone that admits a full partition. This section is self-contained and applies to any local $h$-polynomial of a geometric triangulation. See \cite{LPS2} for further work on necessary conditions for the vanishing of the local $h$-polynomial in a more general setting, for quasi-geometric homology triangulations. 

Recall  
that $\ell(\Delta, C';t)$ is naturally identified with the Hilbert function of a module $L(\Delta, C')$  \cite{Athanasiadis12, Athanasiadis12b}, as follows.  Consider the ideal in the face ring $\Q[\lk_{\Delta}(C')]$ generated by monomials $x^{C}$ such that $C \sqcup C'$ meets the interior of the orthant $\R_{>0}^n$.
Then $L(\Delta, C')$ is the image of this ideal in the quotient of $\Q[\lk_{\Delta}(C')]$ by a special linear system of parameters.  Thus $\ell(\Delta, C';t) = 0$ if and only if every such monomial is contained in the ideal generated by a special linear system of parameters.  
When $C$ admits a full partition, we can associate a 
distinguished monomial with image in $L(\Delta, C')$.
By reducing to a result in \cite{LPS2}, we show that this monomial is nonzero in $L(\Delta, C')$. 

Using this calculation, we prove the following theorem, which is  an immediate consequence of Theorem~\ref{thm:nonvanish}.
A cone $C$ in $\lk_{\Delta}(C')$ is a \emph{$U$-pyramid} if it meets the interior of the positive orthant and there is a ray $r \in C$ such that $(C \sqcup C') \smallsetminus r$ is contained in a unique 
coordinate hyperplane in $\R^n$,
i.e., $C$ is a pyramid with a \emph{unique} base 
direction with respect to the apex $r$ in $\lk_\Delta(C')$. See Definition~\ref{d:Upyramid}.

\begin{theorem}\label{thm:intrononvanish}
Let $\Delta$ be a simplicial fan supported on $\R^n_{\ge 0}$. Let $C'$ be a cone in $\Delta$, and let $C \in \lk_\Delta(C')$. If $\ell(\Delta, C';t) = 0$ and $C$ admits a full partition, then $C$ is a 
$U$-pyramid. 
\end{theorem}

When $G$ is compact and 
$C = C_G \smallsetminus C_E \in \lk_{\Delta}(C_E)$, the
condition that $G$ is $\UB$ is equivalent to the condition that $C$ is a $U$-pyramid. 
See Lemma~\ref{l:weaklyUB1}. 
This leads to the following theorem, which is our main result on existence of eigenvalues.

\begin{theorem}\label{t:mainsimplicialeigenvalue}
Suppose $\Newt(f)$ is simplicial and $f$ is nondegenerate.
Let $\alpha \in \mathbb{Q}$. Then either every face in 
$\Contrib(\alpha)$ is $\UB$, or $\exp(2 \pi i \alpha)$ 
is an eigenvalue of monodromy for the reduced cohomology of the Milnor fiber at the generic point of some coordinate subspace $\A^I \subset X_f$.
\end{theorem}

\noindent Section~\ref{sec:nearby} proves that this theorem follows from Theorem~\ref{thm:intrononvanish}, whose proof is given in Section~\ref{sec:vanishing}.

\subsubsection{The local formal zeta function and its candidate poles}

In Sections~\ref{sec:poles} and \ref{sec:fakepoles}, we prove the following
theorem, which is complementary to Theorem~\ref{t:mainsimplicialeigenvalue}. 

\begin{theorem}\label{thm:nopolesimplicial}
Suppose $\Newt(f)$ is simplicial and $f$ is nondegenerate.
Let $$\mathcal{P} = \{ \alpha \in \mathbb{Q} : \Contrib(\alpha) \neq \emptyset \} \cup \{ -1 \}, \text{ and }\mathcal{P}' = \{ \alpha \in \mathcal{P} : \alpha \notin \Z, \textrm{every face in } \Contrib(\alpha) \textrm { is } \UB \}.$$
Then 
$\mathcal{P} \smallsetminus \mathcal{P}'$
is  a set of candidate poles for $Z_{\mot}(T)$. 

\end{theorem}

Note that Theorem~\ref{thm:lmmsimplicial} follows directly from Theorems~\ref{t:mainsimplicialeigenvalue} and \ref{thm:nopolesimplicial}, using the fact that
$1$ is an eigenvalue of monodromy on $H^0(\mathcal{F}_0, \mathbb{C})$.

\bigskip

Our starting point for the proof of Theorem~\ref{thm:nopolesimplicial} is the formula for $Z_{\mot}(T)$ in \cite[Theorem 8.3.5]{BultotNicaise20}, which expresses $Z_{\mot}(T)$ as a sum over lattice points in the dual fan 
to $\Newt(f)$. We introduce the \emph{local formal zeta function}
$Z_{\for}(T)$, which is a power series over a polynomial ring that specializes to $Z_{\mot}(T)$. The local formal zeta function depends only on $\Newt(f)$, unlike $Z_{\mot}(T)$ which depends on $f$.
The advantage of working with $Z_{\for}(T)$ is that an intersection of two sets of candidate poles of $Z_{\for}(T)$ is a set of candidate poles (Lemma~\ref{l:intersectcandidate}), so it suffices to show that, for each $\alpha \not \in \Z$ such that $\Contrib(\alpha)$ consists entirely of $\UB$-faces, there is a set of candidate poles for $Z_{\for}(T)$ not containing $\alpha$.
Explicitly, 
\begin{equation}\label{e:formalexpression}
Z_{\for}(T) =  \sum_{G}
Y_G \bigg( (L - 1)^{n  - \dim G} \sum_{u \in \sigma^{\circ}_{G} \cap \mathbb{N}^{n}} L^{-\langle u, \mathbf{1} \rangle}T^{N(u)} \bigg),
\end{equation}
where $G$ varies over all nonempty compact faces of $\Newt(f)$, $\sigma_G$ denotes the dual cone to $G$,
$C^\circ$ denotes the relative interior of a polyhedral cone $C$, $N$ is a certain piecewise linear function, and $Y_G$ and $L$ are formal variables satisfying the following relations:
\begin{enumerate}
\item $Y_V = 1$ if $V$ is a primitive vertex of $\Newt(f)$, and
\item\label{i:B1equation} $Y_G + Y_F = \frac{(L-1)^{\dim F}}{1 - L^{-1}T}$ if $F$ is a compact $B_1$-face  with nonempty base $G$.
\end{enumerate}
See Definition~\ref{d:formallocal} for details. 
The key relation above is \eqref{i:B1equation}, which specializes to a natural relation in the $\mh$-equivariant Grothendieck ring of varieties.
See Lemma~\ref{l:simplifiedrelations}.

Given a subset $C \subset \R^n_{\ge 0}$, we can define 
the \emph{contribution} $Z_{\for}(T)|_{C}$ of $C$ to $Z_{\for}(T)$ to be the same expression as the right-hand side of \eqref{e:formalexpression}, except that the second summation runs over $u \in \sigma_G^\circ \cap C \cap \N^n$. See \eqref{e:contribution}. 
When $F$ is a compact $B_1$-face with nonempty base $G$ and apex $A$ in the direction $e^*_{\ell}$, 
and  $C'  \subset \sigma_F^{\circ}$ is a nonzero rational polyhedral cone, then
we deduce the following relation: 
\begin{equation}\label{e:technicaltool}
Z_{\for}(T)|_{C^{\circ}} + Z_{\for}(T)|_{(C')^{\circ}} = (L - 1)^{n} \bigg(\sum_{u \in (C^{\circ} \cup (C')^{\circ}) \cap \mathbb{N}^n} L^{-\langle u, \mathbf{1} \rangle} T^{\langle u, A \rangle} \bigg),
\end{equation}
where $C \subset \sigma_G$ is the cone spanned by $C'$ and
$e_\ell^*$. See Lemma~\ref{lem:cancellation}. The above equation \eqref{e:technicaltool} is a key technical tool underlying our strategy to show fakeness of poles, and it is analogous to a formula involving the local topological zeta function in \cite[Lemma 3.3]{ELT}.
For example, if $\Contrib(\alpha)$ consists of a single $B_1$ (and hence $\UB$) facet $F$ with apex $A$, then
 one can deduce an expression for $Z_{\for}(T)$ with no candidate pole at $\alpha$ by
 applying \eqref{e:technicaltool} with $C' = \sigma_G^\circ$, as $G$ varies over all faces of $G$ 
not containing $A$. This is analogous to approaches to showing fakeness of poles under certain assumptions
for the local topological zeta function in \cite{ELT,LemahieuVanProeyen11} and the local naive motivic zeta function in \cite[Theorem A]{Quek22}.

When we only assume that every face in $\Contrib(\alpha)$ is $B_1$, it may not be possible to extend the above approach. The key difficulty is that it may not be possible to choose a single base direction $e_\ell^*$ for all faces of $\Contrib(\alpha)$. 
In our case, we assume that all elements of $\Contrib(\alpha)$ are $\UB$. 
This assumption implies that 
we may choose base directions for elements of $\Contrib(\alpha)$
satisfying a natural 
compatibility condition:
to every face $G$ in $\Contrib(\alpha)$, we may assign a pair $(A_G, e_G^*)$ such that  $G$ is $B_1$ with apex $A_G$ and base direction $e_G^*$, and, if $G \subset G'$ and $A_{G} = A_{G'}$, then $e_G^* = e_{G'}^*$.
See Definition~\ref{def:operative} and Lemma~\ref{prop:operativeiff}. 

We now sketch the remainder of the proof, and refer the reader to Section~\ref{ssec:overview} for a more detailed overview.
We first fix a minimal element $M$ of $\Contrib(\alpha)$ and reduce to considering only elements of $\Contrib(\alpha)$ that contain $M$. See Section~\ref{ssec:NMdelta}. 
We then use the above compatibility condition to construct a fan with support $\R^n$ satisfying certain properties. 
See Section~\ref{ss:exist1} and Section~\ref{ss:exist2}.
In particular, we assign to every cone $\tau$ in the fan a coordinate vector $e_\tau^*$ such that, if $M \subset G$ and $\sigma_G \cap \tau \neq (0)$, then $G$ is a $B_1$-face with base direction $e_\tau^*$. In this sense, we locally choose a single base direction.
Then
$Z_{\for}(T)$ is the sum of all contributions $Z_{\for}(T)|_{\tau^\circ}$ as $\tau$ varies over all cones of the fan. Analogously to the case when $\Contrib(\alpha)$ is a single $B_1$-facet, we then intersect each cone $\tau$ with the dual fan to $\Newt(f)$ and repeatedly apply \eqref{e:technicaltool} to obtain an expression for $Z_{\for}(T)|_{\tau^\circ}$  with no candidate pole $\alpha$, allowing us to complete the proof of 
Theorem~\ref{thm:nopolesimplicial}. See Section~\ref{ss:s0compatible}.

\subsubsection{Beyond the simplicial case}
In this introduction, we have stated our main results under the assumption that $\Newt(f)$ is simplicial. However, both our arguments about eigenvalues and about poles are carried out in somewhat greater generality. In Section~\ref{sec:beyondsimplicial}, we state our most general result on the local motivic monodromy conjecture for nondegenerate singularities (Theorem~\ref{thm:optimalclass}), which is sufficient to prove the local motivic monodromy conjecture in all cases when $f$ is nondegenerate and $n=3$ (Theorem~\ref{thm:n=3}).

Finally, observe that $\exp(2 \pi i \alpha)$ appearing as 
a zero or pole of the monodromy zeta function implies that $\exp(2 \pi i \alpha)$ is an eigenvalue of monodromy, but the converse is not true.
When $\Newt(f)$ is simplicial and $X_f$ is nondegenerate, we prove that there is a set of candidate poles $\cP$ such that, for all $\alpha \in \mathcal{P} \smallsetminus \mathbb{Z}$, $\exp(2 \pi i \alpha)$ is a zero or pole of the monodromy zeta function at the generic point of some coordinate subspace $\A^I \subset X_f$.
This stronger statement about when certain monodromy zeta functions are sufficient to detect eigenvalues is not true when $\Newt(f)$ is not simplicial and $n \geq 4$. In such cases, there may be poles of local topological zeta functions such that 
$\exp(2 \pi i \alpha)$ appears as 
a zero or pole of the monodromy zeta function only at points 
along strata that are properly contained in coordinate subspaces.  See \cite[Example~7.5]{ELT}. For one combinatorial approach to detecting such eigenvalues,  
see \cite{Esterov21}.

\subsection{Notation}
We now set up some additional notation which we will use for the remainder. With the exception of Section~\ref{sec:beyondsimplicial}, the notation of the various sections is otherwise largely independent.

Let $F \subset \R^n_{\ge 0}$ be a rational polyhedron whose affine span 
does not contain the origin, and let $\Span(F)$ denote the linear span of $F$. 
Let $\psi_F$ be the unique $\Q$-linear function on $\Span(F)$ with value $1$ on $F$. The \emph{lattice distance} $\rho_F$ of $F$ from the origin is the smallest positive integer $\rho_F$ such that $\rho_F \psi_F$ is a $\Z$-linear function. 
If $F \subset G$ is an inclusion of such rational polyhedra, then $\rho_F$ divides $\rho_G$.

A face of $\Newt(f)$ is \emph{interior} if it meets $\mathbb{R}^n_{>0}$. The functions $\psi_F$, for $F$ a face of a  proper interior face of $\Newt(f)$,
assemble into a 
function $\psi$ on $\mathbb{R}^n_{\ge 0}$ that is piecewise linear with respect to $\Delta$.

For a polyhedral cone $C$, let $\partial C$ denote its boundary, defined to be the union of all faces of $C$ of dimension strictly less than $\dim C$.
Let $C^{\circ} = C \smallsetminus \partial C$ denote the relative interior of a polyhedral cone. A nonzero vector $v$ in $\mathbb{Z}^n$ is \emph{primitive} if it generates the group $\mathbb{R} v \cap \mathbb{Z}^n$. 
Recall that $C$ is \emph{simplicial} if it is a pointed cone generated by $\dim C$ rays. For a set of vectors $S$ in   $\mathbb{R}^n$, let $\langle S \rangle$ denote the cone that they span.

A \emph{geometric triangulation} of a simplex is a subdivision of a geometric simplex into a union of geometric simplices that meet along shared faces. 

For a positive integer $\ell$, we write $[\ell] = \{ 1,\ldots, \ell \}$. 

\medskip

\noindent \textbf{Acknowledgments.} We thank M. Musta\c{t}\u{a} and J. Nicaise for helpful conversations related to the monodromy conjectures, K. Karu for explaining aspects of the commutative algebra of local $h$-polynomials, and M. H. Quek for insightful comments on an earlier draft of this paper and telling us about Example~\ref{ex:ELTdim4}. We thank the referee for helpful comments. The work of ML is supported by an NDSEG fellowship, and the work of SP is supported in part by NSF grants DMS-2001502 and DMS-2053261.

\section{Examples}

\subsection{Basic examples}
Our first two examples are intended to serve as a guide to the main constructions in the paper. 
In these examples, $\Newt(f)$ is simplicial and $f$ is supported at the vertices of $\Newt(f)$, so $f$ is nondegenerate \cite{BrzostowskiOleksik16}.

Below, $\widetilde{E}(\cF_x) \in \Z[\Q/\Z]$ is an alternative encoding of $\frac{\zeta_{x}(t)}{1 - t}$ (see \eqref{eq:Etilde} in Section~\ref{sec:nearby}). The local formal zeta function $Z_{\for}(T)$ lies in a quotient ring of 
$\Z[L,L^{-1}][Y_K : K \textrm{ nonempty compact face of } \Newt(f)]\llbracket T \rrbracket$, 
where $L,T,Y_K$ are formal variables. See
Definition~\ref{d:formallocal}.

\begin{example}\label{e:cusp}
Let $f(x_1,x_2) = x_2^2 - x_1^3$. Then $X_f$ has an isolated cusp at $0$, and   
$\Newt(f)$ is convenient and has a unique compact facet $F$ with vertices $v =  (3,0)$ and $w = (0,2)$.
Note that $\alpha_F = -5/6$, and $F$ is not $B_1$.
Theorem~\ref{t:mainsimplicialeigenvalue} says that $\exp(2 \pi i \alpha_F)$ is an eigenvalue of monodromy for $f$ at $0$. 

Then $\Delta$ is the trivial fan, and  $\ell(\Delta,C;t)$ equals $1$ if $C = C_F$, and equals $0$ otherwise.
We have monodromy zeta function $\zeta_{0}(t) = \frac{1 - t}{1 - t + t^2}$, and 
$\widetilde{E}(\cF_0) = [1/6] + [5/6]$. 
The local formal zeta function is  
$$Z_{\for}(T)=  \frac{(L - 1) \left( Y_F  L^{-5}T^6 + Y_{v} L^{-2}T^{3}(1 + L^{-3}T^{3}) + Y_w L^{-1}T^{2}(1 + L^{-2}T^{2} + L^{-4}T^{-4}) \right)
}{1 - L^{-5}T^6}.$$
The local motivic zeta function $Z_{\mot}(T)$ is
$$\frac{(\L - 1) \big( \big(\frac{(\L - 1)\L^{-1}T}{1 - \L^{-1}T} + [Y_F(1)]\big)  \L^{-5}T^6 + [\mu_3] \L^{-2}T^{3}(1 + \L^{-3}T^{3}) + [\mu_2] \L^{-1}T^{2}(1 + \L^{-2}T^{2} + \L^{-4}T^{-4}) \big)
}{1 - \L^{-5}T^6},$$
where
$Y_F(1)$ is an elliptic curve minus $6$ points, 
with a free $\mu_6$-action, and 
$Y_F(1)/\mu_6$ is isomorphic to $\P^1$ minus $3$ points. 
After simplification, the local naive  motivic zeta function is
\[
\frac{ (\L - 1)(
\L^{-1}T^{2} - \L^{-4}T^{5} +  \L^{-4} T^{6}  - \L^{-6}T^7 ) }{(1 - \L^{-1}T)(1 - \L^{-5}T^6)}.
\]
For $p \notin \{2,3\}$, $f$ has good reduction mod $p$, and the local $p$-adic zeta function is 
$$Z_{(p)}(s) =   \frac{(p - 1)(
p^{5s + 5} - p^{2s + 2} + p^{s + 2}  -1 )}{(p^{s + 1} - 1)(p^{6s + 5} - 1)}.$$
The local topological zeta function is   $Z_{\t}(s) = \frac{(4s + 5)}{(s + 1)(6s + 5)}$.

\end{example}

\begin{example}\label{e:Whitney}
Let $f(x_1, x_2, x_3) = x_1^2 - x_2^2x_3$.
Then
$X_f$ is the Whitney umbrella. There are three coordinate subspaces contained in $X_f$:
$\A^{\emptyset} = \{0\}$, $\A^{\{ 2 \}} =  \{ x_1 = x_3 = 0 \}$  and 
$\A^{\{ 3 \}} =  \{ x_1 = x_2 = 0 \}$.
The singular locus of $X_f$ is $\A^{\{ 3 \}}$. In particular, $f$ does not have an isolated singularity at the origin and $\Newt(f)$ is not convenient. 
The only maximal compact face of $\Newt(f)$ is a $1$-dimensional face $F$ with vertices $v =  (2,0,0)$ and $w = (0,2,1)$. Note that $F$ is $\UB$ and $(1,1,1) \notin \Span(F)$.
There are two unbounded 
interior facets: 
$F_1 = \{ e_1^* + e_2^* = 2 \} = F + \R_{\ge 0} e_3$ and  $F_2 = \{ e_1^* + 2e_3^* = 2 \} = F + \R_{\ge 0} e_2$, with $F_1 \cap F_2 = F$.

Then $\alpha_{F_1} = -1$, and $F_1$ is not $B_1$.
Theorem~\ref{t:mainsimplicialeigenvalue} predicts that $\exp(2 \pi i \alpha_{F_1}) = 1$ is a nearby eigenvalue of monodromy for reduced cohomology. 
Also, $\alpha_{F_2} = -3/2$, and  $F_2$ is $\UB$.
Then Theorem~\ref{thm:nopolesimplicial} predicts that there is a set of candidate poles for $Z_{\mot}(T)$ not containing $-3/2$. In this particular case,  the corresponding candidate eigenvalue $\exp(2 \pi i \alpha_{F_2}) = -1$ is a nearby eigenvalue  of monodromy.

The fan $\Delta$ has two maximal cones $C_{F_1}$ and $C_{F_2}$ intersecting in a unique interior 
$2$-dimensional face $C_F$.  
We have 
\[
\ell(\Delta,C;t) = \begin{cases}
1 & \textrm{ if } C = C_{F_1} \textrm{ or } C = C_{F_2}, \\
1 + t & \textrm{ if } C = C_{F}, \\
0 & \textrm{ otherwise. } \\
\end{cases}
\]
The monodromy zeta function at a general point $x_I$ of each coordinate subspace $\A^I$ is given by:
$\zeta_{0}(t) = (1 - t)(1 + t)$, $\zeta_{x_{\{ 2 \}}} = 1 - t$,  $\zeta_{x_{\{ 3 \}}} = 1$. 
Then $\prod_{\A^I \subset X_f} \left( \frac{\zeta_{x_I}(t)}{1 - t}\right)^{(-1)^{n-1-|I|}} = 1 - t^2.$  
Equivalently,
$\widetilde{E}(\cF_0) = [1/2]$, $\widetilde{E}(\cF_{x_{\{ 2 \}}}) = 0$, 
$\widetilde{E}(\cF_{x_{\{ 3 \}}}) = -1$, and 
$\sum_{\A^I \subset X_f} (-1)^{n - 1 - |I|} \widetilde{E}(\cF_{x_I}) = 1 + [1/2]$.
The local formal zeta function is  
$$Z_{\for}(T)=  \frac{L^{-3} T^2 (L - 1)^3  ( Y_{v} (1 - L^{-1}T) + L^{-1}T(1 - L^{-1}))}{(1 - L^{-1})^2(1 - L^{-1}T)(1 - L^{-2} T^{2})}.$$
The local motivic zeta function is
$$Z_{\mot}(T)=  \frac{\L^{-3} T^2 (\L - 1)^3  ( [\mu_2] (1 - \L^{-1}T) + \L^{-1}T(1 - \L^{-1}))}{(1 - \L^{-1})^2(1 - \L^{-1}T)(1 - \L^{-2} T^{2})}.$$
After simplifying, the local  naive motivic zeta function is 
\[
\frac{\L^{-1} T^2 (\L - 1) (  1 - \L^{-2}T)}{(1 - \L^{-1}T)(1 - \L^{-2} T^{2})}.
\]
For $p \neq 2$, $f$ has good reduction mod $p$, and the local $p$-adic zeta function is 
$$Z_{(p)}(s) =   \frac{(p - 1)(
p^{s + 2} - 1 )}{(p^{s + 1} - 1)^2(p^{s + 1} + 1)}.$$
The local topological zeta function is   $Z_{\t}(s) = \frac{(s + 2)}{2(s + 1)^2}$.

\end{example}

\subsection{Counterexamples}

We present counterexamples to Conjecture~1.3,  Proposition~3.7, and Lemma~5.18 of \cite{ELT} and show that the answer to \cite[Question 5.1.8]{Quek22} is negative. Example~\ref{ex:ELTdim4} shows that \cite[Conjecture~1.3]{ELT} is wrong already for $n = 4$.

The polyhedral computations in these examples were done using polymake \cite{polymake:2000}, and the computation of the zeta functions can be verified using the Sage code of \cite{ViuSos}.

\begin{example}\label{ex:ELTconj}
Let 
$$f(x_1, x_2, x_3, x_4, x_5, x_6) = x_1^8 + x_2^5 + x_3^{24} + x_4^{13} + x_5^{17} + x_6^{14} + x_3x_5x_6 + x_2^3x_4 + x_4x_5x_6 + x_1x_3^2x_4x_6.$$
Then $f$ is a nondegenerate polynomial with an isolated singularity at $0$ whose Newton polyhedron is simplicial and convenient, with sixteen compact facets and ten vertices. There are five compact facets containing the face with vertices $\{(8, 0, 0, 0, 0, 0), (0, 5, 0, 0, 0, 0), (0, 0, 1, 0, 1, 1), (0, 3, 0, 1, 0, 0)\},$
each of which contributes the candidate pole $-69/40$. All of these facets are $B_1$, and no other facets contribute $-69/40$. Two of these facets are obtained by adding either $\{(0, 0, 0, 0, 0, 14), (0, 0, 0, 1, 1, 1)\}$ or $\{(0, 0, 0, 0, 0, 14), (1, 0, 2, 1, 0, 1)\}$ to the above face. The existence of these two facets implies that the condition in \cite[Conjecture 1.3]{ELT} is not satisfied, so the conjecture predicts that $\exp(2 \pi i (-69/40))$ is an eigenvalue of monodromy. But this is not one of the $1912$ eigenvalues of monodromy at the origin. 

The local topological zeta function of $f$ is
$Z_{\t}(s) = -\frac{6142656 s^3 - 2948088 s^2 - 93769198 s - 115234075}{17 (s + 1) (104 s + 157)^2 (168 s + 275)},$
which does not have $-69/40$ as a pole. One can deduce from Theorem~\ref{thm:nopole} that there is a set of candidate poles for $Z_{\mot}(T)$ which does not contain $-69/40$. 
\end{example}

\begin{example}\label{ex:ELTdim4}
Let
$$f(x_1, x_2, x_3, x_4) = x_1^{10} + x_2^{5} + x_3^6 + x_4^{10} + x_2x_4^2 + x_2x_3 + x_1x_2 + x_1x_3.$$
Then $f$ is a nondegenerate polynomial with an isolated singularity at the origin whose Newton polyhedron is
simplicial and convenient, with six compact facets and eight vertices. There are two compact facets containing the face with vertices $\{(0, 0, 0, 10), (0, 1, 0, 2), (1, 0, 1, 0)\}$, each of which contributes the candidate pole $-19/10$. 
Both of these facets are $B_1$, and they are obtained by adding either $(0, 1, 1, 0)$ or $(1, 1, 0, 0)$.
No other facets contribute $-19/10$.
 The condition in \cite[Conjecture 1.3]{ELT} is not satisfied, so their conjecture predicts that $\exp(2 \pi i (-19/10))$ is an eigenvalue of monodromy. However, the only eigenvalues of monodromy are $\{1, \pm i\}$. This contradicts the claim in \cite{ELT} that they have proven \cite[Conjecture 1.3]{ELT} for polynomials in four variables. 

The local topological zeta function of $f$ is 
$Z_{\mathrm{top}}(s) = \frac{s + 7}{(s+1)(4s + 7)},$
which does not have $-19/10$ as a pole. One can deduce from Theorem~\ref{thm:nopolesimplicial} that there is a set of candidate poles for $Z_{\operatorname{mot}}(T)$ which does not contain $-19/10$.
We thank M. H. Quek for informing us that there are counterexamples to \cite[Conjecture 1.3]{ELT} in four variables. 
\end{example}

\begin{example}\label{ex:QuekConj}
Let 
$$f(x_1, x_2, x_3, x_4, x_5) = x_1^{21} + x_2^{22} + x_3^{24} + x_4^{6} + x_5^{12} + x_1x_2 + x_2x_3^6x_5^5 + x_3 x_4x_5 + x_3^2x_5^9.$$
Then $f$ is a nondegenerate polynomial with an isolated singularity at the origin whose Newton polyhedron is simplicial and convenient, with ten compact facets and nine vertices. Every compact facet contains the face with vertices $\{(1, 1, 0, 0, 0), (0, 0, 1, 1, 1)\}$, and the candidate pole of every facet is $-2$. These ten facets have a choice of compatible apices in the sense of \cite[Definition 5.1.5]{Quek22}, so \cite[Question 5.1.8]{Quek22} predicts that $\{-1\}$ is a set of candidate poles for the local naive  motivic zeta function. 

The local topological zeta function of $f$ is $ Z_{\t}(s) = \frac{7 s^2 + 45 s + 96}{24 (s + 1) (s + 2)^2}$, so any set of candidate poles for the  local naive motivic zeta function contains $-2$. When $n=6$, there are counterexamples to \cite[Question 5.1.8]{Quek22} whose candidate pole is not an integer. 
\end{example}

\begin{example}\label{ex:ELTprop}
In \cite[Definition 3.1]{ELT}, the authors give a different definition of a $B_1$-facet when the facet is non-compact. They say that a facet $F$ with 
$\Unb(C_F) \neq \emptyset$
is a \emph{$B_1$-facet of non-compact type} if the image $\overline{F}$ of $F$ under the projection $\mathbb{R}^n \to
\mathbb{R}^n/\langle \Unb(C_F) \rangle$
 is a $B_1$-facet. Then \cite[Proposition 3.7]{ELT} claims that if a pole $\alpha \not= -1$ is contributed only by a single $B_1$-facet, then $\alpha$ is not a pole of $Z_{\t}(s)$. Consider the nondegenerate polynomial
$$f(x_1, x_2, x_3, x_4) = x_1x_3 + x_2x_3 + x_2x_4^5 + x_4^6.$$
Then $\Newt(f)$ has four vertices and eight facets, one of which is compact. There is a $B_1$-facet $F$ of non-compact type with vertices $\{(1, 0, 1, 0), (0, 1, 1, 0), (0, 1, 0, 5)\}$ and $\Unb(C_F) = \{ e_1, e_2 \}$ 
whose candidate pole is $-6/5$. It is the only facet whose candidate pole is $-6/5$, so \cite[Proposition 3.7]{ELT} claims that $-6/5$ is not a pole of $Z_{\t}(s)$. In fact, $Z_{\t}(s) = \frac{6}{(s + 1) (5 s + 6)}$. 

At the origin, the monodromy zeta function is $1$. The singular locus of $X_f$ is the set of points of the form $\{(c, -c, 0, 0)\}$. At any $c \not= 0$, the monodromy zeta function is equal to $ -(t - 1) (t^4 + t^3 + t^2 + t + 1)$, so $\exp(2 \pi i (-6/5))$ is an eigenvalue of monodromy. 
\end{example}

\begin{example}\label{ex:ELT518}
Let
$$f(x_1, x_2, x_3, x_4) = x_1^{45} + x_2^{45} + x_3^{23} + x_4^{18} + x_1x_4 + 2x_2x_4.$$
Then $f$ is a nondegenerate polynomial with an isolated singularity at the origin whose Newton polyhedron is convenient, with two compact facets and six vertices. The vertices 
$$\{(45, 0, 0, 0), (0, 45, 0, 0), (0, 0, 23, 0), (1, 0, 0, 1), (0, 1, 0, 1)\}$$
form a facet $F$. Then $F$ is not a ``$B$-facet'' in the sense of \cite[Definition 3.10]{ELT}, but any four affinely independent vertices form a $B_1$-facet. This contradicts \cite[Lemma 5.18]{ELT}. See \cite{Selyanin} for a revised classification of $B$-facets for $n = 4$. 

Observe that $F$ is the unique facet with candidate pole 
$-1103/1035$. 
The topological zeta function is $Z_{\mathrm{top}}(s) = \frac{67s + 1103}{(s + 1)(1035s + 1103)}$, so $-1103/1035$ is a pole of the topological zeta function. The lattice point $(1, 44, 0, 0) = 1/45 (45, 0, 0, 0) + 44/45 (0, 45, 0, 0)$ is contained in $F$, and the lattice simplex with vertices
$$\{(1, 44, 0, 0), (0, 0, 23, 0), (1, 0, 0, 1), (0, 1, 0, 1)\}$$
is contained in $F$ and is not $B_1$. This implies that $F$ is not pseudo-$\UB$ (see Definition~\ref{d:pseudoUB1} below), and so $\exp(2 \pi i (-1103/1035))$ is an eigenvalue of monodromy by Theorem~\ref{thm:optimalclass}(1). 
\end{example}

\section{A nonnegative formula for nearby eigenvalues} \label{sec:nearby}

Here and throughout, $f \in \k[x_1, \ldots, x_n]$ is a nondegenerate polynomial that vanishes at $0$.  
In this section, we do not assume that $\Newt(f)$ is simplicial. 
For a geometric point $x$ in the hypersurface $X_f$, we write $\widetilde{m}_x(\alpha)$ to denote the multiplicity of $\exp(2 \pi i \alpha)$ in the virtual representation $\widetilde{\chi}(\cF_x) := \sum_i (-1)^i \widetilde{H}^i(\cF_x,\C)$,  where $\widetilde{H}$ denotes reduced cohomology. We define 
\begin{equation} \label{eq:Etilde}
\widetilde{E}(\cF_x)  := \sum_{[\alpha] \in \Q/\Z}
\widetilde{m}_x(\alpha) [\alpha]   \in \Z[\Q/\Z], 
\end{equation}
where $\Z[\Q/\Z]$ is the group algebra of $\Q/\Z$. 
For example, when $\Newt(f)$ is convenient, then $X_f$ has an isolated singularity at the origin \cite[1.13(ii)]{Kouchnirenko76}.  In this case, 
the Milnor fiber $\cF_0$ has the homotopy type of a wedge sum of $(n - 1)$-dimensional spheres \cite{Milnor68}, and 
$\exp(2 \pi i \alpha)$ is a nearby eigenvalue for reduced cohomology if and only if the coefficient $\widetilde{m}_0(\alpha)$ is nonzero.

Note that $\widetilde{E}(\cF_x)$ encodes information equivalent to that in the monodromy zeta function $\zeta_{x}(t)$.  This alternative notation is useful 
when studying how the monodromy action interacts with finer invariants of the cohomology of the Milnor fiber, such as its mixed Hodge structure \cite[Section~6.3]{Stapledon17}.  
The additive structure of $\widetilde{E}(\cF_x)$ and the restriction to reduced cohomology are also more natural from a combinatorial perspective: it aligns with standard formulas for local $h$-polynomials and thereby leads to the nonnegative formula for nearby eigenvalues that we prove here.

Given $\beta \in \Q$, let $D(\beta) \in \Z_{> 0}$ be the denominator of $\beta$, 
written as a reduced fraction. Fix  $M \in \Z_{> 0}$, and
consider the following $\Z$-module homomorphism:
\[
\Psi_M \colon \Z[\Q/\Z] \to \Z[\Q/\Z], \text{ given by } 
\Psi_M([\beta])  =  \begin{cases}
[\beta] &\textrm{ if } M \textrm{ divides } D(\beta), \\
0   &\textrm{ otherwise. }
\end{cases}
\]
For example, when $M = 1$, then $\Psi_M$ is the identity map.

Let $x_I$ denote a 
general point in the coordinate subspace $\A^I \subset \A^n$ for $I \subset [n]$. Under appropriate conditions, the main result of this section is a formula with nonnegative integer coefficients for
\[
\Psi_M \left( \sum_{\A^I \subset X_f} (-1)^{n - 1 - |I|}
\widetilde{E}(\cF_{x_I}) \right).
\]
For example, when $\Newt(f)$ is simplicial or convenient, we will see that we obtain a nonnegative formula when $M=1$ (see Remark~\ref{r:simplicialconvenient}). 
The set of coordinate subspaces contained in $X_f$ depends only on $\Newt(f)$.  
Explicitly, $\A^I \subset X_f$ if and only if $\R^I_{\ge 0} \cap \Newt(f) = \emptyset$.

\subsection{The local $h$-polynomial}

Let $\Delta'$ be a simplicial fan with support $\R^n_{\ge 0}$. 
If we forget the lattice structure of the fan $\Delta'$, we may view $\Delta'$ as encoding a triangulation of a simplex, e.g., by slicing with a transverse hyperplane.
 For $C'$ a cone in $\Delta'$, let $\sigma(C')$ be the smallest face of the cone $\R^n_{\ge 0}$ containing $C'$. Set
 $$e(C') := \dim \sigma(C') - \dim C'.$$
We use the following definition of the local $h$-polynomial (see, for example, \cite[Lemma~4.12]{KatzStapledon16}).

\begin{definition}\label{d:localdef}
Let $\Delta'$ be a simplicial fan with support $\R^n_{\ge 0}$, and let $C'$ be a cone in $\Delta'$. 
Then the local $h$-polynomial $\ell(\Delta', C'; t)$ is defined as
\[
\ell (\Delta',C';t) := \sum_{\substack{ C' \subset C \in \Delta'}} (-1)^{\codim(C)} t^{\codim(C') - e(C)} (t - 1)^{e(C)}.
\]
\end{definition}
The local $h$-polynomial has several important properties. Most important for our purposes is that its coefficients are nonnegative integers 
\cite{Athanasiadis12b}.
We refer the reader to \cite{KatzStapledon16} for more details and a more general setting. Observe that
\begin{equation}\label{e:specializelocalreplace}
\ell (\Delta',C';1) := \sum_{\substack{ C' \subset C \in \Delta' \\ e(C) = 0}} (-1)^{\codim(C)}.
\end{equation}

Recall that $\Delta$ denotes the fan over the faces of $\Newt(f)$. The following definition extends the notions of $\gens(C)$,$\ver(C)$ and $\Unb(C)$, for $C$ a cone in $\Delta$. 

\begin{definition}
 Let $C'$ be a polyhedral 
 cone contained in a cone of $\Delta$, such that every ray of $C'$ not in $\Delta$ 
  intersects $\partial \Newt(f)$ at a lattice point.
    Let $\gens(C') = \ver(C') \cup \Unb(C')$ be the set of distinguished lattice point generators of the rays of $C'$ defined as follows: 
$$\ver(C') = \{ w \in \Z^n : \{ w \} = r \cap \partial \Newt(f) \textrm{ for some ray } r \textrm{ of } C' \}, \text{ and}$$ 
$$\Unb(C') = \{ e_i : e_i \in C', \R_{\ge 0} e_i \cap \Newt(f) = \emptyset \}.$$
\end{definition}

When $C'$ is simplicial, we need the following  definition.

\begin{definition}\label{d:boxgeneral}
 Let $C'$ be a simplicial cone contained in a cone of $\Delta$, such that every ray of $C'$ not in $\Delta$ 
  intersects $\partial \Newt(f)$ at a lattice point. Let $\gens(C') = \{w_1,\ldots,w_r \}$, and define finite sets 
$\BBox_{C'}^\circ$
and $\BBox_{C'}$ as follows:  
\[
\BBox_{C'}^\circ := 
\Big\{ w \in \Z^n : 
w = \sum_{i = 1}^r \lambda_i w_i, 0 < \lambda_i < 1 \Big\} \mbox{ \ \ and \ \ } \BBox_{C'} 
= 
\Big\{ w \in \Z^n : 
w = \sum_{i = 1}^r \lambda_i w_i, 0 \leq \lambda_i < 1 \Big\}.
\]
When $C' = \{0\}$, then $\BBox_{C'}^{\circ} = \BBox_{C'} = \{0\}$. 
\end{definition}

\noindent Note that 
$\BBox_{C'} = \cup_{C \subset C'} \BBox_{C}^\circ$,
and 
$\BBox_{C'} = \{ 0 \}$ if $\gens(C') \subset \{ e_1,\ldots, e_n \}$.

Recall that if $F \subset \R^n_{\ge 0}$ is a rational polyhedron whose affine span $\aff(F)$ does not contain the origin, we write $\rho_F$ for the lattice distance of $F$ to the origin. 
If $F$ is a lattice polytope, then 
we may consider the \emph{normalized volume} $\Vol(F) \in \Z_{> 0}$ of $F$, i.e., 
the Euclidean volume on $\aff(F)$ scaled such that the volume of a unimodular lattice simplex is $1$. 
When $F = \emptyset$, $\rho_F = \Vol(F) = 1$. 
We will need the following basic lemma.

\begin{lemma}\label{l:normalizedvolume}
 Let $C'$ be a simplicial cone contained in a cone of $\Delta$ such that every ray of $C'$ not in $\Delta$ 
  intersects $\partial \Newt(f)$ at a lattice point.
Let $F$ be the convex hull of the elements of $\gens(C')$, and let $\phi$ be the linear function on $C'$ with value $1$ on $F$.
Then
\[
\sum_{w \in \BBox_{C'}} [-\phi(w)] = \Vol(F) \sum_{i = 0}^{\rho_F - 1} [i/\rho_F].
\]
\end{lemma}
\begin{proof}
The result follows from the fact that $\phi$ induces a group homomorphism:
\[
\widetilde{\phi} \colon (\Span(C') \cap \Z^n)/(\Z w_1 + \cdots + \Z w_r) \to \Q/\Z, 
\] 
where the domain is a finite set in bijection with $\BBox_{C'}$, 
$\Ima(\widetilde{\phi}) = \frac{1}{\rho_F}\Z/\Z$ and 
$|\ker(\widetilde{\phi}) | = \Vol(F)$. 
See also \cite[Examples~4.12-4.13]{Stapledon17}.
\end{proof}

\subsection{Nearby eigenvalues along coordinate subspaces} 
We now state our nonnegative formula for the multiplicity of nearby eigenvalues along coordinate subspaces. We first introduce some notation.

Recall that  $\psi$ is the unique piecewise linear function on $\mathbb{R}^n_{\ge 0}$ with value $1$ on all interior faces of $\partial \Newt(f)$.
Let $C$ be a cone in $\Delta$. If $C = C_F$ for some face $F$ of $\Newt(f)$, then we set $\rho_C := \rho_F$ to be the lattice distance of $F$ to the origin. Otherwise, we set $\rho_C := 1$. 
Equivalently, $\rho_C$  is the smallest positive integer such that $\rho_C \psi|_C$ is the restriction of a $\Z$-linear function on $\Span(C)$. 
Observe that if $C \subset \widetilde{C} \in \Delta$, then $\rho_C$ divides $\rho_{\widetilde{C}}$.

Fix $M \in \Z_{> 0}$. Let $\Delta_{M}$ be the (possibly empty) subfan of $\Delta$ consisting of all maximal cones $C$ in $\Delta$ such that $M$ divides $\rho_C$, together with all the faces of $C$.
Observe that
if $C$ is a cone in $\Delta$ and $M$ divides $\rho_C$, then all cones in $\Delta$ containing $C$ lie in $\Delta_M$.

If $\Delta'$ is a fan refining $\Delta$, 
 let $\Delta'_M$ denote the restriction of 
$\Delta'$ to $\Delta_M$. Given a cone $C'$ in $\Delta'$ contained in a cone of $\Delta$, let $\tau(C')$ denote the smallest cone in $\Delta$  containing $C'$. 

\begin{theorem}\label{t:nonnegativeVarchenko}
Assume that $f$ is nondegenerate.
Let $M \in \Z_{>0}$ and let $\Delta'$ be a simplicial fan refining $\Delta$. 
Assume that
every ray of $\Delta' \smallsetminus \Delta$
 intersects the boundary of  $\Newt(f)$ at a lattice point, 
and $\Unb(C') = \Unb(\tau(C'))$ 
for all $C'$ in $\Delta'_M$.
Then 
\begin{equation}\label{e:nonnegativeVarchenkoreplace}
\Psi_{M} \left( \sum_{\A^I \subset X_f} (-1)^{n - 1 - |I|}
\widetilde{E}(\cF_{x_I}) \right)
= 
\Psi_{M} \left(
\sum_{  C' \in \Delta' }  
\ell(\Delta', C'; 1)
\sum_{ w \in \BBox_{C'}^\circ}  [-\psi(w)] \right).
\end{equation}
\end{theorem}

\begin{remark}\label{r:simplicialconvenient}
We consider two important special cases when such a $\Delta'$ exists when $M=1$. Firstly, if $\Newt(f)$ is simplicial, then the hypotheses of the theorem hold with $\Delta' = \Delta$. Secondly, if $\Newt(f)$ is convenient, then there exists a simplicial fan $\Delta'$ that refines $\Delta$ and has the same rays as $\Delta$.
In this case, as $\Unb(C') = \Unb(\tau(C')) = \emptyset$ for all $C'$ in $\Delta'$, the hypotheses of the theorem hold.

As a simple example where the hypotheses fail, suppose there exists $C \in \Delta$ with $\dim C = 3$
and $|\ver(C)| = |\Unb(C)| = 2$. 
Then there is no simplicial refinement of $C$ in which all maximal cones contain $\Unb(C)$. 
\end{remark}

\begin{remark}\label{r:monodromyzetaformula}
When $M = 1$, \eqref{e:nonnegativeVarchenkoreplace}
 can be restated in terms of monodromy zeta functions at $x_I$, as follows:
\[
\prod_{\A^I \subset X_f} \left( \frac{\zeta_{x_I}(t)}{1 - t}\right)^{(-1)^{n-1-|I|}} = \prod_{C' \in \Delta'} \prod_{w \in \BBox_{C'}^\circ} (1 - \exp(-2 \pi i \psi(w) )t)^{\ell(\Delta', C'; 1)}.
\]
For the remainder, we work in terms of $\widetilde E(\cF_{x_I})$ rather than the corresponding monodromy zeta function.
\end{remark}

We now give three examples which show that the sum $\sum_{\A^I \subset X_f} (-1)^{n - 1 - |I|}
\widetilde{E}(\cF_{x_I})$ appearing in Theorem~\ref{t:nonnegativeVarchenko} can fail to detect nearby eigenvalues and  can have strictly negative coefficients. 
\begin{example}\label{ex:ELTprop2}
In Example~\ref{ex:ELTprop}, consider the facet $F$ with candidate pole $\alpha = -6/5$. It was shown that 
 there exists   $x \in X_f$ arbitrarily close to the origin such that $\widetilde{m}_x(\alpha)$ is nonzero. 
On the other hand, $X_f$ is smooth at $x_I$ for $I \neq \emptyset$, so 
 $\sum_{\A^I \subset X_f} (-1)^{n - 1 - |I|}
\widetilde{E}(\cF_{x_I}) = -\widetilde{E}(\cF_{0}) = 1$.  In particular, $[-\alpha]$ does not appear. 
 See also Example~\ref{ex:ELTprop3} below.

\end{example}

\begin{example}\label{e:ELT}
The following example appeared in \cite[Example~7.4]{ELT}. 
Consider the nondegenerate polynomial $f(x_1,x_2,x_3,x_4) = x_3^6 + x_2^4x_3^5 + (x_1^2 + x_4^2)x_2^{13}x_3^2$.
Let $F \subset \R^4$ be the $3$-dimensional lattice simplex with vertices
$w_1 = (0,0,6,0)$, 
$w_2 = (0,4,5,0)$, 
$w_3 = (2,13,2,0)$,
$w_4 = (0,13,2,2)$. Then $F$ is the unique compact facet of $\Newt(f)$, and it has candidate pole $\alpha = -1/3$. 
The authors showed that 
$\widetilde{E}(\cF_0) = -(\sum_{0 \le i < 24} [i/24] - \sum_{1 \le i < 6} [i/6])$, 
 and that $[\alpha]$ does not appear in $\widetilde{E}(\cF_{x_I})$  for a general point $x_I$ in any $\A^I \subset X_f$, but that there exists   $x \in X_f$ arbitrarily close to the origin such that $\widetilde{E}(\cF_{x})$ contains $[\alpha]$. We have 
\[
(-1)^{n - 1 - |I|}\widetilde{E}(\cF_{x_I}) = \begin{cases}
\sum_{0 \le i < 24} [i/24] - \sum_{1 \le i < 6} [i/6] &\textrm{if } I = \emptyset \\
\sum_{1 \le i < 5} [i/5] &\textrm{if } I = \{ 2 \} \\
-\sum_{0 \le i < 78} [i/78] + \sum_{1 \le i < 6} [i/6] &\textrm{if } I = \{ 1 \} \textrm{ or } \{ 4 \} \\
\sum_{0 \le i < 78} [i/78] - \sum_{1 \le i < 6} [i/6] &\textrm{if } I = \{ 1 , 4 \}  \\
-[1/2] &\textrm{if } I = \{ 1, 2 \} \textrm{ or } \{ 2, 4 \} \\
[1/2] &\textrm{if } I = \{ 1, 2 , 4 \}.
\end{cases}
\]
In particular, $\sum_{\A^I \subset X_f} (-1)^{n - 1 - |I|}
\widetilde{E}(\cF_{x_I})$ has strictly negative integer coefficients.  See also Example~\ref{e:ELT2} below.

\end{example}

\begin{example}
Consider $f$ as a polynomial in $n + 1$ variables, i.e.,  $f \in \k[x_1, \ldots, x_n] \subset \k[x_1, \ldots, x_n, x_{n + 1}]$. Then one can verify that both sides of \eqref{e:nonnegativeVarchenkoreplace} are identically zero. Geometrically, if $I \subset \{1, \dotsc, n\}$ and $I' = I \cup \{n + 1\}$, then $(-1)^{n - 1 - |I|} \widetilde{E}(\mathcal{F}_{x_I}) + (-1)^{n - 1 - |I'|}\widetilde{E}(\mathcal{F}_{x_{I'}}) = 0$.
\end{example}

We will deduce Theorem~\ref{t:nonnegativeVarchenko} from Varchenko's formula for the monodromy zeta function of a nondegenerate singularity \cite[Theorem~4.1]{Varchenko76}.
First, recall that $\rho_C$ is the lattice distance from the origin to $F$ if $C = C_F$, and is $1$ otherwise. 
Recall that $\Delta_M$ is the (possibly empty) subfan of $\Delta$ consisting of all maximal cones $C$ in $\Delta$ such that $M$ divides $\rho_C$, together with all the faces of $C$.
\begin{lemma}\label{l:conesnotinDeltaM}
Let $M \in \Z_{>0}$. 
If $C$ is a cone in $\Delta \smallsetminus \Delta_M$, then 
$\Psi_M([i/\rho_C]) = 0$ for any $i \in \Z$. In particular, $\Psi_M([-\psi(w)]) = 0$ for all $w \in C \cap \Z^n$. 
\end{lemma}
\begin{proof}
If $\Psi_M([i/\rho_C]) \neq 0$, then $M$ divides $D(i/\rho_C)$. 
Therefore for $\widetilde{C}$ a maximal cone in $\Delta$ containing $C$, $M$ divides $\rho_{\widetilde{C}}$ and hence $C \in \Delta_M$, a contradiction. The second statement follows since the restriction of $\rho_C \psi$ to $C$ is the restriction of a $\Z$-linear function on $\Span(C)$ for all cones $C$ in $\Delta$. 
\end{proof}

Observe that when $\Delta_M$ is empty, the proposition below states that 
$\Psi_{M} ( \widetilde{E}(\cF_{0})) = 0$.

\begin{proposition} \label{p:varchenko}
Let $M \in \Z_{>0}$ and let $\Delta'_M$ be a simplicial fan refining $\Delta_M$. 
Assume that
every ray of $\Delta'_M \smallsetminus \Delta_M$ intersects the boundary of $\Newt(f)$ at a lattice point, 
and $\Unb(C') = \Unb(\tau(C'))$ 
for all $C'$ in $\Delta'_M$.
Then 
\[
\Psi_{M} ( \widetilde{E}(\cF_{0}) )
=   \sum_{\substack{ C' \in \Delta'_M \\ \Unb(C') = \emptyset  \\ e(C') = 0}} (-1)^{\dim C' + 1}  
\Psi_{M} \left( \sum_{w \in \BBox_{C'}} [-\psi(w)] \right).
\]
\end{proposition}
\begin{proof}
Suppose that $C$ in $\Delta$ satisfies $\Unb(C) = \emptyset$. Then  $C = C_F$ for some bounded face $F$ of $\Newt(f)$, and we may set $\Vol(C) := \Vol(F)$. For example, when $C = \{ 0 \}$, then $F = \emptyset$ and $\Vol(C) = 1$. 
With this notation, Varchenko's formula for the monodromy zeta function  \cite[Theorem~4.1]{Varchenko76} states that
\begin{equation}\label{e:varchenko}
\widetilde{E}(\cF_{0}) 
= \sum_{\substack{ C \in \Delta \\ \Unb(C) = \emptyset \\ e(C) = 0}} (-1)^{\dim C + 1}  \Vol(C) \sum_{i = 0}^{\rho_C - 1} [i/\rho_C].
\end{equation}

Let $S = \{C \in \Delta_M : \Unb(C) = \emptyset, e(C) = 0 \}.$
Applying $\Psi_M$ to both sides of the above equation, and using Lemma~\ref{l:conesnotinDeltaM}, we obtain the equation

\begin{equation*}
\Psi_M ( \widetilde{E}(\cF_{0}) )
= \sum_{C \in S} (-1)^{\dim C + 1}  \Vol(C) \Psi_M \left( \sum_{i = 0}^{\rho_C - 1} [i/\rho_C] \right).
\end{equation*}

Let $C \in S$ and set $S'_C = \{ C' \in \Delta'_M : \tau(C') = C, \dim C' = \dim C \}$. Let $C' \in S'_C$.
By assumption, $\Unb(C') = \Unb(C) = \emptyset$.
Then
 $C = C_F$ for some lattice polytope $F$, and 
 $C'$ is the cone over a lattice polytope $G \subset F$. Define $\Vol(C') := \Vol(G)$ and $\rho_{C'} := \rho_{G}$.  Note that $\dim C' = \dim C$ implies that $\rho_{C'} = \rho_{C}$.
By the additivity of normalized volume, we have 
\[
\Psi_M ( \widetilde{E}(\cF_{0}) )
= \sum_{C \in S} (-1)^{\dim C + 1} \Psi_M \left( \sum_{i = 0}^{\rho_{C} - 1} [i/\rho_{C}] \right) \sum_{C' \in S'_C} \Vol(C').
\]
Since $\dim C' = \dim C$, the condition $e(C') = 0$ is equivalent to the condition $e(C) = 0$. Let $S' = \{C' \in \Delta'_M : \Unb(C') = \emptyset, e(C') = 0 \}.$
Then rearranging the above equation gives
\[
\Psi_M ( \widetilde{E}(\cF_{0}) )
= \sum_{C' \in S'} (-1)^{\dim C' + 1} \Psi_M \left( \Vol(C')   \sum_{i = 0}^{\rho_{C'} - 1} [i/\rho_{C'}]     \right) \\
\]
By Lemma~\ref{l:normalizedvolume}, we obtain our desired result.
\end{proof}

We also need the following remark. 
Given $c = (c_1,\ldots,c_n) \in \k^n$, let 
$f_c(x_1,\ldots,x_n) := f(x_1 + c_1,\ldots,x_n + c_n)$.
Consider a coordinate subspace $\A^I \subset X_f$, and a general point $x_I$ in $\A^I$. 
Let $J = [n] \smallsetminus I$, and consider the projection map
$\pr_{J} \colon \R^n \to \R^{J}$ and the polyhedron $\Newt(f)_J := \pr_{J}(\Newt(f)) \subset \R^{J}$.
Let $g$ be a nondegenerate polynomial with Newton polyhedron $P$ and Milnor fiber $\widehat{\cF}_{0}$ at the origin. By \cite[Theorem~4.1]{Varchenko76}, $\widetilde{E}(\widehat{\cF}_{0})$ depends only on $P$, and not on the choice of $g$. We set $$\widetilde{E}(P) := \widetilde{E}(\widehat{\cF}_{0}).$$ 

\begin{remark}\label{r:corofproofv3}
With the notation above,
$\cF_{x_I}$ is the Milnor fiber of $f_{x_I}$ at the origin. It follows from \cite[Proposition~7.2]{ELT} and its proof that $f_{x_I}$ is nondegenerate with Newton polyhedron equal to $\Newt(f)_J \times \R_{\ge 0}^{I}$.
Then 
\begin{equation}\label{e:reducecompact}
\widetilde{E}(\cF_{x_I}) = \widetilde{E}(\Newt(f)_J \times \R_{\ge 0}^I) = \widetilde{E}(\Newt(f)_J).
\end{equation}
We deduce that if the coefficient of $[\alpha]$ in $\widetilde{E}(\Newt(f)_J)$ is nonzero for some such choice of $I$, then $\exp(2 \pi i \alpha)$ is a nearby eigenvalue of monodromy (for reduced cohomology). 
\end{remark}
\noindent
Recall that for a set of vectors $S$, $\langle S \rangle$ denotes the cone that they span. We now prove Theorem~\ref{t:nonnegativeVarchenko}. Our strategy is to apply Proposition~\ref{p:varchenko} to each coordinate projection of $\Newt(f)$. 

\begin{proof}[Proof of Theorem~\ref{t:nonnegativeVarchenko}]
Consider a coordinate subspace $\A^I$ in $X_f$. 
Let $J = [n] \smallsetminus I$, and consider the projection map
$\pr_{J} \colon \R^n \to \R^{J}$ and the polyhedron $\Newt(f)_J = \pr_{J}(\Newt(f)) \subset \R^{J}$.
By \eqref{e:reducecompact},
$\widetilde{E}(\cF_{x_I}) =  \widetilde{E}(\Newt(f)_J)$, 
where $\widetilde{E}(\Newt(f)_J)$ is the invariant $\widetilde{E}$ applied to the Milnor fiber at the origin of any nondegenerate polynomial with Newton polyhedron $\Newt(f)_J$. Our first goal is to  apply Proposition~\ref{p:varchenko} to compute $\Psi_M(\widetilde{E}(\Newt(f)_J))$.

If $C \subset \R^n_{\ge 0}$ is a cone, then we use the notation 
$C_J := \pr_J(C)$.
Let $\Delta_J$ be the fan over the faces of $\Newt(f)_J$. 
Then $\Delta_J = \{ C_J : C \in \Delta, \R_{\ge 0}^I \subset \langle \Unb(C) \rangle \}$.
 Let  $\Delta_{J, M}$ be the subfan of $\Delta_J$ consisting of all maximal cones $C_J$ in $\Delta_J$ such that $M$ divides $\rho_{C_J}$, together with all the faces of $C_J$.
Observe that if $C \in \Delta$ and  $\R_{\ge 0}^I \subset \langle \Unb(C) \rangle$, then $\rho_{C} = \rho_{C_J}$. It follows that $\Delta_{J, M} = \{ C_J : C \in \Delta_M, \R_{\ge 0}^I \subset \langle \Unb(C) \rangle \}$.

If $\Delta_{J, M}$ is empty, then 
 Proposition~\ref{p:varchenko} implies that $\Psi_M(\widetilde{E}(\Newt(f)_J)) = 0$. Assume that $\Delta_{J, M}$ is nonempty. Then  
$\R^I_{\ge 0} \in \Delta_M$.  In order to  apply Proposition~\ref{p:varchenko}, we want to construct a simplicial refinement of $\Delta_{J, M}$. 
Since each ray in $\Delta' \smallsetminus \Delta$ intersects $\Newt(f)$, no such ray is contained in $\R^I_{\ge 0} $, and hence $\R^I_{\ge 0} \in \Delta'_M$.
Consider the simplicial fan $\Delta'_{J, M} = \{ C_J' : C' \in \Delta'_M, \R_{\ge 0}^I \subset \langle \Unb(C') \rangle \}$. 
This is the star of $\R^I_{\ge 0}$ in $\Delta'_M$, and  
it follows from \cite[Exercise 3.4.8]{CoxLittleSchenck11} that $\Delta'_{J, M}$ is a refinement of $\Delta_{J, M}$.

We next verify that $\Delta'_{J, M}$ satisfies the hypotheses of Proposition~\ref{p:varchenko}. 
Firstly, the intersection of a ray of $\Delta'_{J, M} \smallsetminus \Delta_{J, M}$ with the boundary of $\Newt(f)_J$ is the image of the intersection of a ray in $\Delta'_M \smallsetminus \Delta_M$ with the boundary of $\Newt(f)$, and hence is a lattice point.
Secondly, for a cone $C \subset \R^n$ containing $\R_{\ge 0}^I$, 
let $\Unb_J(C_J) :=  \{ e_i \in \R^J : e_i \in C_J, \R_{\ge 0}e_i \cap \Newt(f)_J = \emptyset \}$. 
 Fix $C'$ in $\Delta'_M$ such that $\R_{\ge 0}^I \subset \langle \Unb(C') \rangle$. Since $\Unb(C') =  \Unb(\tau(C'))$ by assumption, we compute:
 \[
 \Unb_J(C_J') = \{ \pr_J(e_i) : e_i \in \Unb(C'), i \notin I \} = \{ \pr_J(e_i) : e_i \in \Unb(\tau(C')), i \notin I \} = \Unb_J(\tau(C')_J). 
 \]
Moreover, $\tau(C')_J$ is the smallest cone in $\Delta_J$ containing $C_J'$. We conclude that the hypotheses of Proposition~\ref{p:varchenko} hold. 

Let $\psi_J$ be the 
unique piecewise linear function on $\R^J$ with value $1$  on all interior faces of
$\partial \Newt(f)_J$.  Let $\sigma_J(C_J')$ be the smallest face of $\R^J_{\ge 0}$ containing $C_J'$, and let  
 $e_J(C_J') = \dim \sigma_J(C_J') - \dim C_J'$.
Then applying Proposition~\ref{p:varchenko} gives
\[
 \Psi_M(\widetilde{E}(\Newt(f)_J)) 
=   \sum_{\substack{ C_J' \in \Delta'_{J, M} \\ \Unb_J(C_J') = \emptyset  \\ e_J(C_J') = 0}} (-1)^{\dim C_J' + 1}  
\Psi_{M} \left( \sum_{w \in \BBox_{C_J'}} [-\psi_J(w)] \right).
\]
We compute that
$e_J(C_J') = \dim \sigma_J(C_J') - \dim C_J' = (\dim \sigma(C')  - |I|) - (\dim C' - |I|) = e(C'),$ so 
 
\[
 \Psi_M(\widetilde{E}(\Newt(f)_J)) 
=   \sum_{\substack{ C' \in \Delta'_M \\ \langle \Unb(C') \rangle = \R^I_{\ge 0}  \\ e(C') = 0}} (-1)^{\dim C' - |I| + 1}  
\Psi_{M} \left( \sum_{w \in \BBox_{C_J'}} [-\psi_J(w)] \right).
\]

Consider  $C' \in \Delta'_M$ such that  $\R^I_{\ge 0} \subset \langle \Unb(C') \rangle$. 
We define a bijection $\phi \colon \BBox_{C'} \to \BBox_{C_J'}$ as follows. Write $\gens(C') = \{ w_1,\ldots , w_r \} \cup \{ e_i : i \in I \}$. 
If $w = \sum_{i = 1}^r \lambda_i w_i + \sum_{i \in I} \mu_i e_i \in \BBox_{C'}$, 
then $\phi(w) = \sum_{i = 1}^r \lambda_i \pr_J(w_i)$. Observe that $\psi(w) = \psi_J(\phi(w))$. 
We deduce that 
\begin{equation}\label{e:boxcomparison}
\sum_{w \in \BBox_{C'}} [-\psi(w)] = \sum_{w \in \BBox_{(C')_J}} [-\psi_J(w)]. 
\end{equation}
Substituting \eqref{e:boxcomparison} into the above equation gives

\[
\Psi_M(\widetilde{E}(\Newt(f)_J))
=   \sum_{\substack{ C' \in \Delta'_M \\ \langle \Unb(C') \rangle = \R_{\ge 0}^I  \\ e(C') = 0}} (-1)^{\dim C' + |I| + 1}  
\Psi_{M} \left( \sum_{w \in \BBox_{C'}} [-\psi(w)] \right).
\]

When $\Delta_{J, M}$ is empty, we have seen that $\Psi_M(\widetilde{E}(\Newt(f)_J)) = 0$, so the above formula holds then as well.

Let $C' \in \Delta'$ and assume that  $\Psi_{M}([-\psi(w)]) \neq 0$ for some $w \in  \BBox_{C'}$. 
By Lemma~\ref{l:conesnotinDeltaM}, every cone in $\Delta$ containing $C'$ lies in $\Delta_M$, and so 
every cone in $\Delta'$ containing $C'$ lies in $\Delta'_M$.

Putting this all together, we 
compute
\begin{align*}
\Psi_{M} \left( \sum_{\A^I \subset X_f} (-1)^{n - 1 - |I|}
\widetilde{E}(\cF_{x_I}) \right)
&= \sum_{\A^I \subset X_f} (-1)^{n - 1 - |I|} \Psi_M(\widetilde{E}(\Newt(f)_J))
\\
&= \sum_{\A^I \subset X_f} (-1)^{n - 1 - |I|}
\sum_{\substack{ C' \in \Delta'_M \\ \langle \Unb(C') \rangle = \R_{\ge 0}^I  \\ e(C') = 0}} (-1)^{\dim C' + |I| + 1}  
\Psi_{M} \left( \sum_{w \in \BBox_{C'}} [-\psi(w)] \right)
\\
&= \sum_{\substack{ C' \in \Delta'_M \\ e(C') = 0}}
(-1)^{\codim C'} 
\Psi_{M} \left( \sum_{w \in \BBox_{C'}} [-\psi(w)] \right)
\\
&= \sum_{C \in \Delta'_M} \Psi_{M} \left( \sum_{w \in \BBox_{C}^\circ} [-\psi(w)] \right)
\sum_{\substack{ C \subset C' \in \Delta'_M \\ e(C') = 0}}
(-1)^{\codim C'}
\\
&= \Psi_{M} \left( \sum_{C \in \Delta'}  \sum_{w \in \BBox_{C}^\circ} [-\psi(w)] 
\sum_{\substack{ C \subset C' \in \Delta' \\ e(C') = 0}} 
(-1)^{\codim C'} \right)
\\
&= 
\Psi_{M} \left( \sum_{C \in \Delta'}  \sum_{w \in \BBox_{C}^\circ} [-\psi(w)] 
\ell (\Delta,C;1) \right).
\\
\end{align*}
Here the final equality follows from \eqref{e:specializelocalreplace}.
\end{proof}

\subsection{An existence result for nearby eigenvalues of monodromy}
\label{s:essential}
We use  Theorem~\ref{thm:intrononvanish}, a vanishing result for the local $h$-polynomial, and Theorem~\ref{t:nonnegativeVarchenko} to study nearby eigenvalues of monodromy.  

Let $\Delta'$ be a simplicial fan refining $\Delta$. 
Assume that
every ray of $\Delta' \smallsetminus \Delta$
 intersects the boundary of  $\Newt(f)$ at a lattice point. 
 Let $C'$ be a cone in $\Delta'$ with $\gens(C') = \{w_1,\ldots,w_r \}$.
We consider the function
\[
\bbox_{C'} \colon \Span(C') \cap \Z^n \to \BBox_{C'} \text{ defined by }\bbox_{C'}\bigg(\sum_{i = 1}^r \lambda_i w_i \bigg) = \sum_{i = 1}^r \{ \lambda_i \} w_i,
\]
where $\lambda_i \in \Q$ and 
$\{\lambda_i\}$ denotes the fractional part of $\lambda_i$.

Throughout this section, let $G$ be a lattice simplex contained in $\partial \Newt(f)$ such that $\1 \in \Span(G)$ and $C_G \in \Delta'$. 
Let $\psi_{G}$ be the unique linear function on $\Span(G)$ with value $1$ on $G$. Equivalently,  $\psi_{G}$ is determined by the condition that
$\psi_{G}|_{C_G} = \psi|_{C_G}$. Let $\alpha = - \psi_{G}(\1)$. Then $\alpha$ is the candidate pole associated to any proper face $F$ of $\Newt(f)$ containing $G$.  

\begin{definition}\label{d:essentialface}
The \emph{essential face}  of $G$ is the unique face 
$E \subset G$ such that  
$\bbox_{C_G}(\1)$ is in $\BBox_{C_E}^\circ$. 
\end{definition}
Equivalently, one may verify that if $\{w_1,\ldots,w_r \}$ are the vertices of $G$ and we write $\mathbf{1} = \sum_{i=1}^{r} \lambda_i w_i$ for some $\lambda_i \in \Q$, then $E$ is the unique face of $G$ with $\gens(C_E) = \{w_i : \lambda_i \notin \Z \}$. 
Note that, for any lattice point 
$w \in \Span(C_G) \cap \Z_{\geq 0}^n$,
$[ \psi_G (w)] = [ \psi(\bbox_{C_G}(w))]$ in   $\Q/\Z$.
We deduce that 
\begin{equation}\label{e:box3}
[\alpha] = [-\psi(\bbox_{C_G}(\1))]     \mbox{ in }  \Q/\Z.
\end{equation}

We deduce the following corollary of Theorem~\ref{t:nonnegativeVarchenko}. When $\Newt(f)$ is simplicial and we set $\Delta' = \Delta$,  this is equivalent to 
Corollary~\ref{c:intrononvanishingeigenvalue}.

\begin{corollary}\label{c:nonvanishingeigenvalue}
Let $\alpha \in \Q$, and let $M = D(\alpha)$. 
Let $\Delta'$ be a simplicial fan refining $\Delta$. 
Assume that
every ray of $\Delta' \smallsetminus \Delta$
 intersects the boundary of  $\Newt(f)$ at a lattice point, 
and $\Unb(C') = \Unb(\tau(C'))$ 
for all $C'$ in $\Delta'_M$.
Let $G$ be a lattice simplex contained in $\partial \Newt(f)$ such that $\1 \in \Span(G)$ and $C_G \in \Delta'$, and let $E$ be the essential face of $G$.  
Assume that $\alpha = - \psi_{G}(\1)$.
If 
$\ell(\Delta',C_E;t)$ is nonzero, then 
the coefficient of $[\alpha]$ in $\widetilde{E}(\cF_{x_I})$ is nonzero
at a general point $x_I$ of some coordinate subspace $\A^I \subset X_f$.

\end{corollary}
\begin{proof}
By definition, $\Psi_{D(\alpha)}([\alpha]) = [\alpha]$. 
By \eqref{e:box3}, $\ell(\Delta',C_E;1)[\alpha]$ is a term in the right-hand side of 
\eqref{e:nonnegativeVarchenkoreplace}.
The result now follows from the nonnegativity of the local $h$-polynomial.
\end{proof}

Extending Definition~\ref{def:UB1}, we may define the notion of $G$ being $\UB$ exactly as for (compact) faces of $\Newt(f)$. 
Explicitly, a vertex  $A$ of $G$ is an \emph{apex} with \emph{base direction} $e_\ell^*$ if $\langle e_{\ell}^*, A \rangle > 0$, and
$\langle e_{\ell}^*, V \rangle = 0$ for all  $V \in \gens(C_G)$  with $V \neq A$, i.e., for all vertices of $G$ not equal to $A$. 
Then $G$ is $\UB$ if there exists an apex $A$ in $G$ with a unique choice of base direction $e_\ell^*$, 
and $\langle e_{\ell}^*, A \rangle = 1$. 

The following definition is a special case of Definition~\ref{d:Upyramid}. 
We say that $C_G \smallsetminus C_E$ in $\lk_{\Delta'}(C_E)$
is a 
\emph{$U$-pyramid} if there exists 
an apex $A$ in $G$ 
with a unique choice of base direction $e_\ell^*$, and $A \notin
\gens(C_E)$. 

\begin{lemma}\label{l:weaklyUB1}
With the notation above, $G$ is 
$\UB$ if and only if $C_G \smallsetminus C_E$ in $\lk_{\Delta'}(C_E)$ is a $U$-pyramid.
\end{lemma}
\begin{proof}
Let $\gens(C_G) = \{ w_1,\ldots,w_r \}$, and uniquely write 
$\1 = \sum_{i = 1}^r \lambda_i w_i$ for some $\lambda_i \in \Q$. 
Let $w_j$ be an apex with a base direction $e_\ell^*$.
Let $h = \langle e_\ell^* , w_j \rangle \in \Z_{> 0}$.
Then $1 = \langle e_\ell^* , \1 \rangle = \langle e_\ell^* , \sum_{i = 1}^r \lambda_i w_i \rangle = \sum_{i = 1}^r \lambda_i \langle e_\ell^* , w_i \rangle = \lambda_j \langle e_\ell^* , w_j \rangle =  \lambda_j h$, and we deduce that $\lambda_j = 1/h \in \Q_{> 0}$. 
The result then follows since
\[
w_j \notin C_E \iff \lambda_j \in \Z \iff h = 1. \qedhere
\] 
\end{proof}

Let $\mathcal{A}_G$ be the set of apices of $G$ which are not in $E$. For a face $G'$ of $G$, let $\sigma(G')$ be the smallest face of $\R^n_{\ge 0}$ containing $G'$. 
 Let $B_G =  \{ \ell \in [n] : \textrm{ there exists } A \in \mathcal{A}_G \textrm{ with base direction } e_\ell^* \}$.
Below, we identify faces of $\R^n_{\ge 0}$ with their corresponding subsets of $[n]$ and identify simplices with their set of vertices.

The following definition is a special case of Definition~\ref{d:fullpartition}. 
Below, by associating faces of $G$ with their corresponding cones in $\Delta'$, we may view faces of $G \smallsetminus E$ as faces in $\lk_{\Delta'}(C_E)$. 
A \emph{full partition} of $G \smallsetminus E$ is a decomposition
\[
G \smallsetminus E = G_1 \sqcup G_2 \sqcup \mathcal{A}_G
\]
such that 
\begin{enumerate}

\item $\sigma(G_1 \sqcup \mathcal{A}_G \sqcup E) = [n],$

\item $\sigma(G_2 \sqcup E) = [n] \smallsetminus B_G.$

\end{enumerate}

\begin{lemma}\label{l:existencefullpartition}
With the notation above, $G \smallsetminus E$ admits a full partition.
\end{lemma}
\begin{proof}
Let $\gens(C_G) = \{ w_1,\ldots,w_r \}$, and uniquely write 
$\1 = \sum_{i = 1}^r \lambda_i w_i$ for some $\lambda_i \in \Q$. We have 
\begin{equation}\label{e:proposedfullpartition2}
G \smallsetminus E = G_1 \sqcup G_2 \sqcup \mathcal{A}_G,
\end{equation}
where $G_1 = \{ w_i : w_i \notin \mathcal{A}_G, \lambda_i \in \Z_{> 0} \}$, and $G_2 = \{ w_i : \lambda_i \in \Z_{\le 0} \}$. Note that $w_i \in \mathcal{A}_G$ implies that $\lambda_i = 1$. 
We claim that \eqref{e:proposedfullpartition2} is a full partition.

For each $w_i$ in $\gens(C_G)$, write $(w_i)_\ell \in \Z_{\ge 0}$ for the $\ell$th coordinate of $w_i$. For each coordinate 
$\ell \in [n]$, 
\begin{equation}\label{e:1}
1 = \sum_{i = 1}^n \lambda_i (w_i)_\ell.
\end{equation}
If $\ell \notin \sigma(\mathcal{A}_G \sqcup G_1 \sqcup E)$, then the right-hand side of \eqref{e:1} is a sum of nonpositive terms, a contradiction. We conclude that
$\sigma(\mathcal{A}_G \sqcup G_1 \sqcup E) = [n].$

It remains to show that
$\sigma(G_2 \sqcup E) = [n] \smallsetminus B_G.$
It follows from the definitions that $\sigma(G_2 \sqcup E) \subset [n] \smallsetminus B_G$.
It remains to prove that $[n] \smallsetminus \sigma(G_2 \sqcup E) \subset B_G$.
Suppose that $\ell \in [n] \smallsetminus \sigma(G_2 \sqcup E)$. Then all terms on the right-hand side of \eqref{e:1} are nonnegative integers, and we deduce that there is a unique index $k$ such that $\lambda_k = (w_k)_\ell = 1$ and $\lambda_i (w_i)_\ell = 0$ for $i \neq k$. If $(w_i)_\ell \ne 0$ for some $i \ne k$, then $\lambda_i = 0$ and hence $\ell \in \sigma(G_2)$, a contradiction. We deduce that 
$w_i \in \mathcal{A}_G$ has base direction $e_\ell^*$. That is, 
$\ell \in B_G$. 
\end{proof}

The following corollary is immediate from Theorem~\ref{thm:intrononvanish}, 
 together with Lemma~\ref{l:weaklyUB1} and Lemma~\ref{l:existencefullpartition}. Here Theorem~\ref{thm:intrononvanish} is 
 a consequence of Theorem~\ref{thm:nonvanish}, whose proof is the subject of Section~\ref{sec:vanishing}. 

\begin{corollary}\label{c:zeroimpliesUB1}
With the notation above,
if 
$\ell(\Delta',C_E;t)=0$, then $G$ is  $\UB$.
\end{corollary}

\subsection{Existence of simplicial refinements}

We now give a criterion for the existence of a simplicial refinement of 
$\Delta$
 that 
 satisfies the hypotheses of Corollary~\ref{c:nonvanishingeigenvalue} and allows us to prove our strongest result on the existence of eigenvalues. 
We will obtain  Theorem~\ref{t:mainsimplicialeigenvalue} as a consequence.

 We first introduce a combinatorial condition on the unbounded faces of $\Newt(f)$. 

\begin{definition}\label{d:goodprojection}
Say that a cone $C$ in $\Delta$ has \emph{good projection} if for any face $C'$ of $C$ such that 
$C' \cap \Unb(C) = \emptyset$, 
$\dim(C' + \langle \Unb(C) \rangle ) = \dim C' + |\Unb(C)|$. Equivalently, for any face $C'$ of $C$ disjoint from $\Unb(C)$, the images of the elements of $\Unb(C)$ are linearly independent in $\R^n/\Span(C')$. 

We say that $\Newt(f)$ has \emph{good projection} if all cones in $\Delta$ have good projection. 
Let $M \in \Z_{>0}$.
Then $\Newt(f)$ has \emph{$M$-good projection} if every maximal cone $C$ in $\Delta$ such that $M$ divides $\rho_C$ has good projection.
\end{definition}

Observe that $\Delta$ has $M$-good projection if and only if all cones in $\Delta_M$ have good projection. 
Clearly, $\Newt(f)$ has good projection if and only if $\Newt(f)$ has $M$-good projection for $M = 1$ if and only if $\Newt(f)$ has $M$-good projection for all $M \in \Z_{>0}$. 
If all cones $C \in \Delta$ with $\Unb(C) \not= \emptyset$ are simplicial, then  $\Newt(f)$ has good projection.
For example, when $\Newt(f)$ is simplicial or convenient, then $\Newt(f)$ has good projection.
Also, if $|\Unb(C)| \le 1$ for all maximal cones $C$ in $\Delta$, then $\Newt(f)$ has good projection.

\begin{lemma}\label{l:allgood}
Let $\alpha \in \Q$. 
If $\Newt(f)$ has $D(\alpha)$-good projection,
then $C_F$ has good projection for all 
 $F \in \Contrib(\alpha)$. 
\end{lemma}
\begin{proof}
 By definition, $\alpha = i/\rho_{C_F}$ for some $i \in \Z$, and hence $D(\alpha)$ divides $\rho_{C_F}$. Since $\Newt(f)$ has $D(\alpha)$-good projection, it follows that $C_F$ has good projection. 
\end{proof}

If $\Newt(f)$ has $D(\alpha)$-good projection, then we will be able to apply Theorem~\ref{t:nonnegativeVarchenko} to deduce that $\exp(2 \pi i \alpha)$ is a nearby eigenvalue of monodromy if a certain local $h$-polynomial does not vanish. 
The condition that $\Newt(f)$ has $D(\alpha)$-good projection is inspired in part by a stricter condition in
Saito \cite[Definition~3.12]{Saito19},  that itself  follows ideas from
\cite{TakeuchiTibar16}. 
In the language of our paper, they consider
the condition on $\Newt(f)$ that every maximal cone $C$ in $\Delta$ such that $D(\alpha)$ divides $\rho_C$ satisfies $\Unb(C) = \emptyset$.

We now formulate the condition on $\Newt(f)$ that will allow us to apply 
Corollary~\ref{c:zeroimpliesUB1}.
Let $F$ be a face of $\Newt(f)$ such that $C_F \in \Delta$. 
Let $\overline{F}$ denote the image of $F$ under the projection 
$\mathbb{R}^n \to \mathbb{R}^{n}/\langle \Unb(C_F) \rangle$. 

\begin{definition}\label{d:pseudoUB1}
Let $F$ be a face of $\Newt(f)$ such that $C_F \in \Delta$.
Assume that 
$C_F$ has good projection. 
Then $F$ is \emph{pseudo-$\UB$} if every $(\dim \overline{F})$-dimensional lattice simplex contained in $\overline{F}$ is $\UB$.

\end{definition}

If $\Newt(f)$ is simplicial, then all pseudo-$\UB$ faces are $\UB$. The $B_2$-facets of \cite[Definition~3.9]{ELT} are examples of pseudo-$\UB$ faces which are not $\UB$.

\begin{remark}\label{r:goodprojection}
Suppose that $C$ in $\Delta$ has good projection. Then $\Sigma = \{ C_1 + C_2 : C_1 \subset C, \Unb(C_1) = \emptyset, C_2 \subset \langle \Unb(C) \rangle \}$ is a fan refining $C$. 
Consider $I \subset [n]$ such that $\R^I_{\ge 0} = \langle \Unb(C) \rangle$.
Let 
$J = [n] \smallsetminus I$
and consider the projection $\pr_J : \R^n \to \R^{J}$. 
Then $$\Star_\Sigma(\langle \Unb(C) \rangle) := \{ \pr_J(C') : \langle \Unb(C) \rangle \subset C' \in \Sigma \} = \{ \pr_J(C_1) : C_1 \subset C, \Unb(C_1) = \emptyset  \}$$ is a refinement of $\pr_J(C)$ \cite[Exercise 3.4.8]{CoxLittleSchenck11}.  
Let $F$ be a face of $\Newt(f)$ such that $C = C_F$  has good projection. Then it follows that $F$ is $\UB$ if and only if 
$\overline{F} = \pr_J(F)$ is $\UB$. 
\end{remark}

\begin{lemma}\label{l:UB1impliespseudoUB1}
Let $F$ be a face of  $\Newt(f)$ such that $C_F \in \Delta$ and $C_F$ has good projection. If $F$ is $\UB$, then $F$ is pseudo-$\UB$. 
\end{lemma}
\begin{proof}
Using Remark~\ref{r:goodprojection}, we reduce to the case when $F$ is compact. In that case, suppose $F$ has an apex $A$ with unique base direction $e_\ell^*$, and $\langle e_\ell^*, A \rangle = 1$. Let $G$ be a lattice simplex contained in $F$ with $\dim G = \dim F$. Then $A$ is an apex of $G$ with base direction $e_\ell^*$.  
 Suppose $e_j^*$ is a base direction of $A$ in $G$. 
If $V \neq A$ is a vertex of $F$, then 
$V \in  \Span(F \cap \{ e_\ell^* = 0 \}) = \Span(G \cap \{ e_\ell^* = 0 \}) = \Span(G \cap \{ e_j^* = 0 \})$. 
Hence $e_j^*$ is a base direction of $A$ in $F$, and $j = \ell$. We conclude that $G$ is $\UB$, as desired.
\end{proof}

We now state our strongest result on the existence of eigenvalues of monodromy.

\begin{theorem}\label{thm:eigenvalue}
Suppose $f$ is nondegenerate.
Let $\alpha \in \mathbb{Q}$. 
Assume that $\Newt(f)$ has $D(\alpha)$-good projection.
Then either every face in 
$\Contrib(\alpha)$ is pseudo-$\UB$, or $\exp(2 \pi i \alpha)$ 
is a nearby eigenvalue of monodromy (for reduced cohomology).
\end{theorem}

Before giving the proof, we present some applications and examples. 

\begin{proof}[Proof of Theorem~\ref{t:mainsimplicialeigenvalue}]
Assume that $\Newt(f)$ is simplicial. 
Then $\Newt(f)$ has good projection, and a face $F$ of $\Newt(f)$ is  pseudo-$\UB$ if and only if it is $\UB$.
The result now follows from Theorem~\ref{thm:eigenvalue}.
\end{proof}

\begin{theorem}\label{thm:dilate}
Let $f$ be a nondegenerate polynomial with $\Newt(f) = kP$ for some $k \ge 2$ and some Newton polyhedron $P$. If $\Newt(f)$ has good projection, then 
every candidate eigenvalue is a nearby eigenvalue of monodromy. 
\end{theorem}
\begin{proof}
Note that none of the vertices of $kP$ have any coordinate equal to one, so no face of $k P$ is pseudo-$\UB$. The result follows from Theorem~\ref{thm:eigenvalue}. 
\end{proof}

\begin{example}\label{e:convenient}
Suppose that $\Newt(f)$ is convenient. Then $\Newt(f)$ has good projection.
 In this case, Theorem~\ref{thm:eigenvalue} states that  either every face in 
$\Contrib(\alpha)$ is pseudo-$\UB$, or $\exp(2 \pi i \alpha)$ 
is a nearby eigenvalue of monodromy (for reduced cohomology).
\end{example}

\begin{example}\label{ex:ELTprop3}
In Example~\ref{ex:ELTprop} and Example~\ref{ex:ELTprop2}, consider the facet $F$ with candidate pole $\alpha = -6/5$. 
We have $\Contrib(\alpha) = \{ F \}$ and $F$ is not $\UB$, although $\overline{F}$ is $\UB$. 
By Remark~\ref{r:goodprojection}, 
$C_F$ does not have good projection. In particular, $F$ is not pseudo-$\UB$. 
By Lemma~\ref{l:allgood}, $\Newt(f)$ does not have $D(\alpha)$-good projection, so Theorem~\ref{thm:eigenvalue} does not apply.

\end{example}

\begin{example}\label{e:ELT2}
Consider the set up of Example~\ref{e:ELT} with $F$ the bounded facet with candidate pole $\alpha = -1/3$. 
We have $\Contrib(\alpha) = \{ F \}$, and $F$ is not $\UB$.
Consider the unbounded facet $G$ of $\Newt(f)$ defined by $\psi_G = \frac{1}{78}(4 e_2^* + 13 e_3^*) = 1$. Then $\ver(C_G) = \{ w_1, w_3, w_4 \}$ and $\Unb(C_G) = \{ e_1, e_4 \}$. In particular, $C_G$ does not have good projection. Since $D(\alpha) = 3$ divides $\rho_{C_G} = 78$, we conclude that $\Newt(f)$ does not have $D(\alpha)$-good projection, so Theorem~\ref{thm:eigenvalue} does not apply.
\end{example}

\begin{proof}[Proof of Theorem~\ref{thm:eigenvalue}]
Let $F$ be a face in $\Contrib(\alpha)$. Suppose there exists a $(\dim \overline{F})$-dimensional lattice simplex $G$ contained in $\overline{F}$ that is not $\UB$. We need to show that $\exp(2 \pi i \alpha)$ 
is a nearby eigenvalue of monodromy (for reduced cohomology).

First assume that $F$ is compact. 
Our first goal is to construct an appropriate simplicial fan $\Delta'$ that refines $\Delta$ and contains $C_G$ as a cone.
Let $\{ i_1, \ldots, i_s \} = \{ i \in [n] : \A^{\{ i \}} \subset X_f \}$. 
Consider positive integers 
$0 \ll m_{i_1} \ll \cdots \ll  m_{i_s}$, and let 
$\hat{f} := f + \sum_{j = 1}^s x_{i_j}^{m_{i_j}}$ with corresponding Newton polyhedron $\Newt(\hat{f})$ and fan over the faces  
$\widehat{\Delta}$. Then $\widehat{\Delta}$ refines $\Delta$ and has the same rays as $\Delta$. If $C \in \Delta$ has good projection, then every cone in $\widehat{\Delta}|_{C}$ is a sum of a cone $C_1$ with $\Unb(C_1) = \emptyset$ and a cone spanned by a subset of $\Unb(C)$.  
Using, for example, a pulling triangulation  \cite[Section 4.3.2]{Loera10}, one can construct a simplicial refinement
 $\Delta'$ of $\widehat{\Delta}$ such that $C_G$ is a cone in $\Delta'$, and the rays of $\Delta'$ are the union of the rays of $\Delta$ and the rays of $C_G$. 

 Given $C' \in \Delta'$, recall that $\tau(C')$ denotes the smallest face of $\Delta$ containing $C'$, and $\Unb(C') \subset \Unb(\tau(C'))$. 
If $C \in \Delta$ has good projection, then every cone in 
$\Delta'|_{C}$ is a sum of a cone $C_1 \in \Delta'|_{C}$ with $\Unb(C_1) = \emptyset$ and a cone spanned by a subset of $\Unb(C)$. 
In particular, if $C' \in \Delta'|_{C}$, then $\Unb(C') = \Unb(\tau(C'))$.

By Corollary~\ref{c:nonvanishingeigenvalue}  and Corollary~\ref{c:zeroimpliesUB1},
the coefficient of $[\alpha]$ in $\widetilde{E}(\cF_{x_I})$ is nonzero
at a general point $x_I$ of some coordinate subspace $\A^I \subset X_f$.

Now consider the case when $F$ is not necessarily compact. Consider $I \subset [n]$ such that $\R^I_{\ge 0} = \langle \Unb(C_F) \rangle$. Let 
$J = [n] \smallsetminus I$
and consider the projection $\pr_J : \R^n \to \R^{J}$, and  
$\Newt(f)_J := \pr_{J}(\Newt(f)) \subset \R^{J}$.
Let $\Delta_J$ be the fan over the faces of $\Newt(f)_J$. The maximal cones of $\Newt(f)_J$ are precisely the cones of the form
$\pr_J(C)$, where $C$ is a maximal cone 
of $\Delta$ such that $\R_{\ge 0}^I \subset \langle \Unb(C) \rangle$. For any such cone $C$, $\rho_C = \rho_{\pr_J(C)}$ and if $C$ has good projection, then $\pr_J(C)$ has good projection. 
We deduce that $\Newt(f)_J$ has $D(\alpha)$-good projection. Also,  $\overline{F} = \pr_J(F)$ has candidate pole $\alpha$. 
 
Let $g$ be a nondegenerate polynomial with Newton polyhedron $\Newt(f)_J$ and Milnor fiber $\widehat{\cF}_{0}$ at the origin. By the compact case above, we deduce that the coefficient of $[\alpha]$ in $\widetilde{E}(\cF_{y_{\hat{I}}})$ is nonzero
at a general point $y_{\hat{I}}$ of some  coordinate subspace $\A^{\hat{I}} \subset X_g$ with  $\hat{I} \subset J$.
Let $I' = I \cup \hat{I}$ and $J' = [n] \smallsetminus I'$. 
Applying \eqref{e:reducecompact} to both $\widetilde{E}(\cF_{x_{I'}})$ and $\widetilde{E}(\cF_{y_{\hat{I}}})$ yields the equality $\widetilde{E}(\cF_{x_{I'}}) =\widetilde{E}(\Newt(f)_{J'} ) =  \widetilde{E}(\cF_{y_{\hat{I}}})$. 
We conclude that the coefficient of $[\alpha]$ in $\widetilde{E}(\cF_{x_{I'}})$ is nonzero. 
\end{proof}

Finally, as a corollary of the proof above, we may extend the result of Budur and van der Veer \cite[Theorem 1.10]{Budur} on dilates of Newton polyhedron by removing the convenient hypothesis.

\begin{proposition}\label{p:Budurextend}
Fix a Newton polyhedron $P$. 
Let $f$ be a nondegenerate polynomial with $\Newt(f) = kP$ for some $k \in \Z_{>0}$ chosen sufficiently large. Then 
every candidate eigenvalue is a nearby eigenvalue of monodromy.
\end{proposition}
\begin{proof}
Let $F$ be a facet of $P$ and let $\alpha$ be the corresponding candidate pole. Then $\alpha/k$ is the corresponding candidate pole associated to the facet $kF$ of $kP$.  After possibly replacing $F$ by $\overline{F}$, we reduce to the case when $F$ is compact. Then the proof of \cite[Theorem 1.10]{Budur} applies. Explicitly, assume that $F$ is compact and
let $c_{k} \in \Z$ denote the coefficient of $[\alpha/k]$ in 
$\widetilde{E}(kP)$. Then Varchenko's Theorem (see \eqref{e:varchenko} above) implies that 
\begin{equation}\label{e:coefflimit}
\lim_{k \to \infty} c_k/k^{n - 1} = (-1)^n \sum_{F'} \Vol(F'),
\end{equation}
where $F'$ varies over all facets of $P$ such that 
$D(\alpha)$ divides $\rho_{F'}$. 
Since $F' = F$ appears in the sum on the right-hand side of \eqref{e:coefflimit}, we deduce that the left-hand side of \eqref{e:coefflimit} is nonzero, and the result follows.
\end{proof}

\section{A necessary condition for the vanishing of the local $h$-polynomial} \label{sec:vanishing}

\subsection{Overview}

In this section, we prove a necessary condition for the vanishing of the local $h$-polynomial of a geometric triangulation of a simplex. The section is self-contained and combinatorial in nature. As such, the notation used is independent from the rest of the paper.

Let $\sigma\colon \mathcal{S} \to 2^{[n]}$ be a geometric triangulation of a simplex. A face $G$ of  $\mathcal{S}$ is \emph{interior} if $\sigma(G) = [n]$. 
Let $E$ be a face of $\mathcal{S}$, and let $F \in \lk_{\mathcal{S}}(E)$ be a face.  
Then $F$ is a  \emph{pyramid} with apex $A \in F$ if $F \sqcup E$ is interior and $(F \sqcup E) \smallsetminus A$ is not interior.
Let $$\mathcal{A}_F := \{ A \in F : F \mbox{ is a pyramid with apex } A \}, \mbox{ \ and \ } V_A = V_A(F) := [n] \smallsetminus \sigma( (F \sqcup E) \smallsetminus A)$$
for $A \in \mathcal{A}_F$.
In what follows, we identify simplices with their sets of vertices.

\begin{definition}\label{d:Upyramid}
We say that $F$ is a \emph{$U$-pyramid} if $|V_A| = 1$ for some $A \in \mathcal{A}_F$.

\end{definition}

\begin{definition}\label{d:fullpartition}
A \emph{full partition} of $F$ is a decomposition
\[
F = F_1 \sqcup F_2 \sqcup \mathcal{A}_F
\]
such that $F_1 \sqcup \mathcal{A}_F \sqcup E$ is interior and $\sigma( F_2 \sqcup E) = [n] \smallsetminus \bigcup_{A \in \mathcal{A}_F} V_A$. 
\end{definition}

Recall that $\ell(\mathcal{S}, E; t)$ denotes the corresponding local $h$-polynomial. See Definition~\ref{d:localdef}. 
Our goal is to prove the following theorem.

\begin{theorem}\label{thm:nonvanish}
Let $\sigma \colon \mathcal{S} \to 2^{[n]}$ be a 
geometric
triangulation of a simplex, and fix a face $E \in \mathcal{S}$.  Let $F \in \lk_\mathcal{S}(E)$ be a face that admits a full partition 
$F = F_1 \sqcup F_2 \sqcup \mathcal{A}_F$. If the coefficient of $t^{|F_1| + |\mathcal{A}_F|}$ in $\ell(\mathcal{S}, E; t)$ is zero, then $F$ is a $U$-pyramid. 
\end{theorem}

Our strategy is as follows: assume that $F$ admits a full partition and is not a $U$-pyramid. 
We argue that the nonvanishing of the local $h$-polynomial is implied by the nonvanishing of a specific element in the cohomology of a (possibly non-compact) toric variety. We verify this nonvanishing by reducing to a result in \cite{LPS2}.

\subsection{The commutative algebra of local $h$-polynomials}\label{ss:commutative}

Let $\Delta$ be a rational simplicial fan in $\R^n$ with support $\R^n_{\ge 0}$. For each ray of $\Delta$, choose a rational, nonzero point $v$. Consider the unique piecewise $\Q$-linear function $\psi \colon \R^n_{\ge 0} \to \R$ defined by $\psi(v) = 1$ for all such $v$, and let $\mathcal{S} = \{ x \in \R^n_{\ge 0} : \psi(x) = 1 \}$. Then $\mathcal{S}$ is a simplicial complex with vertices $\{ v \}$, and $\mathcal{S}$ induces a geometric triangulation $\sigma \colon \mathcal{S} \to 2^{[n]}$ of a simplex by projecting onto a transverse hyperplane. The combinatorial type of this triangulation is independent of both the choice of $\{ v \}$ and the choice of transverse hyperplane. 
Explicitly, if $F$ is a face of $\mathcal{S}$,
then 
$\R^{\sigma(F)}$ is the smallest coordinate hyperplane containing $F$. 
Conversely, given a geometric triangulation of a simplex, we may deform the vertices without changing the combinatorial type to assume that the triangulation is rational, and then the triangulation is realized by some such $\mathcal{S}$.

If $F$ is a 
face of $\mathcal{S}$, 
let $C_F$ denote the cone over $F$. 
For example, when $F = \emptyset$, then $C_F = \{ 0 \}$. Then $\Delta = \{ C_F : F \in \mathcal{S} \}$.  
Fix a face $E$ of $\mathcal{S}$. 
Then the collection of cones $\Delta_E$ given by the images of 
$\{ C_F : F \in \lk_\mathcal{S}(E) \}$ in $\R^n/\spn(C_E)$ forms a fan. For example, $\Delta_\emptyset = \Delta$, and 
$\Delta_E$ is complete if and only if $E$ is an interior face of $\mathcal{S}$. Consider the standard lattice $\mathbb{Z}^n \subset \mathbb{R}^n$, and let $X_E$ denote the toric variety associated to $\Delta_E$. The torus orbits in $X_E$ are in 
inclusion-reversing bijection with the faces in $\lk_\mathcal{S}(E)$. 
If $E \subset E'$, then $X_{E'}$ is the closure in $X_E$ of the 
torus orbit corresponding to the face $E' \smallsetminus E$ of  $\lk_\mathcal{S}(E)$.

Given a finite simplicial complex $\T$,
let $\mathbb{Q}[\T]$ denote the \emph{face ring} of $\T$ over $\Q$, i.e., the quotient of the polynomial ring over $\mathbb{Q}$ with variables corresponding to the vertices of $\T$ by the ideal generated by monomials corresponding to non-faces. For a face $F \in \T$, let $x^F \in \mathbb{Q}[\T]$ denote the product of the variables corresponding to the vertices of $F$. Note that $\mathbb{Q}[\mathcal{T}]$ is graded by degree. We write $\vert G \vert$ for the number of vertices in a face $G$.  In particular, $x^F$ is a squarefree monomial of degree $\vert F \vert$.

A \emph{linear system of parameters} (l.s.o.p.) for a finitely generated graded $\mathbb{Q}$-algebra $R$ of Krull dimension $d$ is a sequence of elements $\theta_1, \ldots, \theta_d$ in $R_1$ such that $R/(\theta_1, \ldots, \theta_d)$ is a finite-dimensional $\mathbb{Q}$-vector space. 
If $\T$ has dimension $d - 1$, then 
$\mathbb{Q}[\T]$ has Krull dimension $d$.

Let $c = n - \vert E \vert$. Note that $c$ is the Krull dimension of $\mathbb{Q}[\lk_{\mathcal{S}}(E)]$. 
The \emph{support}  of an element $\theta = \sum a_v x^v \in \Q[\lk_{\mathcal{S}}(E)]_1$ 
is $\supp(\theta) := \{v : a_v \not= 0\}$. 
A linear system of parameters $\theta_1, \dotsc, \theta_{c}$ for $\mathbb{Q}[\lk_{\mathcal{S}}(E)]$ is \textit{special}, as defined in \cite{Stanley92, Athanasiadis12b}, if, for each vertex 
$v \in [n] \smallsetminus \sigma(E)$, 
there is an element $\theta_v$ of the l.s.o.p. such that $\supp(\theta_v)$ consists of vertices $w$ in $\lk_{\mathcal{S}}(E)$  such that 
$v \in \sigma(w)$, 
and such that $\theta_v \not= \theta_{v'}$ for $v \not= v'$.

\begin{proposition}{ \cite{Athanasiadis12b,Athanasiadis12}, see also \cite[Proof of Theorem 1.2]{LPS2}}
Let $I$ be the ideal in $\mathbb{Q}[\lk_{\mathcal{S}}(E)]$ generated by  
$ \{ x^F : F \sqcup E \mbox{ is interior}\,\}.$
Let $L(\mathcal{S}, E)$ be the image of $I$ in $\mathbb{Q}[\lk_{\mathcal{S}}(E)]/(\theta_1, \dotsc, \theta_{c})$, where $\theta_1, \dotsc, \theta_{c}$ is a special l.s.o.p. Then the Hilbert series of $L(\mathcal{S}, E)$ is $\ell(\mathcal{S}, E; t)$.
\end{proposition}

\noindent We call $L(\mathcal{S}, E)$ the \emph{local face module}. Note that the local face module depends on the choice of a special l.s.o.p. 
In this paper, we will consider a particular special l.s.o.p. that is defined in terms of $\Delta$.

Below we view elements of $(\mathbb{Q}^n/\spn(E))^* \hookrightarrow (\mathbb{Q}^n)^*$ as $\Q$-linear functions vanishing on $\spn(E)$. 
For 
$u \in (\mathbb{Q}^n/\spn(E))^* $, let $\theta_{u} = \sum_{v \in \lk_{\mathcal{S}}(E)} \langle u , v \rangle x^v \in \mathbb{Q}[\lk_{\mathcal{S}}(E)]$. Consider the ideal $J_E = (\theta_u : u \in (\mathbb{Q}^n/\spn(E))^*)$ in $\mathbb{Q}[\lk_{\mathcal{S}}(E)]$. 
Note that $J_E$ is generated by a special l.s.o.p., obtained by extending 
$\{e_i^* : i \in [n] \smallsetminus \sigma(E)\}$ to a basis for $(\mathbb{Q}^n/\spn(E))^*$.

Let $H^*(E) = \mathbb{Q}[\lk_{\mathcal{S}}(E)]/J_E$.
Then $H^*(E)$  is isomorphic to the rational cohomology ring $H^*(X_E, \Q)$ of $X_E$. 
The ideal in $H^*(E)$ generated by $\{x^F : F \in \lk_{\mathcal{S}}(E), F \sqcup E \text{ interior}\}$ is $L(\mathcal{S}, E)$. 
We will show the nonvanishing of $\ell(\mathcal{S}, E; t)$ by showing that a certain element of $L(\mathcal{S}, E)$ 
is nonzero. To achieve this, we require three constructions.

First, if $E \subset E'$ is an inclusion of faces in $\mathcal{S}$, then there is a graded $\Q$-algebra homomorphism  $\iota^* = \iota_{E, E'}^* \colon H^*(E) \to H^*(E')$ corresponding to the pullback map on cohomology.  The \emph{closed star} $\Star(E' \smallsetminus E)$ of 
$E' \smallsetminus E$ 
is the subcomplex of $\lk_{\mathcal{S}}(E)$ that consists of faces $H$ such that $H \cup (E' \smallsetminus E)$ is a face of $\lk_{\mathcal{S}}(E)$.
Then $\iota^*$ 
may be characterized as follows: let $v \in \lk_\mathcal{S}(E)$. Then 
\begin{enumerate}
\item\label{cond:pullback1} 
$\iota^*(x^v) = 0$ if $v \notin  \Star(E' \smallsetminus E)$, 
\item\label{cond:pullback2} 
$\iota^*(x^v) = x^v$ if $v \in \lk_\mathcal{S}(E')$. 
\end{enumerate}
Note that 
$\Star(E' \smallsetminus E)$ is the join of $\lk_{\mathcal{S}}(E')$ with $E'  \smallsetminus E$. If 
$v \in E'  \smallsetminus E$, then there exists a linear form $u_v$ in $(\mathbb{Q}^n/\spn(E))^*$ that takes value $1$ on $v$ and vanishes on all other $v' \in E'$, and the above properties imply that
$\iota^*(x^v) = -\sum_{v' \in \lk_{\mathcal{S}}(E')} \langle u_v , v' \rangle x^{v'}$. 

\medskip

Second, let $j_* = j_{E',E,*} \colon  H^*(E') \to H^*(E)$ be defined by $j_*(x^G) = x^G x^{E' \smallsetminus E}$ for all $G \in \lk_{\mathcal{S}}(E')$, corresponding to the Gysin pushforward map on cohomology. It then follows from the characterization of $\iota^*$ via \eqref{cond:pullback1} and \eqref{cond:pullback2}, that $j_* \circ \iota^* \colon H^*(E) \to H^*(E)$ is multiplication by $x^{E' \smallsetminus E}$. 

\medskip

Finally, we will make use of a nondegenerate bilinear form $B_E \colon L(\mathcal{S}, E) \times L(\mathcal{S}, E) \to \mathbb{Q}$. In the case $E = \emptyset$, this bilinear form was constructed in \cite[Corollary 4.19]{Stanley92} using a description of the canonical module of the face ring of a triangulation of a disk. We give an equivalent geometric description below.

Let $\hat{\Delta}$ be the complete fan obtained by adding a ray $\rho_{c}$ spanned by $(-1, \dotsc, -1)$ to $\Delta$, and adding the cone generated by $\rho_{c}$ and $C$ for each cone $C$ of $\Delta$ which is contained in the boundary of $\mathbb{R}^n_{\ge 0}$. Let $\hat{\mathcal{S}}$ be the simplicial complex of faces of $\hat{\Delta}$. 

For each face $E$ of $\Delta$, let $X_{\hat{E}}$ be the toric variety whose fan is the image of the cones of $\hat{\Delta}$ that contain $C_E$  in $\mathbb{R}^n/\spn(C_E)$. Note that $X_{\hat{E}}$ is a proper simplicial toric variety, and so its cohomology ring $H^*(X_{\hat{E}})$ is equipped with an isomorphism $\deg \colon H^{n - |E|}(X_{\hat{E}}) \to \mathbb{Q}$ such that the \emph{Poincar\'{e} pairing}, i.e., the induced bilinear form $H^{*}(X_{\hat{E}}) \times H^{n - |E| - *}(X_{\hat{E}}) \to \mathbb{Q}$, is nondegenerate.

The cohomology ring $H^*(X_{\hat{E}})$ has a similar presentation to that of $H^*(X_{E})$: it is the quotient of $\mathbb{Q}[\operatorname{lk}_{\hat{S}}(E)]$ by the linear forms corresponding to linear functions on $\mathbb{R}^n/\spn(C_E)$. In particular, $H^*(X_{\hat{E}})/(x^c) = H^*(X_E)$. 
This implies that the restriction map $H^*(X_{\hat{E}}) \to H^*(X_E)$ induces an isomorphism
$$\frac{(x^F : F \in \lk_{\mathcal{S}}(E), \,  F \sqcup E \text{ interior})}{(x^F : F \in \lk_{\mathcal{S}}(E), \, F \sqcup E \text{ interior}) \cap (x^{c})} \stackrel{\sim}{\to} L(\mathcal{S}, E).$$

\begin{lemma}
The restriction of the Poincar\'{e} pairing to the ideal 
$(x^F : F \in \lk_{\mathcal{S}}(E), \, F \sqcup E \text{ interior})$
descends to a nondegenerate bilinear form on $L(\mathcal{S}, E)$. 
\end{lemma}

\begin{proof}
We first show that $(x^F : F \in \lk_{\mathcal{S}}(E), \, F \sqcup E \text{ interior}) = \operatorname{ann}(x^c).$
For each face $F$ such that $F \sqcup E$ is interior, $x^{F} \cdot x^{c} = 0$ in $\mathbb{Q}[\operatorname{lk}_{\hat{S}}(E)]$, so this also holds in $H^{*}(X_{\hat{E}})$. This gives that $(x^F : F \in \lk_{\mathcal{S}}(E), \, F \sqcup E \text{ interior}) \subseteq \operatorname{ann}(x^c).$ For the other inclusion, note that
$H^*(X_{\hat{E}})/(x^F : F \in \lk_{\mathcal{S}}(E), \, F \sqcup E \text{ interior})$ is the cohomology of the toric divisor on $X_{\hat{E}}$ corresponding to the ray $\rho_c$. In particular, Poincar\'{e} duality holds for $H^*(X_{\hat{E}})/(x^F : F \in \lk_{\mathcal{S}}(E), \, F \sqcup E \text{ interior})$. Therefore, if there is some nonzero
$$u \in \ker\left (H^*(X_{\hat{E}})/(x^F : F \in \lk_{\mathcal{S}}(E), \, F \sqcup E \text{ interior}) \to H^*(X_{\hat{E}})/\operatorname{ann}(x^c)\right ),$$
then there is some $v$ such that $u\cdot v$ is nonzero and lies in the degree $n - |E| - 1$ part of $H^*(X_{\hat{E}})/(x^F : F \in \lk_{\mathcal{S}}(E), \, F \sqcup E \text{ interior})$, which is $1$-dimensional. This contradicts that $H^*(X_{\hat{E}})/\operatorname{ann}(x^c)$ is nonzero in degree $n - |E| - 1$, spanned by the image of any $w \in H^{n - |E| - 1}(X_{\hat{E}})$ such that $x^c \cdot w \not= 0$.

For any ideal $I$ in $H^*(X_{\hat{E}})$, the kernel of the restriction of the Poincar\'{e} pairing to $I$ is $\operatorname{ann}(I) \cap I$. Applying this with $I = \operatorname{ann}(x^c)$ and using that $\operatorname{ann}(\operatorname{ann}(x^c)) = (x^c)$ gives the result. 
\end{proof}

Let $B_E \colon L(\mathcal{S}, E) \times L(\mathcal{S}, E) \to \mathbb{Q}$ be the induced nondegenerate bilinear form.
If $E \subset E'$ is an inclusion of faces, then the pullback map $\iota^*$ and Gysin pushforward $j_*$ are adjoints: for any $u \in L(\mathcal{S}, E)$ and $v \in L(\mathcal{S}, E')$, we have
$$B_{E'}(\iota^* u, v) = B_E(u, j_*v).$$
This follows from the construction of the bilinear forms and the corresponding fact on the complete toric varieties $X_{\hat{E}}$ and $X_{\hat{E}'}$.

\subsection{Proof of nonvanishing}\label{ss:nonvanishing}

We will prove Theorem~\ref{thm:nonvanish} by reducing to a special case which was proved in \cite{LPS2}. Because the Hilbert series of $L(\mathcal{S}, E)$ is $\ell(\mathcal{S}, E; t)$, the following result is a strengthening of the case of Theorem~\ref{thm:nonvanish} when $F_1 = F_2 = \emptyset$. 

\begin{proposition}\cite[Remark 3.2]{LPS2}\label{prop:specialnonvanish}
Let $F \in \operatorname{lk}_{\mathcal{S}}(E)$ be a face such that $F \sqcup E$ is interior and $\mathcal{A}_F = F$. If $x^{\mathcal{A}_F} = 0$ in $L(\mathcal{S}, E)$, then $F$ is a $U$-pyramid. 
\end{proposition}

Proposition~\ref{prop:specialnonvanish} is proved by computing a resolution of $L(\mathcal{S}, E)$. This gives a description of the kernel of the map $(x^F : F \in \lk_{\mathcal{S}}(E), \, F \sqcup E \text{ interior}) \to L(\mathcal{S}, E)$. From this description, it is clear that if the kernel contains $x^{\mathcal{A}_F}$, then $F$ must be a $U$-pyramid.

\begin{proof}[Proof of Theorem~\ref{thm:nonvanish}]
Suppose we have a full partition $F = F_1 \sqcup F_2 \sqcup \mathcal{A}_F$ 
and $F$ is not a $U$-pyramid.
We claim that $x^{F_1 \sqcup \mathcal{A}_F} \in L(\mathcal{S}, E)$ is nonzero; note that this makes sense because $F_1 \sqcup \mathcal{A}_F \sqcup E$ is interior. 
To see this, it suffices to prove that it is nonzero after restriction to $L(\mathcal{S}, F_2 \sqcup E)$. 
Consider the Gysin map $j_* \colon L(\mathcal{S}, F_1 \sqcup F_2 \sqcup E) \to L(\mathcal{S}, F_2 \sqcup E)$. 
Then $x^{F_1 \sqcup \mathcal{A}_F} = j_*(x^{\mathcal{A}_F}) \in L(\mathcal{S}, F_2 \sqcup E)$. 
By Proposition~\ref{prop:specialnonvanish}, $x^{\mathcal{A}_F}$ is nonzero in $L(\mathcal{S}, F_1 \sqcup F_2 \sqcup E)$.

The Gysin map $j_*$ is adjoint to the pullback map $\iota^* \colon L(\mathcal{S}, F_2 \sqcup E) \to L(\mathcal{S}, F_1 \sqcup F_2 \sqcup E)$. Because $\sigma(F_2 \sqcup E) = \sigma(F_1 \sqcup F_2 \sqcup E)$, this pullback map is surjective (cf. the proof of \cite[Theorem~1.6]{LPS2}). Explicitly, if $G \in \lk_{\mathcal{S}}(F_1 \sqcup F_2 \sqcup E)$ and $G \sqcup F_1 \sqcup F_2 \sqcup E$ is interior, then $G \sqcup F_2 \sqcup E$ is interior and  $\iota^*(x^G) = x^G$. This implies that $j_*$
is injective, giving the desired nonvanishing. 
\end{proof}

\begin{remark}
In \cite{LPS2}, Proposition~\ref{prop:specialnonvanish} is proved in the setting of quasi-geometric homology triangulations of simplices, a much more general class of triangulations of simplices. The proof of Theorem~\ref{thm:nonvanish} can be adapted to work in this setting. 
\end{remark}

\begin{remark}
In a previous version of this article \cite{LPSarxiv}, Theorem~\ref{thm:nonvanish} is proved in a different way. Using the language above, the proof there constructs a face $G$ with $G \sqcup E$ interior and checks, via a lengthy computation, that $B_E(x^{F_1 \sqcup \mathcal{A}_F}, x^G) \not=0$. The proof shows that, up to a simple normalization factor and an explicit sign, $B_E(x^{F_1 \sqcup \mathcal{A}_F}, x^G)$ is equal to $1$. 
\end{remark}

\section{The local formal zeta function and its candidate poles} \label{sec:poles}

\subsection{Overview}
In this section, we introduce the local formal zeta function (Definition~\ref{d:formallocal}) and develop its fundamental properties. In Section~\ref{ss:formula}, we recall the formula for $Z_{\mot}(T)$ in \cite[Theorem 8.3.5]{BultotNicaise20}. In Section~\ref{ss:formal}, we define the local formal zeta function and discuss its candidate poles. In Section~\ref{ss:simplifyformal}, we prove a relation that the local formal zeta function satisfies that will be a crucial tool in the proof of Theorem~\ref{thm:nopolesimplicial}.

We first introduce some notation for use in this section and in Section~\ref{sec:fakepoles}. We use capital letters to denote elements of the vector space containing $\Newt(f)$ and use lowercase letters to denote elements of the dual space.
Let $\Gamma$ 
be the union of  the proper interior faces of $\Newt(f)$ and their subfaces. That is, $\Gamma$ is the union of faces $F$ of $\Newt(f)$ that are visible from the origin in the sense that for every $W \in F$, the intersection of $\Newt(f)$ with the interval from the origin to $W$ equals $\{ W \}$.
Let $\Sigma = \Sigma_f$ be the dual fan of $\Newt(f)$. For each face $F$ of $\Newt(f)$, let $\sigma_F$  be the cone of $\Sigma$ dual to $F$.

\subsection{Formula for the local motivic zeta function}\label{ss:formula}
We now recall the formula of Bultot and Nicaise for the local motivic zeta function of a 
nondegenerate polynomial $f$.

Consider a nonempty compact face $K$ of $\Gamma$.
Following \cite{BultotNicaise20}, 
we associate two 
classes in $\mathcal{M}^{\hat{\mu}}$ to $K$.
For $i \in \{ 0,1 \}$, let $Y_K(i)$ 
be the closed subscheme  of 
$\Spec \k[\Span(K) \cap \mathbb{Z}^n]$ cut out by $f_{K} = i$. 
When $i = 0$, we endow $Y_{K}(0)$ 
with the trivial $\hat{\mu}$-action and obtain 
a class $[Y_{K}(0)] \in \mathcal{M}^{\hat{\mu}}$.
We define a $\hat{\mu}$-action on $Y_{K}(1)$ 
as follows. 
Let $\rho_K$ be the lattice distance of $K$ to the origin, and let $w = w_K := \rho_K \psi_K \in 
\Hom(\Span(K) \cap \mathbb{Z}^n,\Z)$.
Then $w$ determines a cocharacter $\Spec \k[\Z] \to \Spec \k[\Span(K) \cap \mathbb{Z}^n]$, which we can restrict via 
$\Spec \k[T]/(T^\rho - 1) \to \Spec \k[\Z]$ to determine a $\mu_\rho$-action on $\Spec \k[\Span(K) \cap \mathbb{Z}^n]$. This induces an action of $\mu_{\rho}$ on $Y_K(1)$. Explicitly, choose a basis for $\Span(K) \cap \mathbb{Z}^n$ and write $w = (w_1, \ldots, w_r)$ and $f = \sum_{a \in \Z^r} \lambda_a x^a$. Then for each
$a = (a_1,\ldots,a_r)$ with $\lambda_a \neq 0$,  $\sum_{i = 1}^r a_i w_i$ is divisible by $\rho$, and the action is 
$$\zeta \cdot (x_1, \ldots, x_r) = (\zeta^{w_1} x_1, \ldots, \zeta^{w_r} x_r).$$
This gives a class $[Y_{K}(1)]$ in $\mathcal{M}^{\hat{\mu}}$. 
When $K$ is the empty compact face of $\Gamma$, $\Span(K) \cap \mathbb{Z}^n = \{ 0 \}$, and we let $Y_K(0)$ be the point $\Spec \k[\Span(K) \cap \mathbb{Z}^n]$ and let $Y_K(1) = \emptyset$. Then $[Y_K(0)] = 1$ and $[Y_K(1)] = 0$.

\begin{remark}
The above construction differs slightly from that in
\cite{BultotNicaise20}. Explicitly,  
for $i \in \{ 0,1 \}$,  \cite{BultotNicaise20} let $X_K(i)$ 
be the closed subscheme  of 
$\Spec \k[\Z^n]$ cut out by $f_{K} = i$. Consider $X_K(0)$ with the trivial $\hat{\mu}$-action. Let 
$w$ be any linear function in $\mathbb{Z}^n$ that restricts to $\rho \psi_K$, and, as above, consider the corresponding $\mu_\rho$-action on $X_K(1) \subset \Spec \k[\Z^n]$. 
They consider the classes $[X_{K}(i)]$ in $\mathcal{M}^{\hat{\mu}}$. It follows from \cite[Proposition~7.1.1]{BultotNicaise20} that 
$[X_K(i)] = [Y_K(i)](\L - 1)^{n - 1 - \dim K}$.  
\end{remark}

The following lemma will be important in the proof of 
Theorem~\ref{thm:nopole}.

\begin{lemma}\label{l:simplifiedrelations}
Let $G \subset F$ be an inclusion of compact faces of $\Gamma$. Suppose there exists a vertex $A$ of $F$ such that $F = \Conv{G,A}$ and $\Span(F) \cap \mathbb{Z}^n = \Span(G) \cap \mathbb{Z}^n + \Z \cdot A$. Then for $i \in \{0,1\}$, 
$[Y_G(i)] + [Y_F(i)] = (\L - 1)^{\dim F} \in \mathcal{M}^{\hat{\mu}}$. 
\end{lemma}
\begin{proof}
Let $r = \dim F$. 
Let 
$\rho_G$ and $\rho_F$ be the smallest positive integers such that $w_G = \rho_G \psi_G$  and $w_F = \rho_F \psi_F$
lie in  $\Hom(\Span(G) \cap \mathbb{Z}^n,\Z)$ and $\Hom(\Span(F) \cap \mathbb{Z}^n,\Z)$ respectively.

Then, we may choose coordinates such that
$Y_F(i)$ is defined by $\{ f_F(x_0,\ldots,x_r) = i \}$ in $\Spec \k[\Span(F) \cap \mathbb{Z}^n]$, 
$Y_G(i)$ is defined by $\{ f_G(x_1,\ldots,x_r) = i \}$ in $\Spec \k[\Span(G) \cap \mathbb{Z}^n]$, 
and 
$f_F(x_0,\ldots,x_r) = x_0 + f_G(x_1,\ldots,x_r)$. 
Also, we may set $\rho = \rho_F = \rho_G$ and write 
$w_G = (w_1,\ldots,w_r)$ and $w_F = (1,w_1,\ldots,w_r)$. As above, $w_G$ and $w_F$ induce $\mu_\rho$-actions on $\Spec \k[\Span(G) \cap \mathbb{Z}^n]$ 
and 
$\Spec \k[\Span(F) \cap \mathbb{Z}^n]$
respectively. Consider the $\mu_\rho$-equivariant map
\[
\phi\colon
\Spec \k[\Span(G) \cap \mathbb{Z}^n]
\smallsetminus Y_G(i) \to Y_F(i),
\]
\[
\phi(x_1,\ldots,x_r) = (i - f_G(x_1,\ldots,x_r), x_1,\ldots,x_r).
\]
Then $\phi$ is an isomorphism, with inverse $\phi^{-1}(x_0,x_1,\ldots,x_r) = (x_1,\ldots,x_r)$. 
By \cite[Lemma~7.1.1]{BultotNicaise20}, the class of any $r$-dimensional torus in $\mathcal{M}^{\hat{\mu}}$ is $(\L - 1)^r$, and the result follows.
\end{proof}

\begin{example}\label{e:primitivevertex}
Let $A$ be a primitive vertex of $\Gamma$. 
Then Lemma~\ref{l:simplifiedrelations}, with 
$F = \{ A \}$ and $G = \emptyset$, implies that $[Y_A(0)] = 0$ and $[Y_A(1)] = 1$. 
\end{example}

\begin{example}\label{e:nonemptybase}
Let $F$ be a compact $B_1$-face of $\Gamma$ with nonempty base $G$. Then Lemma~\ref{l:simplifiedrelations} implies that 
$[Y_G(0)] \frac{\L^{-1}T}{1-\L^{-1}T} + [Y_G(1)] + [Y_F(0)] \frac{\L^{-1}T}{1-\L^{-1}T} + [Y_F(1)] = \frac{(\L - 1)^{\dim F}}{1 - \mathbb{L}^{-1}T}$.
\end{example}

We now discuss two results on lattice point enumeration. The first result is standard. 
Let $C$ be a nonzero rational simplicial polyhedral cone in $\R_{\ge 0}^n$ with rays spanned by primitive integer vectors $u_1,\ldots,u_r$. 
Let $\BBox^+(C) = \{ u \in \N^n : u = \sum_{i = 1}^r \lambda_i u_i \text{ for some } 0 < \lambda_i \le 1 \}$.
Then 
\begin{equation}\label{e:infiniteseries}
\sum_{u \in C^\circ \cap  \N^n} x^u = \frac{\sum_{u \in \BBox^+(C)} x^u}{\prod_{i = 1}^r (1 - x^{u_i})} \in \Z\llbracket x_1,\ldots,x_n\rrbracket .
\end{equation}

We also need the following lemma.

\begin{lemma}\label{lem:powerseries}
Let $C$ be a rational polyhedral cone of dimension $d$ contained in $\mathbb{R}^n_{\ge 0}$, and let $Y$ be 
a $\Z$-linear function
that 
takes nonnegative values on $C$ and is not identically zero on $C$.
Let $u_1, \dotsc, u_r$ be the primitive generators of the rays of $C$. Let $I = \{i \in [r] : \langle u_i, Y \rangle \not= 0\}$. Assume that $\langle u_j, \mathbf{1} \rangle = 1$ for $j \not \in I$. Then
$$(L - 1)^{d - 1} \sum_{u \in C^{\circ} \cap \mathbb{N}^{n}} L^{-\langle u, \mathbf{1} \rangle}T^{\langle u, Y \rangle}$$
lies in the subring 
$$\mathbb{Z}[L, L^{-1}, T]\Big [\frac{1}{1 - L^{-\langle u_i, \mathbf{1} \rangle} T^{\langle u_i, Y\rangle }} \Big]_{i \in I} \subset \mathbb{Z}[L]\llbracket  L^{-1}, T\rrbracket .$$
\end{lemma}

\begin{proof}
This is essentially \cite[Lemma 5.1.1]{BultotNicaise20}; the point is that we may reduce to the case when $C$ is simplicial and then apply 
\eqref{e:infiniteseries}.
The $1/(1 - L^{-1})$ terms that arise from $j \not \in  I$ are cancelled by the $(L-1)^{d-1}$ factor. 
\end{proof}

Define a piecewise linear function $N$ on $\Sigma$ by 
$$N(u) = \min \{ \langle u, W \rangle : W \in \Newt(f) \}.$$

\begin{remark}\label{r:Nzerorays}
If $u$ is a primitive generator of a ray in the dual fan,  corresponding to a facet $F$ of $\Newt(f)$, then $N(u)$ is the lattice distance of $F$ to the origin. 
If $N(u) = 0$, then $u = e_i^*$ for some $1 \le i \le n$, and hence $\langle u , \One \rangle = 1$.
\end{remark}

\begin{lemma}\cite[proof of Theorem 8.3.5]{BultotNicaise20}\label{l:liesinsubring}
Let $u_1,\ldots,u_r$ be the primitive generators of the rays of $\sigma_{K}$.
The element
$$ (L - 1)^{n - \dim K} \sum_{u \in \sigma^{\circ}_{K} \cap \mathbb{N}^{n}} L^{-\langle u, \mathbf{1} \rangle}T^{N(u)}$$
lies in the subring $\mathbb{Z}[L, L^{-1}, T]\Big[
\frac{1}{1 - L^{-\langle u_i, \mathbf{1} \rangle} T^{N(u_i)}}
\Big]_{\{ i \in [r] : N(u_i) \neq 0\}}$ of 
$\Z[L]\llbracket L^{-1},T\rrbracket $.
\end{lemma}
\begin{proof}
Observe that the restriction of $N$ to $\sigma_K$ is a nonnegative linear function. 
The result follows from Lemma~\ref{lem:powerseries} and Remark~\ref{r:Nzerorays}.
\end{proof}

In \cite{BultotNicaise20}, they
define 
$$ (\L - 1)^{n - \dim K} \sum_{u \in \sigma^{\circ}_{K} \cap \mathbb{N}^{n}} \L^{-\langle u, \mathbf{1} \rangle}T^{N(u)} \in \mathcal{M}^{\hat{\mu}}\llbracket T \rrbracket$$
to be the image of the expression in Lemma~\ref{l:liesinsubring} 
under the specialization map  
$\Z[L,L^{-1}]\llbracket T \rrbracket \to \mathcal{M}^{\hat{\mu}}\llbracket T \rrbracket$ that sends $L$ to $\L$.

\begin{theorem}\cite[Theorem 8.3.5]{BultotNicaise20}\label{thm:combformula}
Suppose 
$f$ is nondegenerate.
Then
\begin{equation}\label{eq:BNformula}
Z_{\mot}(T) = \sum_{K}
\left( [Y_K(0)] \frac{\L^{-1}T}{1-\L^{-1}T} + [Y_K(1)] \right) \bigg ((\L - 1)^{n  - \dim K} \sum_{u \in \sigma^{\circ}_{K} \cap \mathbb{N}^{n}} \mathbb{L}^{-\langle u, \mathbf{1} \rangle}T^{N(u)} \bigg) \in \mathcal{M}^{\hat{\mu}}\llbracket T \rrbracket,
\end{equation}
where the sum is over nonempty compact faces $K \in \Gamma$.
\end{theorem}

\begin{remark}

There is an extra factor of $(\L - 1)$ in \eqref{eq:BNformula} 
that does not appear in \cite[Theorem 8.3.5]{BultotNicaise20} for consistency with our choice of normalization of the local motivic zeta function, cf. Remark~\ref{r:conventions}.
\end{remark}

\subsection{The local formal zeta function}\label{ss:formal} 
We now introduce the \emph{local formal zeta function} of $f$, denoted $Z_{\for}(T)$, which is a power series over a polynomial ring that specializes to $Z_{\mot}(T)$. The key advantage of $Z_{\for}(T)$ is that it lies in a power series ring  over an integral domain, so it is easier to understand 
sets of candidate poles 
of $Z_{\for}(T)$. Also,  $Z_{\for}(T)$ depends only on $\Newt(f)$, as opposed to $Z_{\mot}(T)$ which depends on $f$. 

Let $D$ be a ring containing $\Z[L,L^{-1},T,\frac{1}{1 - L^{-1}T}]$ as a subring. 
Let $$R_D = D[Y_K : \emptyset \neq K \in \Gamma, K \textrm{ compact}]/(\mathcal{I}_1 + \mathcal{I}_2), \text{ where}$$
$$\mathcal{I}_1 = (Y_{V} - 1 : V \text{ primitive vertex of } \Gamma), \text{ and}$$
$$\mathcal{I}_2 = (Y_G + Y_F - \frac{(L-1)^{\dim F}}{1 - L^{-1}T} : F \text{ compact $B_1$-face with nonempty base $G$}).$$
When $D =  \Z[L,L^{-1}]\llbracket T \rrbracket$, we write $R := R_D$. 
It follows from Example~\ref{e:primitivevertex} and Example~\ref{e:nonemptybase} that we have a well-defined $\Z\llbracket T \rrbracket$-algebra homomorphism\[
\csp \colon R \to \mathcal{M}^{\hat{\mu}}\llbracket T \rrbracket,  \text{ given by }
\csp(L) = \L, 
\: \csp(Y_K) = [Y_K(0)] \frac{\L^{-1}T}{1-\L^{-1}T} + [Y_K(1)].
\]

\begin{definition}\label{d:formallocal}
With the notation above,
$$Z_{\for}(T) :=  \sum_{\substack{\emptyset \neq K \in \Gamma \\ K \textrm{ compact}}}
Y_K \bigg( (L - 1)^{n  - \dim K} \sum_{u \in \sigma^{\circ}_{K} \cap \mathbb{N}^{n}} L^{-\langle u, \mathbf{1} \rangle}T^{N(u)} \bigg) \in R.$$
\end{definition}

Above, the fact that the right-hand side lies in $R$ follows from Lemma~\ref{l:liesinsubring}. By Theorem~\ref{thm:combformula}, $\csp(Z_{\for}(T)) = Z_{\mot}(T)$.

\begin{lemma}\label{l:polynomialring}
The ring $R_D$ 
is isomorphic to a polynomial ring over $D$. 
Moreover, if $D$ is a subring of $D'$, then $R_D$ is naturally a subring of $R_{D'}$. 
\end{lemma}
\begin{proof}
Consider the following change of variables. If $K$ is a nonempty compact face of $\Gamma$, then let
$$Z_K := (-1)^{\dim K}\Bigg( Y_K - \frac{(L - 1)^{\dim K + 1}}{1 - L^{-1}T} \sum_{i = 0}^{n - \dim K - 1} (1 - L)^i\Bigg).$$  
Then $R_D = D[Z_K : \emptyset \neq K \in \Gamma, K \textrm{ compact}]/(\mathcal{I}_1 + \mathcal{I}_2)$, where
$$\mathcal{I}_1 = (Z_{V} - 1 + \frac{L - 1}{1 - L^{-1}T} \sum_{i = 0}^{n - 1} (1 - L)^i : V \text{ primitive vertex of } \Gamma),$$
$$\mathcal{I}_2 = (Z_G - Z_F  : F \text{ compact $B_1$-face with nonempty base $G$}).$$
Consider the equivalence relation on nonempty compact faces in $\Gamma$ generated by $G \sim F$ whenever $F$ is a compact $B_1$-face with nonempty base $G$. Then $R$ is isomorphic to a polynomial ring over $D$ with variables indexed by all equivalence classes that do not contain a primitive vertex of $\Gamma$.
If $D$ is a subring of $D'$, then $R_{D'}$ is a polynomial ring over $D'$ in the same variables as above. It follows that the natural map $R_D \to R_{D'}$ is injective. 
\end{proof}

When $D = \Z[L]\llbracket L^{-1},T\rrbracket $, we let $\tilde{R} := R_D$. By Lemma~\ref{l:polynomialring}, $R$ is a subring of $\tilde{R}$. In what follows,  
we will freely view $Z_{\for}(T)$ as an element of $\tilde{R}$,    in order to ensure relevant infinite sums in $L^{-1}$ and $T$ are well-defined.

We next define the notion of a set of candidate poles for the local formal zeta function.
Let $\mathcal{P}$ be a finite set of rational numbers containing $-1$. Then  $\mathcal{P}$ is a \emph{set of candidate poles} for some power series $Z(T) \in R$ if $Z(T)$ belongs to the subring $R_D$ of $R$, where 
$$D = \Z [L,L^{-1},T]\bigg[ \frac{1}{1 - L^aT^b}\bigg]_{(a, b) \in \mathbb{Z} \times \mathbb{Z}_{>0}, a/b \in \mathcal{P}}.$$
By Lemma~\ref{l:liesinsubring}, 
$\{ \alpha \in \mathbb{Q} : \Contrib(\alpha) \neq \emptyset \} \cup \{ -1 \}$
is a 
set of candidate poles for $Z_{\for}(T)$. 

\begin{remark}\label{r:specializecandidate}
Since $\csp(Z_{\for}(T)) = Z_{\mot}(T)$, any set of candidate poles for $Z_{\for}(T)$ is a set of candidate poles for $Z_{\mot}(T)$.
\end{remark}

\begin{remark}\label{r:unioncandidate}
Let $\mathcal{P}_1$ and $\mathcal{P}_2$ be sets of candidate poles for elements $Z_1(T)$ and $Z_2(T)$ in $R$ respectively. 
It follows from the definition that 
$\mathcal{P}_1 \cup \mathcal{P}_2$ is a set of candidate poles for $Z_1(T) + Z_2(T)$. 
\end{remark}

The main benefit of working with candidate poles of $Z_{\for}(T)$ is that they satisfy the following lemma. 

\begin{lemma}\label{l:intersectcandidate}
Let $\mathcal{P}_1$ and $\mathcal{P}_2$ be sets of candidate poles for $Z(T) \in R$. Then 
$\mathcal{P}_1 \cap \mathcal{P}_2$ is a set of candidate poles for $Z(T)$. 
\end{lemma}
\begin{proof}
Let $D' = \Z[L,L^{-1},T,\frac{1}{1 - L^{-1}T}]$ and $R' = R_{D'}$. 
We can write $Z(T) = \frac{F_i(T)}{G_i(t)}$ for $i \in \{1,2\}$ for some $F_i(T) \in R$ and some $G_i(t)$ a finite product of terms of the form $\{ 1 - L^a T^b : (a,b) \in \Z \times \Z_{> 0}, a/b \in \mathcal{P}_i \}$. Suppose that $G_1(T) = (1 - L^c T^d) G_1'(T)$ for some $(c,d) \in \Z \times \Z_{> 0}, c/d \in \mathcal{P}_1 \smallsetminus \mathcal{P}_2$. 
By induction on $\deg(G_1(T))$, the result will follow if we can show that $F_1(T) = (1 - L^cT^d)F_1'(T)$ for some $F_1'(T) \in R'$. 

Since the leading coefficient of $1 - L^cT^d$ is a unit in $R'$, we may apply the division algorithm to write $F_1(T) = (1 - L^cT^d)F_1'(T) + \tilde{F}_1(T)$ for some $F_1'(T),\tilde{F}_1(T) \in R'$, with $\deg \tilde{F}_1(T) < d$. 
The equality $F_1(T)G_2(T) = F_2(T)(1 - L^c T^d) G_1'(T)$ in $R'$ implies that 
$1 - L^cT^d$ divides $\tilde{F}_1(T)G_2(T)$. 
Over an appropriate choice of ring $R_D$ containing $R$ as a subring, $1 - L^c T^d$ has roots
$\{ \exp(\frac{2 \pi i j}{d}) L^{-c/d} : 0 \le j < d \}$. Similarly, $G_2(T)$ has roots contained in $\{ 
\exp(\frac{2 \pi i j}{b}) L^{-a/b} : 0 \le j < b, (a,b) \in \Z \times \Z_{> 0}, a/b \in \mathcal{P}_2 \}$. Since $c/d \notin \mathcal{P}_2$ and $\deg \tilde{F}_1(T) < d$, we deduce that 
$\tilde{F}_1(T) = 0$, and $Z(T) = \frac{F_1'(T)}{G_1'(T)}$, as desired. 
\end{proof}

\begin{remark}\label{r:minimalsetcandidate}
By Lemma~\ref{l:intersectcandidate}, if $Z(T) \in R$ 
admits a set of candidate poles, then there exists a 
minimal set of candidate poles. 
In particular, there exists a minimal set of candidate poles of $Z_{\for}(T)$. 
\end{remark}

\subsection{Simplifying the local formal zeta function}\label{ss:simplifyformal}
We now develop some tools to manipulate the local formal zeta function. 
Given a subset $C \subset \R^n_{\ge 0}$, 
we write
\begin{equation}\label{e:contribution}
Z_{\for}(T)|_{C} := \sum_{\substack{\emptyset \neq K \in \Gamma \\ K \textrm{ compact}}}
Y_K \bigg ((L - 1)^{n  - \dim K} \sum_{u \in \sigma^{\circ}_{K} \cap C \cap \mathbb{N}^{n}} L^{-\langle u, \mathbf{1} \rangle}T^{N(u)} \bigg) \in \tilde{R}.
\end{equation}
We call $Z_{\for}(T)|_{C}$ the \emph{contribution} of $C$ to $Z_{\for}(T)$.

We will now prove a key technical tool we will use to manipulate $Z_{\for}(T)$.
Lemma~\ref{lem:cancellation} is analogous to \cite[Lemma 3.3]{ELT}.

\begin{lemma}\label{lem:cancellation}
Let $F$ be a compact $B_1$-face with nonempty base $G$ and apex $A$ in the direction $e^*_{\ell}$. 
Let $C'$ be a nonzero rational polyhedral cone with $(C')^{\circ} \subset \sigma_F^{\circ}$, and let $C \subset \sigma_G$ be the convex hull of $C'$ and 
$\mathbb{R}_{\ge 0}e_\ell^*$. 
Then
$$Z_{\for}(T)|_{(C^{\circ}\cup (C')^{\circ})}= (L - 1)^{n} \bigg(\sum_{u \in (C^{\circ} \cup (C')^{\circ}) \cap \mathbb{N}^n} L^{-\langle u, \mathbf{1} \rangle} T^{\langle u, A \rangle} \bigg)
\in \tilde{R}.$$
\end{lemma}
\begin{proof}
A simplicial refinement of $C'$ induces a simplicial refinement of $C$, so we may reduce to the case that $C'$ is simplicial. 
Let $u_1,\ldots,u_r$ be the primitive integer vectors spanning the rays of $C'$. Then $u_0 = e_{\ell}^*, u_1,\ldots,u_r$ are the primitive integer vectors spanning the rays of $C$.
With the notation of 
\eqref{e:infiniteseries},
$\BBox^+(C') = \{ u \in \N^n : u = \sum_{i = 1}^r \lambda_i u_i \text{ for some } 0 < \lambda_i \le 1 \}$ and $\BBox^+(C) = \{ u \in \N^n : u = \sum_{i = 0}^r \lambda_i u_i \text{ for some } 0 < \lambda_i \le 1 \}$. We claim that $\BBox^+(C) = \{ u + e_{\ell}^* : u \in \BBox^+(C')\}$. 

Clearly, $\{ u + e_{\ell}^* : u \in \BBox^+(C')\} \subset \BBox^+(C)$. Conversely, consider an element $u' = \sum_{i = 0}^r \lambda_i u_i \in \BBox^+(C)$ for some $0 < \lambda_i \le 1$. 
Let $X$ be a point in $G$. For $1 \le i \le r$, since $u_i \in \sigma_F$, we have $\langle u_i, A - X \rangle = 0$. Also, $N(e_{\ell}^*) = \langle e_{\ell}^* , X \rangle = 0$ and $\langle e_{\ell}^* , A \rangle = 1$. We compute:
\[
\langle u', A - X \rangle = \lambda_0 + \sum_{i = 1}^r \lambda_i \langle u_i, A - X \rangle = \lambda_0 \in \Z.
\]
Hence $\lambda_0 \in \Z \cap (0,1] = \{1 \}$. 
Then $u' = \sum_{i = 1}^r \lambda_i u_i + e_{\ell}^* \in \{ u + e_{\ell}^* : u \in \BBox^+(C')\}$, 
which proves the claim. 

Observe that if $u \in \BBox^+(C')$, then $N(u + e_{\ell}^*) = N(u) = \langle u, A \rangle$ and $\langle u + e_{\ell}^*, \One \rangle = \langle u , \One \rangle + 1$.  Also, $N(u_i)  = \langle u_i , A \rangle$ for $1 \le i \le r$. 
Then using 
\eqref{e:infiniteseries}
and the relations in $\mathcal{I}_2$, we compute the left-hand side of the equality in Lemma~\ref{lem:cancellation}:
\begin{align*}
&(L - 1)\left( Y_F (L - 1)^{n - 1 - \dim F} + Y_G (L - 1)^{n - 1 - \dim G} \frac{L^{-1}}{ 1 - L^{-1}} \right) \frac{\sum_{u \in \BBox^+(C')} L^{-\langle u , \One \rangle} T^{N(u)}}{\prod_{i = 1}^r (1 - L^{-\langle u_i , \One \rangle} T^{N(u_i)})}
\\
&= \frac{(L-1)^{n}}{1 - L^{-1}T} \frac{\sum_{u \in \BBox^+(C')} L^{-\langle u , \One \rangle} T^{\langle u, A \rangle}}{\prod_{i = 1}^r (1 - L^{-\langle u_i , \One \rangle} T^{\langle u_i, A \rangle})}. 
\end{align*}
Similarly, using 
\eqref{e:infiniteseries},
we compute the right-hand side of the equality in Lemma~\ref{lem:cancellation}:
\begin{align*}
(L - 1)^{n} \left(1 + \frac{L^{-1}T}{1 - L^{-1}T}\right) \frac{\sum_{u \in \BBox^+(C')} L^{-\langle u , \One \rangle}  T^{\langle u, A \rangle}}{\prod_{i = 1}^r (1 - L^{-\langle u_i , \One \rangle} T^{\langle u_i, A \rangle})}. 
\end{align*}
The result follows.
\end{proof}

\begin{remark}\label{r:manipulateapex}
We note that a version of Lemma~\ref{lem:cancellation} holds when $G$ is empty, i.e., $F = \{A\}$ is a vertex with some coordinate $1$.
Then $A$ is a primitive vertex, so the relations in 
$\mathcal{I}_1$ imply that for  $C \subset \sigma_A^\circ$,
\[
Z_{\for}(T)|_{C} = (L - 1)^{n}
\left(\sum_{u \in C \cap \mathbb{N}^n} L^{-\langle u, \mathbf{1} \rangle} T^{\langle u, A \rangle} \right)
\in \tilde{R}.
\]
\end{remark}

\section{Fake poles for the local formal zeta function}\label{sec:fakepoles}

\subsection{Overview}\label{ssec:overview}
In this section, we prove Theorem~\ref{thm:nopolesimplicial}. We first introduce some notation before stating a strengthening of Theorem~\ref{thm:nopolesimplicial} and outlining its proof.

We let 
$\ver(F)$ denote the set of vertices of $F$. Recall from Section~\ref{sec:poles} that $\Gamma$ 
is the union of  the proper interior faces of $\Newt(f)$ and their subfaces.
Given a face $F$ of $\Gamma$, recall that $C_F$ is the closure of the cone over $F$, with distinguished generators $\gens(C_F)$.
Then $\Span(F) = \Span(C_F)$ and $\gens(C_F) = \ver(F) \cup \Unb(C_F)$.
Given an inclusion of faces $M \subset F$, let  
$\gens(C_F \smallsetminus C_M) = \gens(C_F) \smallsetminus \gens(C_M)$. 
Recall that a face $F$ of $\Newt(f)$ is \emph{$B_1$} 
if it has an apex $A$ with base direction 
$e_\ell^*$, and $\langle e_{\ell}^*, A \rangle = 1$.
Given a $B_1$-face $F$, let 
$\mathcal{A}_F$ be the set of all choices of such an apex $A$.

\begin{definition}\label{d:s0simplicial}
We say that the Newton polyhedron $\Newt(f)$ is \emph{$\alpha$-simplicial} if for any minimal element $M$ in $\Contrib(\alpha)$ and any face $F \supset M$, 
$\dim C_F = \dim C_M + |\gens(C_F \smallsetminus C_M)|$.
Equivalently, 
the images of the elements of $\gens(C_F \smallsetminus C_M)$ are linearly independent in $\R^n/\Span(C_M)$.
\end{definition}

For example, 
if 
$\Newt(f)$
is simplicial,
then it 
is $\alpha$-simplicial.
If all minimal elements in $\Contrib(\alpha)$ are facets, then $\Newt(f)$ is $\alpha$-simplicial. One key property of $\alpha$-simplicial Newton polyhedra is that every face of $\Contrib(\alpha)$ contains a unique minimal face of $\Contrib(\alpha)$ (Lemma~\ref{l:separate}).
We now state our main theorem.

\begin{theorem}\label{thm:nopole}
Suppose 
$f$ is nondegenerate.
Let 
$$\mathcal{P} = \{ \alpha \in \mathbb{Q} : \Contrib(\alpha) \neq \emptyset \} \cup \{ -1 \}, \text{ and}$$
$$\mathcal{P}' = \{ \alpha \in \mathcal{P} : \alpha \notin \Z, \textrm{every face in } \Contrib(\alpha) \textrm { is } \UB  \textrm{ and } \Newt(f) \textrm{ is } \alpha\textrm{-simplicial} \}.$$  
Then 
$\mathcal{P} \smallsetminus \mathcal{P}'$
is  a set of candidate poles for $Z_{\mot}(T)$. 

\end{theorem} 

Our strategy to prove Theorem~\ref{thm:nopole} involves repeatedly applying Lemma~\ref{lem:cancellation}, which will require us to choose apices and base directions for various $B_1$-faces. We will require the following compatibility condition. 

\begin{definition}\label{def:operative}
A \emph{locally unique labeling} of $\Contrib(\alpha)$ 
is a choice of an apex $A_F$ and a base direction $e_F^*$ for each $F \in \Contrib(\alpha)$ such that:
\renewcommand{\labelitemi}{$(*)$}
\begin{itemize}
\item whenever $F \subset F'$ and $A_F = A_{F'}$, we have $e_F^* = e^*_{F'}$. 
\end{itemize}
\end{definition}

If every face of $\Contrib(\alpha)$ is $\UB$, then $\Contrib(\alpha)$ has a locally unique labeling (Lemma~\ref{prop:operativeiff}).

\medskip

We now summarize the rest of the proof of Theorem~\ref{thm:nopole}. We first establish some notation and basic results in Section~\ref{ssec:combprelim}.   Then, using Lemma~\ref{l:intersectcandidate}, we reduce to showing that for each candidate pole $\alpha \not \in \Z$ that is contributed only by $\UB$-faces and such that $\Newt(f)$ is $\alpha$-simplicial, there is a set of candidate poles for $Z_{\for}(T)$ not containing $\alpha$. 
Fix such an $\alpha$. 

Because $\Newt(f)$ is $\alpha$-simplicial, every face of $\Contrib(\alpha)$ contains a unique minimal face $M$ of $\Contrib(\alpha)$. In Section~\ref{ssec:NMdelta}, we develop the tools to argue that we can consider each minimal face $M$ separately. We construct a neighborhood $N_{M, \le \delta}$ of $\sigma_M$. 
In Lemma~\ref{l:outerformalzeta}, 
we show that if $Z_{\for}(T)|_{N_{M, \le \delta}^\circ}$ admits a set of candidate poles not containing $\alpha$ for each minimal face $M$, then the theorem follows. 

To analyze $Z_{\for}(T)|_{N_{M, \le \delta}^\circ}$, our strategy is to construct a complete fan $\Sigma_{\mathcal{Z}}$ where each maximal cone is labeled by a face containing $M$, which necessarily lies in $\Contrib(\alpha)$. 
We construct $\Sigma_{\mathcal{Z}}$ as the normal fan of a polytope where each vertex is labeled by a face containing $M$.
This polytope is determined by an \emph{$\alpha$-compatible pair} (see Definition~\ref{d:compatible3}).

Because we have fixed a locally unique labeling, 
we may associate to each maximal cone of $\Sigma_{\mathcal{Z}}$, which corresponds to a face $F$ containing $M$, a pair $(A_F, e_F^*)$ consisting of an apex and a base direction. 
We then associate a face containing $M$ and 
a pair $(A,e_\ell^*)$ to all cones in  the fan $\Sigma_{\mathcal{Z}} \cap N_{M, \le \delta}$. 
See Definition~\ref{d:max}.
Consider a nonzero cone $C$ in the fan $\Sigma_{\mathcal{Z}} \cap N_{M, \le \delta}$ and an associated pair $(A,e_\ell^*)$. We arrange that for any face 
$F$ containing $M$,
if the dual cone 
of $F$ intersects $C$, then $F$ is 
a $B_1$-face with apex $A$ and base direction $e_\ell^*$.
 By repeatedly applying Lemma~\ref{lem:cancellation}, we show that the contribution of $C^\circ$ to $Z_{\for}(T)$ is equal to $(L - 1)^n \sum_{u \in C^\circ \cap \mathbb{N}^n} L^{- \langle u, \mathbf{1} \rangle} T^{\langle u, A \rangle}$.
 In order to do this, we need to show that  $C$ is locally defined by elements 
 ``orthogonal to $e_\ell^*$'', see Lemma~\ref{l:Cinvariant}. 
 Using an additional genericity condition (see Definition~\ref{d:compatible3}), we deduce that 
the contribution $Z_{\for}(T)|_{C^\circ}$ 
admits a set of candidate poles not containing $\alpha$. Using this strategy, in Section~\ref{ss:s0compatible} we prove that the existence of an $\alpha$-compatible pair implies Theorem~\ref{thm:nopole}.

\begin{figure}
\begin{center}
\begin{tikzpicture}[scale = 0.7]
\filldraw (0,0) circle [fill=black, radius = 0.1];
\draw [very thick] (0,0) -- (3, 3);
\draw[very thick] (0,0) -- (3, -3);
\filldraw[draw=black,fill=black!20, opacity=0.2] (0,0) -- (3,3) -- (3,-3) -- (0,0);
\draw (-0.2, 0.3) node {$\sigma_F$};
\draw (1.4, 2.1) node {$\sigma_{G_1}$};
\draw (1.4, -2.1) node {$\sigma_{G_2}$};
\draw (2.6, 0) node {$\sigma_M$};
\draw [very thick, red] (-3, 0) -- (1, 0);
\draw [very thick, red] (1, 0) -- (2.3, 1.3);
\draw [very thick, red] (1, 0) -- (2.3, -1.3);
\draw [very thick, red] (2.3, 1.3) -- (3, 1.6);
\draw [very thick, red] (2.3, -1.3) -- (3, -1.6);
\draw [very thick, red] (2.3, 1.3) -- (-2.5, 0);
\draw [very thick, red] (2.3, -1.3) -- (-2.5, 0);
\draw (2.3, -0.5) node {$\textcolor{red}M$};
\draw (0, -1.5) node {$\textcolor{red}{ G_2}$};
\draw (0, 1.5) node {$\textcolor{red}{ G_1}$};
\draw (-1.9, 1) node {$\textcolor{red}{ F}$};
\draw (-1.9, -1) node {$\textcolor{red}{ F}$};
\draw[-{Latex[length=2mm]}] (-1.9, 0.7) -- (-0.7, 0.2);
\draw[-{Latex[length=2mm]}] (-1.9, -0.7) -- (-0.7, -0.2);
\end{tikzpicture}
\qquad
\begin{tikzpicture}[scale = 0.7]
\filldraw (0,0) circle [fill=black, radius = 0.1];
\draw [very thick] (0,0) -- (3, 3);
\draw[very thick] (0,0) -- (3, -3);
\filldraw[draw=black,fill=black!20, opacity=0.2] (0,0) -- (3,3) -- (3,-3) -- (0,0);
\draw [very thick, red] (-3, -0.5) -- (1, -0.5);
\draw [very thick, red] (1, -0.5) -- (2.3, 1.2);
\draw [very thick, red] (1, -0.5) -- (2.3, -1.8);
\draw [very thick, red] (2.3, 1.2) -- (3, 1.35);
\draw [very thick, red] (2.3, -1.8) -- (3, -2.1);
\draw [very thick, red] (2.3, 1.2) -- (-2.8, -0.5);
\draw [very thick, red] (2.3, -1.8) -- (-2.5, -0.5);
\draw (2.3, -0.5) node {$\textcolor{red}M$};
\draw (0, -1.95) node {$\textcolor{red}{ G_2}$};
\draw (0, 1.5) node {$\textcolor{red}{ G_1}$};
\draw (-2.1, 0.5) node {$\textcolor{red}{ F}$};
\draw (-1.9, -1.4) node {$\textcolor{red}{ F}$};
\draw[-{Latex[length=2mm]}] (-2.1, 0.25) -- (-0.7, -0.1);
\draw[-{Latex[length=2mm]}] (-1.9, -1.1) -- (-0.7, -0.6);
\end{tikzpicture}
\end{center}

\caption{
A complete fan with maximal cones indexed by 
faces containing $M$
and a corresponding 
deformation.
We show the intersection of $\Span(\sigma_M)$ with an affine hyperplane. 
The cone $\sigma_M$ is shown
in black and grey, 
while the codimension $1$ cones of the complete fan appear in red, with their maximal cones labeled in red.
}\label{fig:saturatedchains}
\end{figure}
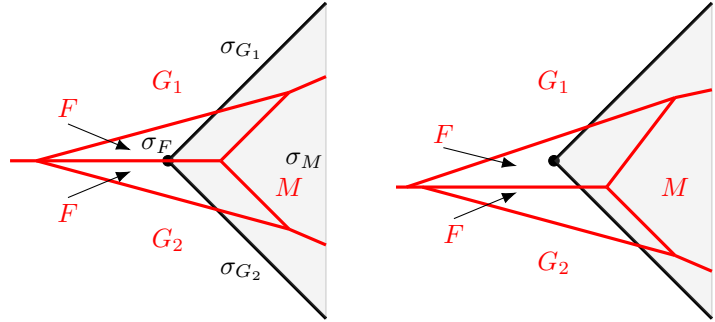

In Section~\ref{ss:exist1}, we show that the existence of an $\alpha$-compatible pair is implied by the existence of a \emph{restricted, weakly $\alpha$-compatible} pair (Definition~\ref{d:weaklys0compatible}, Definition~\ref{d:restricted2}). More precisely, we show that each restricted, weakly $\alpha$-compatible pair can be ``deformed'' into an $\alpha$-compatible pair.  In Section~\ref{ss:exist2}, we give an explicit construction of a restricted, weakly $\alpha$-compatible pair, which is where we use the locally unique condition.  
Figure~\ref{fig:saturatedchains} shows an example of a fan corresponding to a restricted, weakly $\alpha$-compatible pair, and a corresponding 
deformation.

\subsection{Combinatorial preliminaries}\label{ssec:combprelim}
In this section, we prepare for the constructions that will take the rest of the section.
Recall from Definition~\ref{def:UB1} that a face $G$ of $\Newt(f)$ is $\UB$ 
if there exists an apex $A$ in $G$ with a unique choice of  base direction 
$e_\ell^*$, and 
$\langle e_{\ell}^*, A \rangle = 1$. 

\begin{lemma}\label{prop:operativeiff}
If every face in $\Contrib(\alpha)$ is $\UB$, then $\Contrib(\alpha)$ admits a locally unique labeling.
\end{lemma}

\begin{proof}
Suppose that every face $F$ in $\Contrib(\alpha)$ is $\UB$. Then it has an apex $A_F$ 
with a unique choice of  base direction 
$e_F^*$.
We may choose one such apex, and label $F$ by $(A_F, e_F^*)$ to get a locally unique labeling. Indeed, if $F \subset {F'}$ and $A_F = A_{F'}$, then $e^*_{F'}$ is a base direction for $F$ with apex $A_F$.
We deduce that $e_F^* = e^*_{F'}$. 
\end{proof}

Given an element $V \in \Gamma$, we say that $u \in \mathbb{R}^n$ is \emph{critical with respect to $(\alpha, V)$} if $\alpha \langle u, V \rangle + \langle u, \mathbf{1} \rangle = 0$. 
Recall that we have a piecewise linear function $N(u) = \min \{ \langle u, W \rangle : W \in \Newt(f) \}.$
We say that $u \in \R^n_{\ge 0}$ is \emph{critical with respect to $\alpha$} if 
$\alpha N(u) + \langle u, \mathbf{1} \rangle = 0$. Equivalently, 
for some/any $G \in \Gamma$ such that $u \in \sigma_G$, and some/any 
element $V \in G$, $u$ is critical with respect to $(\alpha,V)$.  
A set is critical with respect to $\alpha$ (or $(\alpha,V)$) if 
every element of the set is critical with respect to $\alpha$ (or $(\alpha,V)$).

\begin{remark}\label{r:hyperplane}
Unless $V = -(1/\alpha) \mathbf{1}$, the points $u \in \mathbb{R}^n$ critical with respect to $(\alpha, V)$ form a hyperplane. A special case of this observation is \cite[Lemma 3.4]{ELT}. 
\end{remark}

We conclude this section with three combinatorial observations. 

\begin{lemma}\label{l:existsVM}
Let $M$ be an element of $\Contrib(\alpha)$,
and assume that every face of
$\Contrib(\alpha)$ is $\UB$. 
Assume that $\alpha \neq -|\ver(M)| \in \Z$. 
Then there exists a vertex $V_M$ of $M$ 
such that $M$ is not a $B_1$-face with apex $V_M$.
\end{lemma}
\begin{proof}
Suppose every vertex of $M$ is an apex.
Since $\1 \in \Span(M)$, it follows that $\1 - \sum_{V \in \ver(M)} V$ is a linear combination of the unbounded directions of $M$, and hence $\alpha = -\psi_M(\One) = - |\ver(M)|$, a contradiction.
\end{proof}

\begin{lemma}\label{l:separate}
Assume that $\Newt(f)$ is $\alpha$-simplicial. 
If $M_1,M_2$ are distinct minimal elements in $\Contrib(\alpha)$, then $\sigma_{M_1} \cap \sigma_{M_2} = \{ 0 \}$. 
\end{lemma}
\begin{proof}
We argue by contradiction. Suppose  $\sigma_{M_1} \cap \sigma_{M_2} \neq  \{ 0 \}$. 
Then there exists a facet $F$ in $\partial \Newt(f)$ such that 
$\sigma_F \subset \sigma_{M_1} \cap \sigma_{M_2}$. 
Note that $F$ is interior and hence $F \in \Gamma$. 
Equivalently, 
$M_1, M_2 $ are common faces of $F$. 
In particular, $M_1 \cap M_2$ is a (possibly empty) face of $F$, and $C_{M_1 \cap M_2} = C_{M_1} \cap C_{M_2}$.
Let $\gens(C_{M_1} \smallsetminus C_{M_1 \cap M_2}) = \{ W_1,\ldots,W_s \}$ and $\gens(C_{M_2} \smallsetminus C_{M_1 \cap M_2}) = \{ W_1',\ldots,W_{s'}' \}$.
By Definition~\ref{d:s0simplicial} applied to $M_2 \subset F$, 
$\gens(C_{M_1} \smallsetminus C_{M_1 \cap M_2})$
is linearly independent in $\R^n/\Span(C_{M_2})$. 
By Definition~\ref{d:s0simplicial} applied to $M_1 \subset F$,  
$\gens(C_{M_2} \smallsetminus C_{M_1 \cap M_2})$
is linearly independent in $\R^n/\Span(C_{M_1})$.
We claim that $W_1,\dotsc ,W_s,W_1',\dotsc ,W_{s'}'$ are linearly independent in $\R^n/\Span(C_{M_1 \cap M_2})$. 
Indeed, if $\sum_i a_i W_i + \sum_j b_j W_j' = 0$ in  $\R^n/\Span(C_{M_1 \cap M_2})$, then the corresponding equation in $\R^n/\Span(C_{M_1})$ implies 
that $b_j = 0$ for all $j$, and the corresponding equation in $\R^n/\Span(C_{M_2})$ implies that $a_i = 0$ for all $i$.

By assumption,
$\One = a_1 W_1 + \cdots + a_s W_s \in \R^n/\Span(C_{M_1 \cap M_2})$ for some $a_1,\ldots,a_s \in \R$, and 
$\One = a_1' W_1' + \cdots + a_{s'}' W_{s'}' \in \R^n/\Span(C_{M_1 \cap M_2})$ for some $a_1',\dotsc,a_{s'}' \in \R$.
Subtracting one equation from the other, and using the linear independence of $W_1,\dotsc,W_s,W_1',\dotsc,W_{s'}'$ in $\R^n/\Span(C_{M_1 \cap M_2})$, we deduce that $a_i = a_j' = 0$ for all $i,j$. Hence $\One \in \Span(C_{M_1 \cap M_2})$, and $M_1 \cap M_2 \in \Contrib(\alpha)$. This contradicts the minimality of $M_1$ and $M_2$.
\end{proof}

\begin{remark}\label{r:apexisbounded}
If $e^*_{\ell}$ is a base direction of $F$ with apex $A$, then  $e_{\ell} \not \in \Unb(C_F)$. Indeed, if  $e_{j} \in \Unb(C_F)$, then it follows from Definition~\ref{def:B1} that $\langle e^*_{\ell}, e_j \rangle = 0$. 
\end{remark}

\subsubsection*{Assumptions and notation}

For the remainder, we will assume that $\alpha \not \in \Z$, $\Newt(f)$ is $\alpha$-simplicial, and that all faces of $\Contrib(\alpha)$ are $\UB$.
Let $M$ be a minimal element of $\Contrib(\alpha)$. 
Recall that we have chosen a locally unique labeling $(A_F, e_F^*)$ of $\Contrib(\alpha)$. 
By Lemma~\ref{l:existsVM}, we may fix a vertex $V_M$ of $M$ which is not an apex, and hence satisfies 
\begin{equation}\label{e:VM}
\langle e_F^*, V_M \rangle = 0 \textrm{ for all } F \supset M.
\end{equation}
We also 
fix a 
point $W_M$ in the relative interior of $M$. 
Given a nonzero point $W$  
and $\epsilon \in \R$, let 
$\mathrm{H}_{W,\epsilon} = \{ u : \langle u , W \rangle = \epsilon \}$, and consider the associated half-spaces
$\mathrm{H}_{W,\ge \epsilon},\mathrm{H}_{W,> \epsilon},\mathrm{H}_{W,\le \epsilon},\mathrm{H}_{W,< \epsilon}$. We let $\mathrm{H}_W := \mathrm{H}_{W,0}$.
Let $S' = \{ u \in \R^n : \langle u , \One \rangle = 1 \}$, and let $S = \Conv{e_1^*,\ldots,e_n^*} \subset S'$ be the standard $(n-1)$-dimensional simplex.

\subsection{Covering the critical locus}\label{ssec:NMdelta}

The goal of this section is to build small neighborhoods covering the locus of $u \in \mathbb{R}^n_{\ge 0}$ that is critical with respect to $\alpha$. This will allow us to concentrate our attention on a single minimal face in $\Contrib(\alpha)$.

\begin{definition}\label{d:NMdelta}
Let $M$ be a minimal element of $\Contrib(\alpha)$ and let $\delta \in \Q_{\ge 0}$. We define 
$N_{M,\leq \delta}$ to be the cone over $\{ u \in S : \langle u , W_M \rangle - N(u) \le \delta \}$.

\end{definition}

Similarly, we let  $N_{M,< \delta}$, $N_{M,\delta}$ and $N_{M,\ge \delta}$ be the cones over 
$\{ u \in S : \langle u , W_M \rangle - N(u) < \delta \}$, $\{ u \in S : \langle u , W_M \rangle - N(u) = \delta \}$ and 
$\{ u \in S : \langle u , W_M \rangle - N(u) \ge \delta \}$ respectively.

Here $\langle u , W_M \rangle - N(u)$ is a nonnegative function on $\R^n_{\ge 0}$ that is piecewise linear with respect to $\Sigma$. 
Because $N(u)$ is the support function of a 
polyhedron
and hence convex, $N_{M, \leq \delta}$ is convex. 
It follows that $N_{M, \leq \delta}$ is a rational polyhedral cone of dimension $n$. 
Note that $\sigma_M$ is the cone over  $\{ u \in S : \langle u , W_M \rangle - N(u) = 0 \}$, and hence 
$\sigma_M \subset N_{M,\leq \delta}$. 
We can equivalently write 
$N_{M,\leq \delta} = \{ u \in \R^n_{\ge 0} : \langle u , W_M \rangle - N(u) \le \delta \langle u , \One \rangle \}$. 
Also, $N_{M,\leq \delta}^\circ = N_{M,< \delta} \cap \R^n_{> 0}$.

\begin{lemma}\label{l:intersectM2}
Let $C \subset \R^n_{\ge 0}$ be a closed cone such that  $C \cap \sigma_M = \{ 0 \}$. Then $C \cap N_{M,\leq \delta} = \{ 0 \}$ for $\delta$ sufficiently small. 
\end{lemma}
\begin{proof}
We may assume that $C \cap S \neq \emptyset$. 
Since $C \cap S$ is compact, we may consider the  minimal element $b > 0$ of $\{ \langle u , W_M \rangle - N(u) : u \in C \cap S \}$. Then $C \cap N_{M,\leq \delta} \cap S = \emptyset$ for $\delta < b$.
\end{proof}

\begin{lemma}\label{l:notinGamma}
Let $K$ be a face of $\partial \Newt(f)$ and suppose that $K \notin \Gamma$. Let $M$ be a minimal element in $\Contrib(\alpha)$.
Then 
$\sigma_K \cap N_{M,\leq \delta} = \{ 0 \}$ for $\delta$ sufficiently small. 
\end{lemma}
\begin{proof}
By Lemma~\ref{l:intersectM2}, it is enough to show that $\sigma_K \cap \sigma_M = \{0\}$. Suppose that
$\sigma_K \cap \sigma_M \neq \{0\}$. Then  $\sigma_K \cap \sigma_M = \sigma_{K'}$ for some face $K'$ of $\partial \Newt(f)$ containing both $K$ and $M$. Since $M \subset K'$, 
$K' \in \Contrib(\alpha)$. 
Since $K \subset K' \in \Gamma$, we deduce that $K \in \Gamma$, a contradiction.
\end{proof}

The following  lemma is immediate from Lemma~\ref{l:separate} and Lemma~\ref{l:intersectM2}.

\begin{lemma}\label{l:separateN}
If $M_1,M_2$ are distinct minimal elements in $\Contrib(\alpha)$, then $N_{M_1,\leq \delta} \cap N_{M_2,\leq \delta} = \{ 0 \}$ for $\delta$ sufficiently small. 
\end{lemma}

\begin{lemma}\label{l:Npositive}
For $\delta$ sufficiently small, if $u \in N_{M,\leq \delta} \smallsetminus \{ 0 \}$, then $N(u) > 0$.  
\end{lemma}
\begin{proof}
After scaling, we may assume that $u \in N_{M,\leq \delta} \cap S$. Since $S$ is compact, we may consider $b = \min \{ \langle u , W_M \rangle : u \in S \}$.  Since $M$ is interior,  $W_M \notin \partial \R^n_{\ge 0}$, and hence $b > 0$. For $\delta < b$ and $u \in N_{M,\leq \delta} \cap S$, 
$N(u) = \langle u , W_M \rangle - (\langle u , W_M \rangle - N(u)) \ge \langle u , W_M \rangle - \delta \ge b - \delta > 0$. 
\end{proof}

Let $\Sigma'$ be a fan supported on $\R^n_{\ge 0}$ that refines $\Sigma$. We may consider the fan 
$\Sigma' \cap N_{M,\leq \delta}$ supported on $N_{M,\leq \delta}$ 
given by all cones of the form 
$\{ C \cap C' : C \in \Sigma', C' \textrm{ is a face of } N_{M,\leq \delta} \}$.
The lemma below gives a more explicit description.

\begin{lemma}\label{l:DeltacapN}
Let $\Sigma'$ be a fan supported on $\R^n_{\ge 0}$ that refines $\Sigma$. Let $C \in \Sigma'$, and fix $\delta$ sufficiently small. Then
\begin{enumerate}
\item If $C \cap \sigma_M = \{ 0 \}$, then $C \cap N_{M,\leq \delta} = \{ 0 \}$.

\item  If $C \subset \sigma_M$, then $C$ is a cone in 
$\Sigma' \cap N_{M,\leq \delta}$. 

\item If $C \cap \sigma_M \neq \{ 0 \}$ and $C \not \subset \sigma_M$, then $\dim (C \cap N_{M,\leq \delta}) = \dim C$. 
Moreover, $\dim (C \cap N_{M,\delta}) = \dim C - 1$, and $C \cap N_{M,\delta}$ is the only proper face of $C \cap N_{M,\leq \delta}$ that is not contained in a proper face of $C$.

\end{enumerate}

\end{lemma}
\begin{proof}
If $C \cap \sigma_M = \{ 0 \}$, then $C \cap N_{M,\leq \delta} = \{ 0 \}$ by Lemma~\ref{l:intersectM2}. If 
$C \subset \sigma_M$, then since $\sigma_M \subset N_{M,\leq \delta}$, 
$C \cap N_{M,\leq \delta} = C$.
This establishes the first two properties.
Assume that $C \cap \sigma_M \neq \{ 0 \}$ and $C \not \subset \sigma_M$. 
Since $\Sigma'$ refines $\Sigma$, $C \subset \sigma_K$ for some face $K$ of $\partial \Newt(f)$. Fix a vertex $V$ of $K$, and let $P = C \cap S$. Then $P \cap  N_{M,\leq \delta} = P \cap \mathrm{H}_{W_M - V,\le \delta}$. 
For any $\delta > 0$, the relative interior of $P$ 
intersects both $\mathrm{H}_{W_M - V,>0}$ and $\mathrm{H}_{W_M - V,< \delta}$ because $P \not \subset \sigma_M$ and $P \cap \sigma_M \neq \emptyset$. It follows that for $\delta$ sufficiently small, 
$\mathrm{H}_{W_M - V,\delta}$ intersects the relative interior of $P$. We deduce that $P \cap  N_{M,\leq \delta}$ has dimension $\dim P$, and that the only proper face of $P \cap  N_{M,\leq \delta}$ that is not contained in a 
proper face of $P$ is $P \cap \mathrm{H}_{W_M - V, \delta}$, 
which has dimension $\dim P - 1$. This establishes the result.
\end{proof}

The following two remarks are corollaries of the 
Lemma~\ref{l:DeltacapN} and its proof.

\begin{remark}\label{r:relativeinteriors}
Let $C$ be a rational polyhedral cone such that $C \subset \sigma_K$ for some face $K$ of $\partial \Newt(f)$. 
Then for $\delta$ sufficiently small, 
$C^\circ \cap N_{M,< \delta} = (C \cap N_{M,\leq \delta})^\circ$, and $C^\circ \cap N_{M,\ge \delta} = (C \cap N_{M,\geq \delta})^\circ \cup (C \cap N_{M,\delta})^\circ$.
\end{remark}

\begin{remark}\label{r:tildesigmaK}
Let $K$ be a nonempty face of $\partial \Newt(f)$, 
and let 
$\tilde{\sigma}_K = \sigma_K \cap (\cap_i  N_{M_i,\geq \delta})$. Assume that $\delta$ is chosen sufficiently small.
If $K \in \Contrib(\alpha)$, then $M_i \subset K$ for some $1 \le i \le r$, and hence $\sigma_K \subset N_{M_i,< \delta}$, and $\tilde{\sigma}_K = \emptyset$. 
Assume that $K \notin \Contrib(\alpha)$.
By Lemma~\ref{l:separateN} and Lemma~\ref{l:DeltacapN}, $\tilde{\sigma}_K$ is a rational polyhedral cone of dimension $\dim \sigma_K$, and 
$\sigma_K^\circ \cap (\cap_i  N_{M_i,\geq \delta}) = \tilde{\sigma}_K^\circ \cup (\cup_i (\tilde{\sigma}_K \cap N_{M_i,\delta})^\circ)$. 
\end{remark}

\begin{lemma}\label{l:RcriticalC}
Let $\Sigma'$ be a fan supported on $\R^n_{\ge 0}$ that refines $\Sigma$.
Let $\gamma$ be a ray of the fan $\Sigma' \cap N_{M,\leq \delta}$ such that $\gamma \not \subset \sigma_M$, 
for some $\delta$ chosen sufficiently small. 
Let $C$ be the smallest cone in $\Sigma'$ containing 
$\gamma$. 
The following properties hold:
\begin{enumerate}
\item\label{i:DD1} $C \cap \sigma_M \neq \{ 0 \}$,
\item\label{i:DD2} 
the smallest cone in $\Sigma$ containing $C$ has the form $\sigma_K$ for some $K$ in $\Gamma$,
\item\label{i:DD4} if $\gamma$ is critical with respect to $(\alpha,V)$,
for some vertex $V$ of either $K$ or $M$, then 
$C$ is critical with respect to $(\alpha,V)$. 
\end{enumerate}

\end{lemma}
\begin{proof}
It follows from Lemma~\ref{l:DeltacapN} that 
$\gamma = C \cap N_{M,\delta}$,
$C \cap \sigma_M \neq \{ 0 \}$, and $\dim (C \cap N_{M,\leq \delta}) = \dim C = 2$. This establishes 
\eqref{i:DD1}. Lemma~\ref{l:notinGamma} implies \eqref{i:DD2}.
Since $\sigma_K \cap \sigma_M \neq \{ 0 \}$, 
$\sigma_K \cap \sigma_M = \sigma_{K'}$ for some $K'$ in $\Gamma$ containing both $M$ and $K$. 
Fix any vertex $V$ of $K'$. Since $M \in \Contrib(\alpha)$, we have $\sigma_{K'} \subset \mathrm{H}_{\alpha V + \One}$. 
Let $\gamma' \neq \gamma$ be the other ray spanning $C \cap N_{M,\leq \delta}$. 
Then $\gamma' \subset \sigma_{K'} \subset  \mathrm{H}_{\alpha V + \One}$. 
Hence, if $\gamma
\subset  \mathrm{H}_{\alpha V + \One}$, then  $\Span(C) = \Span(C \cap N_{M,\leq \delta}) = \R \gamma + \R \gamma' \subset \mathrm{H}_{\alpha V + \One}$, and $C \subset \mathrm{H}_{\alpha V + \One}$. 
This establishes \eqref{i:DD4}.
\end{proof}

\begin{lemma}\label{l:raynotcrit}
Let $\gamma$ be a ray of the fan $\Sigma \cap N_{M,\leq \delta}$ such that $\gamma \not \subset \sigma_M$,  for some $\delta$ chosen sufficiently small. Then $\gamma$ is not critical with respect to $\alpha$. 
\end{lemma}
\begin{proof}
By Lemma~\ref{l:RcriticalC}, there exists $K \in \Gamma$ such that $\sigma_K$ is the smallest cone in $\Sigma$ containing $\gamma$.
Suppose that 
$\gamma$ is critical with respect to $\alpha$. Then $\gamma$ is critical with respect to $(\alpha,V)$ for any vertex $V$ of $K$. By Lemma~\ref{l:RcriticalC}, $\sigma_K$ is critical with respect to $(\alpha,V)$, and hence $K \in \Contrib(\alpha)$. Since $\sigma_K \cap N_{M,\leq \delta} \neq \{ 0 \}$, Lemma~\ref{l:separateN} implies that $M \subset K$, and hence $\gamma \subset \sigma_M$, a contradiction.
\end{proof}

\begin{lemma}\label{l:outersummandnopole}
Let $M_1,\ldots,M_r$ be the minimal elements in $\Contrib(\alpha)$, and let $K$ be a nonempty face of $\partial \Newt(f)$.
Then for $\delta$ sufficiently small,  
$(L - 1)^{n  - \dim K} \sum_{u \in
\sigma_K^\circ \cap (\cap_i  N_{M_i,\geq \delta})   \cap \mathbb{N}^{n}} L^{-\langle u, \mathbf{1} \rangle}T^{N(u)}$
lies in $R$ and admits a set of candidate poles not containing $\alpha$. 

\end{lemma}
\begin{proof}

Let $\tilde{\sigma}_K = \sigma_K \cap (\cap_i  N_{M_i,\geq \delta})$. 
By Remark~\ref{r:tildesigmaK},  if  $K \in \Contrib(\alpha)$, then $\tilde{\sigma}_K  = \emptyset$.
Assume that $K \notin \Contrib(\alpha)$.
By Remark~\ref{r:tildesigmaK},
$\tilde{\sigma}_K$ is a rational polyhedral cone of dimension $\dim \sigma_K$, and 
$\sigma_K^\circ \cap (\cap_i  N_{M_i,\geq \delta})  = \tilde{\sigma}_K^\circ \cup (\cup_i (\tilde{\sigma}_K \cap N_{M_i,\delta})^\circ)$. 
By Lemma~\ref{l:intersectM2}, the rays of $\tilde{\sigma}_K$ are the union of the  rays of $\sigma_K$ that are not critical with respect to $\alpha$, and the rays of $\sigma_K \cap N_{M_i,\leq \delta}$ that do not lie in $\sigma_{M_i}$ for $1 \le i \le r$.
By Lemma~\ref{l:raynotcrit}, none of the rays of 
$\tilde{\sigma}_K$ are critical with respect to $\alpha$. 
By Remark~\ref{r:Nzerorays} and Lemma~\ref{l:Npositive}, if $u$ is a primitive generator of a ray of $\tilde{\sigma}_K$ and $N(u) = 0$, then $u = e_i^*$ for some $1 \le i \le n$, and hence $\langle u , \One \rangle = 1$.
Since the restriction of $N$ to $\tilde{\sigma}_K \subset \sigma_K$ is linear,
Lemma~\ref{lem:powerseries} implies that
the following elements of $\tilde{R}$ lie in $R$ and admit sets of candidate poles not containing $\alpha$:
\[
(L - 1)^{n  - \dim K} \sum_{u \in \tilde{\sigma}^{\circ}_{K} \cap \mathbb{N}^{n}} L^{-\langle u, \mathbf{1} \rangle}T^{N(u)}, \text{ and } (L - 1)^{n  - \dim K} \sum_{u \in 
(\tilde{\sigma}_K \cap N_{M_i,\delta})^\circ
\cap \mathbb{N}^{n}} L^{-\langle u, \mathbf{1} \rangle}T^{N(u)} 
\]
for $1 \le i \le r$. 
The result now follows from Remark~\ref{r:unioncandidate}. 
\end{proof}

The lemma below follows immediately from Definition~\ref{d:formallocal} and Lemma~\ref{l:outersummandnopole} and will allow us to reduce our study of 
$Z_{\for}(T)$ to the study of 
$Z_{\for}(T)|_{N_{M,< \delta}}$, when $M$ is a minimal element of $\Contrib(\alpha)$. 

\begin{lemma}\label{l:outerformalzeta}
Let $M_1,\ldots,M_r$ be the minimal elements in $\Contrib(\alpha)$.
Then for $\delta$ sufficiently small,  
$Z_{\for}(T)|_{\cap_i  N_{M_i,\geq \delta}}$
lies in $R$ and admits a set of candidate poles not containing $\alpha$.
\end{lemma}

\subsection{Establishing fake poles using $\alpha$-compatible sets}\label{ss:s0compatible}

The goal of this section is to show that the existence of a fan with certain properties implies that there is a set of candidate poles for the
local formal zeta function not containing $\alpha$.

From this point on, we fix a minimal element $M$ in
$\Contrib(\alpha)$. 
Let $\Contrib(\alpha)_M := \{ F \in \Contrib(\alpha) : F \supset M \}$. 
Fix a nonempty finite set $\mathcal{S}$.
Given a finite collection $\mathcal{Z}  = (Z_s)_{s \in \mathcal{S}}$ of elements in $\mathbb{Q}^n$ indexed by $\mathcal{S}$, 
we let $Q_{\mathcal{Z}}$ denote the convex hull of $\mathcal{Z}$, and let $\Sigma_{\mathcal{Z}}$ be the corresponding (rational) dual fan supported on $\R^n$. 
Given a nonempty face $J$ of $Q_{\mathcal{Z}}$, we write $\tau_J$ for the corresponding cone in 
$\Sigma_{\mathcal{Z}}$. 
Given an element 
$s \in \mathcal{S}$,
we write 
$J_{Z_s}$ for the smallest face of $Q_{\mathcal{Z}}$
containing $Z_s$, and we write $\tau_{Z_s} := \tau_{J_{Z_s}}$.
Explicitly,
$\tau_{Z_s} := \{ u \in \R^n : \langle u , Z_s \rangle \le  \langle u , Z_{s'} \rangle, \textrm{ for all } s' \in \mathcal{S} \}.$
Observe that $\tau_J = \cap_{Z_s \in J} \tau_{Z_s}$.
Also, note that we do not require the elements of $\mathcal{Z}$ to be distinct, and that $\tau_{Z_s}$ may equal $\tau_{Z_{s'}}$ even if $s \not= s'$.

Before defining $\alpha$-compatible pairs (Definition~\ref{d:compatible3}), we introduce a weaker notion, which 
satisfies all
of the properties of $\alpha$-compatible pairs with the exception of a genericity condition. 

\begin{definition}\label{d:weaklys0compatible} 
Consider a pair $(\mathcal{Z},\mathcal{F})$, where 
$\mathcal{Z} = (Z_s)_{s \in \mathcal{S}}$ and $\mathcal{F} = (F_s)_{s \in \mathcal{S}}$ are collections of elements  of $\Q^n$ and $\Contrib(\alpha)_M$ respectively. Then $(\mathcal{Z},\mathcal{F})$ is 
\emph{weakly $\alpha$-compatible} if it satisfies the following property:

Whenever
$\sigma_K^\circ \cap \tau_{Z_{s}} \cap \tau_{Z_{s'}}  \ne \emptyset$ for some  $K \in \Contrib(\alpha)_M$  and $s,s' \in \mathcal{S}$, with possibly $s = s'$,  
then 
\begin{enumerate}
\item\label{i:weak1} $K \subset F_s$, 
\item\label{i:weak2} either $F_s \subset F_{s'}$ or $F_{s'} \subset F_s$, and
\item\label{i:weak3} $
\langle e^*_{F_s}, Z_s \rangle =
\langle e^*_{F_s}, Z_{s'} \rangle = 0$.
\end{enumerate}

\end{definition}

Conditions \eqref{i:weak1} and \eqref{i:weak2} are used below to associate a face in $\Contrib(\alpha)_M$ to every cone of $\Sigma_{\mathcal{Z}}$ with nonzero interection with $\sigma_M$.
See Definition~\ref{d:max}.
 Condition \eqref{i:weak3} is used to control the structure of cones in $\Sigma_{\mathcal{Z}}$ in a way that will allow us to apply Lemma~\ref{lem:cancellation}. See Lemma~\ref{l:Cinvariant}.

\begin{lemma}\label{l:maxelem}
Let $(\mathcal{Z}, \mathcal{F})$ be a weakly $\alpha$-compatible pair, and let $J$ be a nonempty face of $Q_{\mathcal{Z}}$ such that $\sigma_M \cap \tau_J \neq \{ 0 \}$. Then $\{ F_s : Z_s \in J \}$ is the set of elements of a chain of faces in $\Gamma$. Moreover, if $\sigma_K^\circ \cap \tau_J \neq \emptyset$ for some $K \in \Contrib(\alpha)_M$, then $K \subset F_s$ for all $s \in \mathcal{S}$ such that $Z_s \in J$.  

\end{lemma}
\begin{proof}
There exists $K \in \Contrib(\alpha)_M$ such that $\sigma_K^\circ \cap \tau_J \neq \emptyset$. Hence, for any 
$s,s' \in \mathcal{S}$ such that $Z_s,Z_{s'} \in J$,
$\sigma_K^\circ \cap \tau_{Z_s} \cap \tau_{Z_{s'}}  \ne \emptyset$. By property \eqref{i:weak2} of Definition~\ref{d:weaklys0compatible}, $F_s$ and $F_{s'}$  are comparable under inclusion. Hence if we consider the set $\{ F_s : Z_s \in J \}$ as a poset under inclusion, then all elements are comparable, and the poset is a chain. The second statement follows from property \eqref{i:weak1} of Definition~\ref{d:weaklys0compatible}. 
\end{proof}

\begin{definition}\label{d:max}
Let $(\mathcal{Z}, \mathcal{F})$ be a weakly $\alpha$-compatible pair, and let $J$ be a nonempty face of $Q_{\mathcal{Z}}$ such that $\sigma_M \cap \tau_J \neq \{ 0 \}$. Then set
$F_J  := \max \{ F_s : Z_s \in J \}$. 

\end{definition}

Lemma~\ref{l:maxelem} implies that the above is well-defined.
Also, it follows from 
Lemma~\ref{l:intersectM2} that we can replace the condition $\sigma_M \cap \tau_J \neq \{ 0 \}$ with the condition that 
$\tau_J \cap N_{M,\leq \delta} \neq \{ 0 \}$ for some $\delta$ chosen sufficiently small.

\begin{remark}\label{r:GinFJ}
Let $J$ be a nonempty face of   
$Q_{\mathcal{Z}}$ and $K \in \Contrib(\alpha)_M$ such that 
$\sigma_K^\circ \cap \tau_J \neq \emptyset$. 
Then Lemma~\ref{l:maxelem} implies that $K \subset F_J$. 
\end{remark}

\begin{remark}\label{r:Jinclusion}
If $J \subset J'$ is an inclusion of nonempty faces of 
$Q_{\mathcal{Z}}$ and $\sigma_M \cap \tau_{J'} \neq \{ 0 \}$, then  $\sigma_M \cap \tau_{J} \neq \{ 0 \}$ and 
$F_{J} \subset F_{J'}$. 
\end{remark}

\begin{lemma}\label{l:NMGinFJ}
Let $(\mathcal{Z}, \mathcal{F})$ be a weakly $\alpha$-compatible pair, and
let $J$ be a nonempty face of   
$Q_{\mathcal{Z}}$ and $K \in \Gamma$ such that 
$\sigma_K \cap \tau_J \cap N_{M,\leq \delta} \neq \{ 0 \}$ for some $\delta$ sufficiently small. 
Then $K \subset F_J$. 
\end{lemma}
\begin{proof}
Since $\sigma_K \cap \tau_J \cap N_{M,\leq \delta} \neq \{ 0 \}$,  Lemma~\ref{l:intersectM2} implies that 
$\sigma_K \cap \tau_J \cap \sigma_M \neq \{ 0 \}$. 
In particular, $F_J$ is well-defined. 
Since $\sigma_K \cap \sigma_M  \neq \{0\}$, we have that 
$\sigma_K \cap \sigma_M = \sigma_{K'}$ for some $K' \in \Contrib(\alpha)_M$. 
Since  $\sigma_{K'} \cap \tau_J \neq \{ 0 \}$,   there exists $K' \subset K'' \in \Contrib(\alpha)_M$ such that $\sigma_{K''}^\circ \cap \tau_J \neq \emptyset$. By Remark~\ref{r:GinFJ}, $K'' \subset F_J$. 
Since $K \subset K' \subset K''$, the result follows. 
\end{proof}

Let $\Sigma_1,\Sigma_2$ be 
fans in $\R^n$ dual to polyhedra $P_1,P_2$ in $\R^n$ respectively. Then the Minkowski sum $P_1 + P_2$ is dual to the intersection $\Sigma_1 \cap \Sigma_2$ of $\Sigma_1$ and $\Sigma_2$, where $\Sigma_1 \cap \Sigma_2$ is the fan consisting of all cones $\{ \sigma_1 \cap \sigma_2 : \sigma_i \in \Sigma_i \}$. All faces of $P_1 + P_2$ have the form $J_1 + J_2$ for some faces $J_i$ of $P_i$ for $i = 1,2$. 
If $J_i$ is dual to $\sigma_i$ in $\Sigma_i$ for $i = 1,2$, then a face of the form $J_1 + J_2$ is dual to $\sigma_1 \cap \sigma_2$. 
Conversely, every cone $C$ in $\Sigma_1 \cap \Sigma_2$ has the form $C = \sigma_1 \cap \sigma_2$, where $\sigma_i$ is the smallest face of $\Sigma_i$ containing $C$ for $i = 1,2$. Then $C^\circ = \sigma_1^\circ \cap \sigma_2^\circ$, 
$\Span(C) = \Span(\sigma_1) \cap \Span(\sigma_2)$, 
and if $\sigma_i$ is dual to a face $J_i$ of $P_i$ for $i = 1,2$, then $C$ is dual to 
$J_1 + J_2$. 
We will be interested in the polyhedron 
$\Newt(f)_{\mathcal{Z}} := \Newt(f) + Q_\mathcal{Z}$ 
dual to $\Sigma \cap \Sigma_{\mathcal{Z}}$.

\begin{definition}
Let $(\mathcal{Z}, \mathcal{F})$ be a weakly $\alpha$-compatible pair, let $J$ be a nonempty face of $Q_{\mathcal{Z}}$, and let $C$  be a cone in $\Sigma \cap \Sigma_{\mathcal{Z}}$. Assume that  
$C \subset \tau_J$ 
and 
$C \cap N_{M,\leq \delta} \neq \{ 0 \}$ for some $\delta$ chosen sufficiently small. 
Let 
$
D(C,J) = D(C,J, M,\delta) :=  
(C \cap \sigma_{A_{F_J}} \cap N_{M,\leq \delta}) +  \mathbb{R}_{\ge 0} e^*_{F_J}.
$
\end{definition}

The following lemma will allow us to replace contributions to $Z_{\for}(T)$ from certain cones $C$ by contributions from cones $D(C, J)$, whose structure will allow us to apply Lemma~\ref{lem:cancellation}. 

\begin{lemma}\label{l:Cinvariant}
Let $(\mathcal{Z}, \mathcal{F})$ be a weakly $\alpha$-compatible pair, let $J$ be a nonempty face of $Q_{\mathcal{Z}}$, and let $C$  be a cone in $\Sigma \cap \Sigma_{\mathcal{Z}}$. Assume that  
$C \subset \tau_J$ and that $C \cap N_{M,\leq \delta} \neq \{ 0 \}$ for some $\delta$ chosen sufficiently small. 
Then $C$ is dual to a face $K + J'$ of $\Newt(f)_{\mathcal{Z}}$, for some 
face
$K \in \Gamma$ such that $K  \subset F_J \subset F_{J'}$, and 
for some face $J'$  of $Q_{\mathcal{Z}}$ such that $J \subset J'$. 
Suppose that $C \not \subset \sigma_{A_{F_J}}$.
Then $C \cap N_{M,\leq \delta} = 
D(C,J) \cap N_{M,\leq \delta}$, 
and $e^*_{F_J} \in \Span(C)$.  
\end{lemma}
\begin{proof}
By Lemma~\ref{l:notinGamma}, 
$C$ is dual to 
a face of $\Newt(f)_{\mathcal{Z}}$
of the form $K + J'$, for some  $K \in \Gamma$, and some nonempty face $J'$ of $Q_{\mathcal{Z}}$. Here  $C = \sigma_{K} \cap \tau_{J'}$ and $\tau_{J'}$ is the smallest cone in $\Sigma_{\mathcal{Z}}$ containing $C$.  
In particular,
$\tau_{J'} \subset \tau_J$ and $J \subset J'$. 
By Lemma~\ref{l:intersectM2}, $\sigma_M \cap \tau_{J'}  \neq \{ 0 \}$, and $F_J,F_{J'}$ are well-defined.
Also,
$\{ 0 \} \neq  \sigma_{K} \cap \tau_{J} \cap N_{M,\leq \delta}$.  Then Lemma~\ref{l:NMGinFJ} implies that $K \subset F_{J}$. 
By Remark~\ref{r:Jinclusion}, $F_J \subset F_{J'}$. 

Fix a vertex $V$ of $K$. Then 
\begin{equation}\label{e:sigmaG}
\sigma_K = \bigg( \bigcap_{V \neq V' \in K} \mathrm{H}_{V' - V} \bigg) \cap \bigg( \bigcap_{V' \notin K} \mathrm{H}_{V' - V, \ge 0} \bigg) \cap \bigg( \bigcap_{W \in \Unb(C_K)} \mathrm{H}_W  \bigg) \cap \bigg( \bigcap_{W \in \{ e_1,\ldots,e_n \} \smallsetminus \Unb(C_K)} \mathrm{H}_{W,\ge 0}  \bigg),
\end{equation}
where $V'$ varies over vertices in $\Gamma$. 
Let $\hat{s} \in \mathcal{S}$ such that $Z_{\hat{s}} \in J$ and 
$F_{\hat{s}} = F_J$. Then $Z_{\hat{s}} \in J'$ and 
\begin{equation}\label{e:tauJ'}
\tau_{J'} =  \Bigg( \bigcap_{\substack{\hat{s} \neq s \in S \\ Z_s \in J'}} \mathrm{H}_{Z_s - Z_{\hat{s}}}  \Bigg) \cap \Bigg( \bigcap_{\substack{s \in S \\ Z_s \notin J'}} \mathrm{H}_{Z_s - Z_{\hat{s}}, \ge 0}  \Bigg).
\end{equation}
It follows that 
we may choose a finite collection of nonzero elements $P_C = \{  W \} \subset \Q^n \smallsetminus \{0\}$ 
such that each $W \in P_C$ is of the form 
\begin{enumerate}
\item $W = Z_s - Z_{\hat{s}}$ for some $s \in \mathcal{S}$,
\item $W = V' - V$ for some vertex $V'$  in $\Gamma$, or 
\item $W \in \{ e_1,\ldots,e_n \}$,
\end{enumerate}
and $C = \sigma_K \cap \tau_{J'}$ is the intersection of half-spaces of the form $\mathrm{H}_{W, \ge 0}$ or hyperplanes of the form $\mathrm{H}_W$ for various $W \in P_C$. 

Suppose that $C \not \subset \sigma_{A_{F_J}}$.
Note that $C \not \subset \sigma_{A_{F_J}}$ if and only if $\sigma_K \not \subset \sigma_{A_{F_J}}$, if and only if $A_{F_J} \notin K$.
Let $\tilde{\sigma}$ be the intersection of all such half-spaces and hyperplanes appearing in
the description \eqref{e:sigmaG} of $\sigma_K$ 
such that the $W \in P_C$ that defines the hyperplane or half-space has $C \cap N_{M,\leq \delta} \cap \mathrm{H}_W \neq  \{ 0 \}$, 
and $W$ doesn't have the form $W = V' - V$ with $V' = A_{F_J}$. Similarly, let $\tilde{\tau}$ be the intersection of all such half-spaces and hyperplanes 
appearing in
the description \eqref{e:tauJ'} of $\tau_{J'}$ such that $C \cap N_{M,\leq \delta} \cap \mathrm{H}_W \neq  \{ 0 \}$. Let $\tilde{C} = \tilde{\sigma} \cap \tilde{\tau}$. 
Then, by construction, 
there exists a cone $U$ over a small open neighborhood of $N_{M,\leq \delta} \cap S'$ in $S'  = \{ u \in \R^n : \langle u , \One \rangle = 1 \}$ 
such that 
\begin{equation}\label{e:Cinit}
C \cap U = \tilde{C} \cap \mathrm{H}_{A_{F_J} - V, \ge 0} \cap  U.
\end{equation}

We claim that $\R e^*_{F_J} \subset \tilde{C}$. 
Let $\mathrm{H}_{W,\ge 0}$ or $\mathrm{H}_W$ be a defining half-space or hyperplane of $\tilde{C}$. We need to show that $e^*_{F_J} \in \mathrm{H}_W$. Equivalently, we need to show that 
$\langle e^*_{F_J}, W \rangle = 0$. 
By assumption, $C \cap N_{M,\leq \delta} \cap \mathrm{H}_W \neq \{ 0 \}$. By Lemma~\ref{l:intersectM2}, $C \cap  \sigma_M \cap \mathrm{H}_W \neq \{ 0 \}$.

First, assume that $W = V' - V$ for some vertex $V' \neq A_{F_J}$ in $\Gamma$.
Then $\{ 0 \} \neq C \cap N_{M,\leq \delta} \cap \mathrm{H}_W \subset\sigma_{V'} \cap \tau_J  \cap N_{M,\leq \delta}$, and 
Lemma~\ref{l:NMGinFJ} implies that $V' \in F_J$. 
Then $V',V$ are vertices of $F_J$ that are not equal to 
$A_{F_J}$, and hence $\langle e^*_{F_J}, V' \rangle =
\langle e^*_{F_J}, V \rangle = 0$, so 
$\langle e^*_{F_J}, W \rangle = 0$. 

Second, assume that  $W =  Z_{s} - Z_{\hat{s}}$ for some $s$ in $\mathcal{S}$. Since $C \subset \tau_{J'}$,
we have  $\{ 0 \} \neq C \cap  \sigma_M \cap \mathrm{H}_W \subset \tau_{J'} \cap  \sigma_M \cap \mathrm{H}_W \subset \sigma_M \cap \tau_{Z_s} \cap \tau_{Z_{\hat{s}}} $.
Hence, there exists  $K' \in \Contrib(\alpha)_M$ such that 
$\sigma_{K'}^\circ \cap \tau_{Z_s} \cap \tau_{Z_{\hat{s}}}  \neq \emptyset$.
By \eqref{i:weak3} in Definition~\ref{d:weaklys0compatible}, 
$\langle e^*_{F_J}, Z_s \rangle =
\langle e^*_{F_J}, Z_{\hat{s}} \rangle = 0$, so 
$\langle e^*_{F_J}, W \rangle = 0$. 

Finally, assume 
$W \in \{ e_1,\ldots,e_n \}$. 
Then $\{ 0 \} \neq C \cap N_{M,\leq \delta} \cap \mathrm{H}_W \subset \sigma_{K}  \cap \tau_J  \cap N_{M,\leq \delta} \cap \mathrm{H}_W$. 
By Lemma~\ref{l:intersectM2}, 
$\sigma_{K} \cap \tau_J \cap \sigma_M \cap \mathrm{H}_W \neq \{ 0 \}$.
It follows that there exists $K' \in \Contrib(\alpha)_M$ such that 
$K \subset K'$ and $\sigma_{K'}^\circ \cap \tau_J \cap \mathrm{H}_W \neq \{ 0 \}$. 
By Remark~\ref{r:GinFJ}, $K' \subset F_J$.  
Then $W \in \Unb(C_{K'}) \subset \Unb(C_{F_J})$, and 
hence $\langle e^*_{F_J}, W \rangle = 0$ by Remark~\ref{r:apexisbounded}.

We conclude that $\R e^*_{F_J} \subset \tilde{C}$. 
Next, we claim that 
\begin{equation}\label{e:Ctilde}
\tilde{C} \cap \mathrm{H}_{A_{F_J} - V, \ge 0} \cap  N_{M,\leq \delta} = 
((\tilde{C} \cap \mathrm{H}_{A_{F_J} - V} \cap  N_{M,\leq \delta}) + \mathbb{R}_{\ge 0} e^*_{F_J}) \cap  N_{M,\leq \delta}.
\end{equation}

By \eqref{e:Cinit}, the left-hand side of \eqref{e:Ctilde} is $C \cap  N_{M,\leq \delta}$.
Let $u \in \tilde{C} \cap \mathrm{H}_{A_{F_J} - V, \ge 0} \cap  N_{M,\leq \delta}$. We aim to show that $u$ lies in the right-hand side of \eqref{e:Ctilde}. It is enough to consider the case when $\langle u , \One \rangle = 1$. 
Consider the function
\[
\phi \colon \R_{\ge 1} \to S' \subset \R^n, \text{ defined by }
\phi(\lambda) = \lambda u + (1 - \lambda) e^*_{F_J}. 
\]
Since $\R e^*_{F_J} \subset \tilde{C}$,  the image of $\phi$ is contained in $\tilde{C}$.
It is enough to show that $\phi^{-1}(\mathrm{H}_{A_{F_J} - V} \cap  N_{M,\leq \delta}) \neq \emptyset$, since if 
$\lambda \in \phi^{-1}(\mathrm{H}_{A_{F_J} - V} \cap  N_{M,\leq \delta})$, then 
\begin{equation}\label{e:phialphaeFJ}
u = (1/\lambda)(\phi(\lambda) + 
(\lambda - 1) e^*_{F_J}),
\end{equation}
and $u$ lies in the right-hand side of \eqref{e:Ctilde}.
Moreover, if we choose $u \in C^\circ$, then $u \notin \sigma_{A_{F_J}}$ and hence $\lambda > 1$ in \eqref{e:phialphaeFJ}. Then 
$e^*_{F_J} = (1/(\lambda - 1)) ( \lambda u - \phi(\lambda)) \in \Span(C)$, which establishes the last statement of the lemma.

Consider the linear function\[
f(\lambda) = \langle \phi(\lambda), W_M - V \rangle. 
\]
Since $u \in S' \cap N_{M,\leq \delta}$, $f(1) \le \delta$. 
Since $\langle e^*_{F_J}, V \rangle = 0$, 
we compute:
$f'(\lambda) = f(1) - 
\langle e^*_{F_J} , W_M  \rangle \le \delta - \langle e^*_{F_J} , W_M  \rangle < 0$. The last inequality follows since $M$ is interior implies that  
$\langle e^*_{F_J} , W_M  \rangle > 0$ and $\delta$ is chosen sufficiently small.  In particular, 
the image of  $f$ is unbounded because it is a non-constant linear function. Hence the image of $\phi$ is unbounded and so 
$\phi^{-1}(U \smallsetminus N_{M, \leq \delta}) \neq \emptyset$. 

We claim that $\phi^{-1}(\mathrm{H}_{A_{F_J} - V, \ge 0} \cap U) \subset \phi^{-1}(N_{M,\leq \delta})$. Indeed, if $\phi(\lambda) \in \mathrm{H}_{A_{F_J} - V, \ge 0} \cap U$, then $\phi(\lambda) \in C$ by \eqref{e:Cinit}. Then 
$
\langle \phi(\lambda), W_M \rangle - N( \phi(\lambda)) = f(\lambda) \le f(1) \le \delta. 
$
Since $\langle \phi(\lambda) , \One \rangle = 1$, we deduce that $\phi(\lambda) \in N_{M,\leq \delta}$. 

It follows that $\emptyset \neq \phi^{-1}(U \smallsetminus N_{M, \leq \delta}) \subset \phi^{-1}(\mathrm{H}_{A_{F_J} - V, < 0} \cap U)$. Since $1 \in \phi^{-1}(\mathrm{H}_{A_{F_J} - V, \ge 0} \cap U)$, we deduce that 
$\phi^{-1}(\mathrm{H}_{A_{F_J} - V} \cap N_{M,\leq \delta}) = \phi^{-1}(\mathrm{H}_{A_{F_J} - V} \cap U) \neq \emptyset$, so the left-hand side is contained in the right-hand side of \eqref{e:Ctilde}.

Conversely, since $\R e^*_{F_J} \subset \tilde{C}$ and $e^*_{F_J} \in \mathrm{H}_{A_{F_J} - V, \ge 0}$, the right-hand side of \eqref{e:Ctilde} is contained in $\tilde{C} \cap \mathrm{H}_{A_{F_J} - V, \ge 0} \cap  N_{M,\leq \delta}$. This establishes \eqref{e:Ctilde}. 
By \eqref{e:Cinit},
\[
\tilde{C} \cap \mathrm{H}_{A_{F_J} - V} \cap  N_{M,\leq \delta} =
(\tilde{C} \cap \mathrm{H}_{A_{F_J} - V,\ge 0} \cap  N_{M,\leq \delta}) \cap \sigma_{A_{F_J}} = (C \cap N_{M,\leq \delta})  \cap \sigma_{A_{F_J}}. 
\]
Substituting this expression into the right-hand side of \eqref{e:Ctilde} and combining with  \eqref{e:Cinit}, we deduce that 
\[
C \cap N_{M,\leq \delta} = \tilde{C} \cap \mathrm{H}_{A_{F_J} - V, \ge 0} \cap  N_{M,\leq \delta} = 
((C \cap \sigma_{A_{F_J}} \cap N_{M,\leq \delta}) +  \mathbb{R}_{\ge 0} e^*_{F_J}) \cap  N_{M,\leq \delta}. \qedhere
\]
\end{proof}

The following lemma is a corollary of the proof of 
Lemma~\ref{l:Cinvariant}. 

\begin{lemma}\label{l:norayinside}
Let $(\mathcal{Z},\mathcal{F})$ be a weakly $\alpha$-compatible pair, and let $J$ be a nonempty face of $Q_{\mathcal{Z}}$ 
such that  $\tau_J \cap N_{M,\leq \delta} \neq \{ 0 \}$ for some $\delta$ chosen sufficiently small. Then no ray of $\tau_J \cap N_{M,\leq \delta}$ is contained in $\sigma_M$. 
\end{lemma}
\begin{proof}
The proof of Lemma~\ref{l:Cinvariant}, with $\sigma_K$ replaced by $\R^n_{\ge 0}$ and $\tau_{J'}$ replaced by $\tau_J$, shows that 
there exists a polyhedral cone $\tau'$ 
and a cone $U$ over a small open neighborhood of $N_{M,\leq \delta} \cap S'$ in $S'$ such that 
\begin{enumerate}
\item $\R e^*_{F_J} \subset \tau'$, and
\item $\R^n_{\ge 0} \cap \tau_J \cap U = \tau' \cap  U.$
\end{enumerate} 
Let $u$ be a generator of a ray in $\tau_J \cap N_{M,\leq \delta}$. We may assume that $\langle u , \One \rangle = 1$. Suppose that $\langle u , W_M \rangle - N(u) < \delta$. Fix $0 < \epsilon \ll 1$ and let $L_\epsilon = \{ u + \lambda e^*_{F_J} : |\lambda| < \epsilon  \}$. Then $L_\epsilon \subset \tau' \cap U = \R^n_{\ge 0} \cap \tau_J \cap U$. It follows that $L_\epsilon \subset  \tau_J \cap N_{M,\leq \delta}$, contradicting the assumption that $u$ generates a ray. 
We deduce that $\langle u , W_M \rangle - N(u) = \delta$. In particular, $u \notin \sigma_M$. 
\end{proof}

\begin{definition}\label{d:s0criticaltau2}
Let $(\mathcal{Z}, \mathcal{F})$ be a weakly $\alpha$-compatible pair, let $J$ be a nonempty face of $Q_{\mathcal{Z}}$, and let $C$ be a cone in $\Sigma \cap \Sigma_{\mathcal{Z}}$.
Then we say $(C,J)$ is \emph{$\alpha$-critical} if 
the following properties hold:
\begin{enumerate}
\item\label{i:CC1} $C \cap \sigma_M \neq \{ 0 \}$,
\item\label{i:CC2} $C \subset  \sigma_{A_{F_J}}  \cap \tau_J$, and
\item\label{i:CC3} $C$ is critical with respect to
$(\alpha,A_{F_J})$.
\end{enumerate}
\end{definition}

\begin{definition}\label{d:compatible3}
We say a weakly $\alpha$-compatible pair $(\mathcal{Z},\mathcal{F})$ is \emph{$\alpha$-compatible} if 
for every $\alpha$-critical pair $(C,J)$, $C \subset \sigma_M$. 
\end{definition}

Note that the notion of an $\alpha$-compatible pair depends on the choice of a minimal face $M$. The main technical result required to prove Theorem~\ref{thm:nopole} is following result on the existence of $\alpha$-compatible pairs. 

\begin{theorem}\label{thm:existence}
Let $\alpha \not \in \Z$, and assume that all faces of $\Contrib(\alpha)$ are $\UB$ and $\Newt(f)$ is $\alpha$-simplicial. Then for any minimal face $M \in \Contrib(\alpha)$, there 
exists an $\alpha$-compatible pair. 
\end{theorem}

\begin{lemma}\label{l:Cnocritrays}
Consider
an $\alpha$-compatible pair $(\mathcal{Z},\mathcal{F})$. Let $\gamma$ be a ray of $\Sigma \cap \Sigma_{\mathcal{Z}} \cap N_{M,\leq \delta}$ for some $\delta$ chosen sufficiently small. Assume that $\gamma \not \subset \sigma_M$. Then $\gamma$ is not critical with respect to $\alpha$. Moreover, if $\gamma \subset \tau_J$ for some face $J$ of $Q_{\mathcal{Z}}$, then $\gamma$ is not critical with respect to $(\alpha,A_{F_J})$. 
\end{lemma}
\begin{proof}
There is a unique face $J'$ of $Q_{\mathcal{Z}}$ such that $\gamma \subset  \tau_{J'}^\circ$. If $\gamma \subset \tau_J$, then it follows that $J \subset J'$. 
Let $C$ be the smallest cone in $\Sigma \cap \Sigma_{\mathcal{Z}}$ containing $\gamma$. Then $\tau_{J'}$ is the smallest cone of $\Sigma_{\mathcal{Z}}$ containing $C$. 
By Lemma~\ref{l:Cinvariant}, $C$ is dual to 
a face $K + J'$ of $\Newt(f)_{\mathcal{Z}}$, where 
$K \in \Gamma$ such that $K 
\subset F_J$. 
Assume that $\gamma$ is critical with respect to $(\alpha,W)$ for some vertex $W$ of either $K$ or $M$. 
Note that 
one possible choice for $W$ is $A_{F_J}$, since $A_{F_J} \in M$. 
Then Lemma~\ref{l:RcriticalC} implies that $C$ is critical with respect to $(\alpha,W)$, and $C \cap \sigma_M \neq \{ 0 \}$. 

First, assume that $\gamma \subset \sigma_{A_{F_J}}$. Then $C \subset \sigma_{A_{F_J}}$. If $W = A_{F_J}$, then $(C,J)$ is $\alpha$-critical, and Definition~\ref{d:compatible3} implies that $\gamma \subset C \subset  \sigma_M$, a contradiction. We conclude that $W \neq A_{F_J}$. That is, $\gamma$ is not critical with respect to
$(\alpha,A_{F_J})$. Equivalently, in this case, $\gamma$ is not critical with respect to   $\alpha$. 

Second, assume that $\gamma \not \subset \sigma_{A_{F_J}}$. Then $C \not \subset \sigma_{A_{F_J}}$.  Equivalently, 
$A_{F_J} \notin K$. 
By
Lemma~\ref{l:Cinvariant},
$e^*_{F_J} \in \Span(C)$. Let $V$ be a vertex of $K \subset F_J$. 
Since $V \neq A_{F_J}$,  $\langle e^*_{F_J}, V \rangle = 0$, and hence $\langle e^*_{F_J}, \alpha V + \One \rangle = 1 \neq 0$. We deduce that $C$ is not critical with respect to $(\alpha,V)$. 
Hence $W \neq V$, and $\gamma$ is not critical with respect to $\alpha$. 
Similarly,  since $\alpha \neq -1$ by assumption,
$\langle e^*_{F_J}, \alpha A_{F_J} + \One \rangle = \alpha + 1 \neq 0$, and hence $C$ is not critical with respect to $(\alpha, A_{F_J})$. Therefore $W \neq A_{F_J}$, and $\gamma$ is not critical with respect to $(\alpha,A_{F_J})$.
\end{proof}

\begin{proof}[Proof of Theorem~\ref{thm:nopole}]
Let $\alpha \not \in \Z$, and assume that all faces of $\Contrib(\alpha)$ are $\UB$ and $\Newt(f)$ is $\alpha$-simplicial.
By Remark~\ref{r:specializecandidate} and Remark~\ref{r:minimalsetcandidate}, it is enough to show that  there exists a set of candidate poles for $Z_{\for}(T)$ not containing $\alpha$. 
Let $M_1,\ldots,M_r$ be the minimal elements in $\Contrib(\alpha)$.
Assume that $\delta$ is chosen sufficiently small.
By Lemma~\ref{l:separateN},
$$
Z_{\for}(T) = Z_{\for}(T)|_{\cap_i  N_{M_i,\geq \delta}} + \sum_{i = 1}^r Z_{\for}(T)|_{N_{M_i,< \delta}}.
$$ 
By Lemma~\ref{l:outerformalzeta},
$Z_{\for}(T)|_{\cap_i  N_{M_i,\geq \delta}}$
lies in $R$ and admits a set of candidate poles not containing $\alpha$. Choose a minimal face $M \in \Contrib(\alpha)$.
By Theorem~\ref{thm:existence},  we may fix an $\alpha$-compatible pair $(\mathcal{Z}, \mathcal{F})$. By definition, $Z_{\for}(T)|_{\partial \R^n_{\ge 0}} = 0$. Recall that $N_{M,\leq \delta}^\circ = N_{M,< \delta} \cap \R^n_{> 0}$. We have 
$$
Z_{\for}(T)|_{N_{M,< \delta}} = Z_{\for}(T)|_{N_{M,\leq \delta}^\circ} =  \sum_{\substack{\emptyset \neq J \subset Q_{\mathcal{Z}} \\ \tau_J \cap N_{M, \leq \delta}^\circ \not= \emptyset}} Z_{\for}(T)|_{\tau_J^\circ \cap N_{M,\leq \delta}^\circ}.
$$

Let $J$ be a nonempty face of $Q_{\mathcal{Z}}$ such that $\tau_J \cap N_{M, \leq \delta}^\circ \not= \emptyset$.
By Remark~\ref{r:unioncandidate}, 
to show that $Z_{\for}(T)$ lies in $R$ and admits a set of candidate poles not containing $\alpha$,
it is enough to show that 
$Z_{\for}(T)|_{\tau_J^\circ \cap N_{M,\leq \delta}^\circ}$ 
lies in $R$ and admits a set of candidate poles not containing $\alpha$ for every $J$.

Let $C$ be a cone in $\Sigma \cap \Sigma_{\mathcal{Z}}$ such that $C^\circ \subset \tau_J^{\circ}$ and  $C \cap N_{M, \leq \delta}^\circ \not= \emptyset$. By Lemma~\ref{l:Cinvariant}, $C = \sigma_G \cap \tau_{J}$ and $C^\circ = \sigma_G^\circ \cap \tau_{J}^\circ$ for some face $G \in \Gamma$ with $G \subset F_J$. 
Since $C \not \subset \partial \R^n_{\ge 0}$,  
$G$ is compact.

Suppose that $C \not \subset \sigma_{A_{F_J}}$. 
Then $A_{F_J} \notin G$, and 
$G$ is contained in the base of the (possibly unbounded) $B_1$-face $F_J$. 
Since $F_J$ is a $B_1$-face, we may
consider the face 
$F  = \Conv{G,A_{F_J}}$ 
of $F_J$. Then $F$ is a compact $B_1$-face with apex $A_{F_J}$ and base $G$ in the direction $e_{F_J}^*$. 
Consider the face $C' = C \cap \sigma_{A_{F_J}} \cap N_{M,\leq \delta}$ of $C \cap N_{M,\leq \delta}$. 
Note that $C' \subset C \cap \sigma_{A_{F_J}} \subset \sigma_F$. 
By Lemma~\ref{l:Cinvariant}, 
\begin{equation}\label{e:localbehaviorC}
C \cap N_{M,\leq \delta} = D \cap N_{M,\leq \delta},
\end{equation}
where   $
D = D(C,J) =  
C' +  \mathbb{R}_{\ge 0} e^*_{F_J}
\subset \sigma_G $.
By 
Lemma~\ref{l:Cnocritrays},
the rays of  $C \cap N_{M, \leq \delta}$ that are critical with respect to either $\alpha$ or $(\alpha,A_{F_J})$ are contained $\sigma_M$. 
If $V$ is a vertex of $G$, then 
$\langle e^*_{F_J}, V \rangle = 0$, and hence $\langle e^*_{F_J}, \alpha V + \One \rangle = 1 \neq 0$.
Also,  $\langle e^*_{F_J}, \alpha A_{F_J} + \One \rangle = \alpha + 1  \neq 0$, by assumption. 
We deduce that $e^*_{F_J}$ is not critical with respect to $\alpha$ or $(\alpha,A_{F_J})$. 
This implies that no ray of $D \cap N_{M,\geq \delta}$ is critical with respect to $\alpha$ or 
$(\alpha,A_{F_J})$.
Since $\sigma_G^\circ,\sigma_F^\circ \subset \R^n_{>0}$, Remark~\ref{r:relativeinteriors} gives the following equalities:
\begin{equation*}
(C \cap \sigma_{A_{F_J}})^\circ \cap N_{M, \leq \delta}^\circ = (C')^\circ,
\end{equation*}
\begin{equation*}
C^\circ \cap N_{M, \leq \delta}^\circ = (C \cap N_{M,\leq \delta})^\circ = (D \cap N_{M,\leq \delta})^\circ = 
D^\circ \cap N_{M, \leq \delta}^\circ, \text{ and}
\end{equation*}
\begin{equation*}
D^\circ \cap N_{M,\ge \delta} = (D \cap N_{M,\geq \delta})^\circ \cup (D \cap N_{M,\delta})^\circ.
\end{equation*}

Since the restriction of $N$ to $\sigma_G$ is linear, 
Lemma~\ref{lem:powerseries} and
Lemma~\ref{l:Npositive}, 
imply that 
the following elements of $\tilde{R}$ lie in $R$ and admit sets of candidate poles not containing $\alpha$:
\begin{equation*}
(L - 1)^{n - \dim G} \sum_{u \in D^\circ \cap N_{M,\ge \delta} \cap \mathbb{N}^{n}} L^{-\langle u, \mathbf{1} \rangle}T^{N(u)}, \text{ and } 
(L - 1)^{n} \sum_{u \in D^\circ \cap N_{M,\ge \delta} \cap \mathbb{N}^{n}} L^{-\langle u, \mathbf{1} \rangle}T^{\langle u , A_{F_J} \rangle}.
\end{equation*}  

We claim that $(C')^{\circ} \subset \sigma_F^{\circ}$. We have $(C')^{\circ} \subset \sigma_{F'}^{\circ}$ for some $F \subset F'$.
By Lemma~\ref{l:NMGinFJ}, $F' \subset F_J$. If $\sigma_{F'} \subset \partial \R^n_{\ge 0}$, then \eqref{e:localbehaviorC} implies that $C \cap N_{M,\le \delta} \subset \partial \R^n_{\ge 0}$, a contradiction. Hence $F'$ is a compact face of $F_J$ containing $A_{F_J}$. 
It follows that $F'$ is a $B_1$-face with apex $A_{F_J}$ and base $G' \supset G$.  Then \eqref{e:localbehaviorC} implies that $C \cap N_{M,\le \delta} \subset \sigma_{G'}$. Then $\emptyset \neq (C \cap N_{M,\leq \delta})^\circ = C^\circ \cap N_{M, \leq \delta}^\circ \subset \sigma_{G'}$. Since $C^\circ \subset \sigma_{G}^\circ$, we conclude that $G = G'$, and hence $F = F'$, completing the proof of the claim. 
We may then apply 
Lemma~\ref{lem:cancellation} to obtain
$$Z_{\for}(T)|_{(D^{\circ}\cup (C')^{\circ})}= (L - 1)^{n} \Bigg(\sum_{u \in (D^{\circ} \cup (C')^{\circ}) \cap \mathbb{N}^n} L^{-\langle u, \mathbf{1} \rangle} T^{\langle u, A_{F_J} \rangle} \Bigg)
\in \tilde{R}.$$  
Since 
$D^\circ = (D^\circ \cap N_{M,\geq \delta}) \cup  
(D^\circ \cap N_{M,\leq \delta}^\circ)$,
using the above calculations and Remark~\ref{r:unioncandidate},  we deduce that 
\begin{equation}\label{e:diffhasnopole}
Z_{\for}(T)|_{((C \cap \sigma_{A_{F_J}})^\circ \cup C^\circ) \cap N_{M, \leq \delta}^\circ} - 
(L - 1)^{n} \Bigg(\sum_{u \in (((C \cap \sigma_{A_{F_J}})^\circ \cup C^\circ) \cap N_{M,\leq \delta}^\circ) \cap \mathbb{N}^n} L^{-\langle u, \mathbf{1} \rangle} T^{\langle u, A_{F_J} \rangle} \Bigg)
\end{equation}
lies in $R$ and admits a set of candidate poles not containing $\alpha$. 

Since $\Sigma \cap \Sigma_{\mathcal{Z}}$ refines $\Sigma_{\mathcal{Z}}$, we may consider the subfan $\Sigma \cap \Sigma_{\mathcal{Z}} \cap N_{M,\leq \delta}|_{\tau_J}$ of $\Sigma \cap \Sigma_{\mathcal{Z}} \cap N_{M,\leq \delta}$. It follows from the description of $\Sigma \cap \Sigma_{\mathcal{Z}} \cap N_{M,\leq \delta}$ in 
Lemma~\ref{l:DeltacapN} and the fact that $\tau_J \cap \R^n_{> 0} \neq \emptyset$ that 
\begin{equation}\label{e:tauJinterior}
(\tau_J \cap N_{M,\le \delta})^\circ = 
\bigcup_{\substack{C \in \Sigma \cap \Sigma_{\mathcal{Z}} \\ C^\circ \subset \tau_J^\circ \\ C \cap N_{M, \leq \delta}^\circ \not= \emptyset}} (C \cap N_{M,\le \delta})^\circ  = 
\bigcup_{\substack{C \in \Sigma \cap \Sigma_{\mathcal{Z}} \\ C^\circ \subset \tau_J^\circ \\ C \cap N_{M, \leq \delta}^\circ \not= \emptyset}} C^\circ \cap N_{M,\leq \delta}^\circ = 
\tau_J^\circ \cap N_{M,\leq \delta}^\circ.
\end{equation}
It follows from \eqref{e:localbehaviorC} that we may rewrite this as: 
\[
\tau_J^\circ \cap N_{M, \leq \delta}^\circ = (\sigma_{A_{F_J}}^{\circ} \cap \tau_J^\circ \cap N_{M, \leq \delta}^\circ) \cup 
\bigcup_{\substack{C \in \Sigma \cap \Sigma_{\mathcal{Z}} \\ C^\circ \subset \tau_J^\circ \\ C \cap N_{M, \leq \delta}^\circ \not= \emptyset \\ C \not \subset \sigma_{A_{F_J}}}}
((C \cap \sigma_{A_{F_J}})^\circ \cup C^\circ) \cap N_{M, \leq \delta}^\circ.
\]
We deduce from Remark~\ref{r:manipulateapex} and \eqref{e:diffhasnopole} that 
$$Z_{\for}(T)|_{\tau_J^\circ \cap N_{M, \leq \delta}^\circ}
- (L - 1)^n \Bigg(\sum_{u \in \tau_J^\circ \cap N_{M,\le \delta}^\circ \cap \mathbb{N}^n} L^{-\langle u, \mathbf{1} \rangle} T^{\langle u, A_{F_J} \rangle} \Bigg)
$$
lies in $R$ and admits a set of candidate poles not containing $\alpha$. By Remark~\ref{r:unioncandidate} and \eqref{e:tauJinterior}, it is enough to show that 
$$
(L - 1)^n \Bigg(\sum_{u \in (\tau_J \cap N_{M,\le \delta})^\circ \cap \mathbb{N}^n} L^{-\langle u, \mathbf{1} \rangle} T^{\langle u, A_{F_J} \rangle} \Bigg)
$$
lies in $R$ and admits a set of candidate poles not containing $\alpha$.
By Lemma~\ref{l:norayinside} and Lemma~\ref{l:Cnocritrays}, no rays of $\tau_J \cap N_{M,\leq \delta}$  are critical with respect to $(\alpha, A_{F_J})$. The result now follows from Lemma~\ref{lem:powerseries} and
Lemma~\ref{l:Npositive}.
\end{proof}

\subsection{Existence of $\alpha$-compatible sets}\label{ss:exist1}

We continue with the notation of the previous section. Recall that we consider
pairs $(\mathcal{Z},\mathcal{F})$, where 
$\mathcal{Z} = (Z_s)_{s \in \mathcal{S}}$ and $\mathcal{F} = (F_s)_{s \in \mathcal{S}}$ are collections of elements of $\Q^n$ and $\Contrib(\alpha)_M$ respectively.

\begin{definition}\label{d:restricted2}
Consider a pair $(\mathcal{Z},\mathcal{F})$. Then $(\mathcal{Z},\mathcal{F})$ is 
\emph{restricted} if
$Z_s \in \Span(\{ V_M \} \cup \gens(C_{F_s} \smallsetminus C_M) \cup \mathcal{A}_M)$ for every $s \in \mathcal{S}$. 

\end{definition}

Our goal is to reduce the existence of an $\alpha$-compatible set to the existence of a restricted, weakly 
$\alpha$-compatible set. 
We will consider sets $\epsilon = \{ \epsilon_{s} \}_{s \in \mathcal{S} } \in \R^{\mathcal{S}}$; note that we allow $\epsilon_s$ to be negative. 
We say that $\epsilon$ is chosen to be sufficiently 
small if 
$|\epsilon_{s}|$ is chosen to be sufficiently 
small for all 
$s \in \mathcal{S}$. Explicitly, a property holds for $\epsilon$ sufficiently small if there exists $\delta > 0$ such that the property holds for all $\epsilon$ such that 
$|\epsilon_{s}| < \delta$ for all $s \in \mathcal{S}$.
Given a sequence of sets $\{ \epsilon_m \}_{m \in \Z_{\ge 0}}$, where $\epsilon_m = \{ \epsilon_{m,s} \}_{s \in \mathcal{S}}$, for some  $\epsilon_{m,s} \in \R$, we write $\lim_{m \to \infty} \epsilon_m = 0$ if   
$\lim_{m \to \infty} \epsilon_{m,s}  = 0$ for all 
$s \in \mathcal{S}$.
Given a set $\mathcal{Z} = \{ Z_{s} \}_{s \in \mathcal{S} }$ of elements in $\R^n$, we let 
$Z_{s}(\epsilon_{s}) :=  Z_{s} + \epsilon_{s}V_M$, and  let $\mathcal{Z}(\epsilon):= \{ Z_{s}(\epsilon_{s}) \}_{s \in \mathcal{S} }$. 
Then $Q_{\mathcal{Z}(\epsilon)}$ is the convex hull of the elements of $\mathcal{Z}(\epsilon)$, and is dual to the fan $\Sigma_{\mathcal{Z}(\epsilon)}$.
Also, $\Newt(f)_{\mathcal{Z}(\epsilon)} = \Newt(f) + Q_{\mathcal{Z}(\epsilon)}$ is dual to $\Sigma \cap \Sigma_{\mathcal{Z}(\epsilon)}$.

\begin{lemma}\label{l:deformZ}
Consider
a 
pair $(\mathcal{Z},\mathcal{F})$.
For 
$\epsilon \in \Q^{\mathcal{S}}$  sufficiently small, 
$(\mathcal{Z}(\epsilon),\mathcal{F})$
is 
restricted if $(\mathcal{Z},\mathcal{F})$ is restricted,
and  $(\mathcal{Z}(\epsilon),\mathcal{F})$
is 
weakly $\alpha$-compatible if $(\mathcal{Z},\mathcal{F})$ is weakly $\alpha$-compatible.
\end{lemma}

\begin{proof}
Assume that $(\mathcal{Z},\mathcal{F})$ is restricted.
Since $Z_{s}(\epsilon_{s})$ is a linear combination of  $Z_{s}$ and $V_M$, it follows that $(\mathcal{Z}( \epsilon ), \mathcal{F})$ is restricted. 

Assume that $(\mathcal{Z},\mathcal{F})$ is weakly $\alpha$-compatible.
We want to show that  $(\mathcal{Z}( \epsilon ), \mathcal{F})$ is  weakly $\alpha$-compatible.
Fix a face $K \in \Contrib(\alpha)_M$ 
and $s, s' \in \mathcal{S}$. There is nothing to show if, after possibly shrinking $\epsilon$,   
$\sigma_K^\circ  \cap \tau_{\mathcal{Z}(\epsilon),s} \cap \tau_{\mathcal{Z}(\epsilon),s'}  = \emptyset$. 
Hence we may assume that there exists a sequence of sets $\{ \epsilon_m \}_{m \in \Z_{\ge 0}}$ 
such that $\lim_{m \to \infty} \epsilon_m = 0$ and $\sigma_K^\circ \cap \tau_{\mathcal{Z}(\epsilon_m),s} \cap \tau_{\mathcal{Z}(\epsilon_m),s'} \ne \emptyset$. Then Bolzano--Weierstrass implies that $\sigma_K \cap \tau_{\mathcal{Z},s} \cap \tau_{\mathcal{Z},s'} \neq \{ 0 \}$.  
Hence there exists $K \subset K'$ such that $\sigma_{K'}^\circ \cap \tau_{\mathcal{Z},s} \cap \tau_{\mathcal{Z},s'} \neq \emptyset$.
Since $\mathcal{Z}$ is weakly $\alpha$-compatible, we deduce that
$K' \subset F_s$, 
either $F_{s'} \subset F_s$ or 
$F_s \subset F_{s'}$,
and 
$
\langle e^*_{F_s}, Z_{s} \rangle =
\langle e^*_{F_s}, Z_{s'} \rangle = 0$.  Then $K \subset K' \subset F_s$. By \eqref{e:VM},  
$
\langle e^*_{F_s}, Z_{s}(\epsilon_{s}) \rangle =
\langle e^*_{F_s}, Z_{s} \rangle = 0$ and 
$
\langle e^*_{F_s}, Z_{s'}(\epsilon_{s'}) \rangle =
\langle e^*_{F_s}, Z_{s'} \rangle = 0$.
\end{proof}

Before proceeding, we
need a series of basic lemmas on deforming polyhedra.
Let $\rec(P)$ denote the recession cone of a polyhedron $P$. Let $\sigma^\vee$ denote the dual cone to a cone $\sigma$. Given a face $K$ of $P$, let $\tau_K$ denote the corresponding cone in the dual fan to $P$. 
In particular, $\rec(K)$ is a face of $\rec(P)$, 
$\tau_{\rec(K)}$ is a face of $\rec(P)^\vee$, and $\tau_K^\circ \subset \tau_{\rec(K)}^\circ$. 

Fix a nonempty finite set $T$. 
Let $P = \Conv{V_t : t \in T} + \sigma  \subset \R^n$ be a polyhedron, for some $V_t \in \R^n$, and some pointed (polyhedral) recession cone $\sigma = \rec(P)$.    
Let $\{ P(\epsilon) = \Conv{V_t(\epsilon) : t \in T} + \sigma \}_\epsilon$
be a set of polyhedra with the same recession fan indexed by $\epsilon \in \R^{\ell}$ for some $\ell \ge 1$. 
Assume  that 
$V_t(\epsilon) \in \R^n$ is a continuous function of 
$\epsilon \in \R^{\ell}$, 
for  all $t$ in $T$. 
If $J(\epsilon)$ is a nonempty face of $P(\epsilon)$ and 
$J$ is a nonempty face of $P$,  we write 
$T(J(\epsilon)) := \{ t \in T : V_t(\epsilon) \in J(\epsilon) \}$ and  
$T(J) := \{ t \in T  : V_t \in J \}
$. We may also consider the recession cones $\rec(J(\epsilon))$ and $\rec(J)$, which are both faces of $\sigma$.

\begin{definition}\label{d:refine}
For fixed $\epsilon \in \R^{\ell}$, we say $P(\epsilon)$ \emph{refines} $P$ if 
for any proper nonempty face $J(\epsilon)$ of $P(\epsilon)$, there exists a proper nonempty face $J$ of $P$ such that  $T(J(\epsilon)) \subset T(J)$ and $\rec(J(\epsilon)) \subset \rec(J)$. 
\end{definition}

\begin{lemma}\label{l:refine}
For $\epsilon \in \R^{\ell}$ sufficiently small, $P(\epsilon)$ refines $P$. 
\end{lemma}
\begin{proof}
Assume the conclusion fails. Then there exists 
a sequence $\{ \epsilon_m \}_{m \in \Z_{\ge 0}}$ such that $\lim_{m \to \infty} \epsilon_m = 0$, and a sequence of proper nonempty faces $J(\epsilon_m)$ of $P(\epsilon_m)$, such that, for any $m$ and any proper nonempty face $J$ of $P$, either   $T(J(\epsilon_m)) \not \subset T(J)$ or $\rec(J(\epsilon_m)) \not \subset \rec(J)$. 
Since $T$ 
is finite and $\sigma$ has finitely many faces, after possibly replacing 
$\{ \epsilon_m \}_{m \in \Z_{\ge 0}}$ by a subsequence, we may assume that $T(J(\epsilon_m))$ and $\rec(J(\epsilon_m))$
are independent of $m$. Denote these by $\tilde{T} = T(J(\epsilon_m))$ and $\tilde{R} = \rec(J(\epsilon_m))$ respectively.  
Consider a sequence $u_m$ of elements in 
$\tau_{J(\epsilon_m)}^\circ \subset  \tau_{\tilde{R}}^\circ$ 
such that $||u_m|| = 1$.  
After possibly replacing $\{ \epsilon_m \}_{m \in \Z_{\ge 0}}$ by a subsequence, we may  assume that $\lim_{m \to \infty} u_m = u \in \tau_J^\circ \subset  \tau_{\rec(J)}^\circ$ exists for some nonempty face $J$ of $P$.  
Since $\tau_{\tilde{R}}$ is closed,
$u \in \tau_{\tilde{R}}$, and hence $\tau_{\rec(J)} \subset \tau_{\tilde{R}}$ and $\tilde{R}
\subset \rec(J)$. 
For any $t \in \tilde{T}$ and $t' \in T$, 
\begin{equation}\label{e:limit}
\langle u , V_{t} \rangle = \lim_{m \to \infty} \langle u_m , V_t \rangle = \lim_{m \to \infty} \langle u_m , V_t(\epsilon_m) \rangle \le \lim_{m \to \infty} \langle u_m , V_{t'}(\epsilon_m) \rangle = \lim_{m \to \infty} \langle u_m , V_{t'} \rangle = \langle u , V_{t'} \rangle,
\end{equation}
and hence $V_t \in J$. 
We deduce that $\tilde{T} \subset T(J)$, a contradiction.
\end{proof}

\begin{definition}\label{d:combconstant}
We say that $\{ P(\epsilon) \}_\epsilon$ is 
\emph{locally combinatorially constant} if for any $\epsilon$ sufficiently small, 
and for any 
nonempty face $J(\epsilon)$  of $P(\epsilon)$, 
there exists a (unique) 
nonempty face $J$ of $P$ such that $T(J(\epsilon))  = T(J)$ and $\rec(J(\epsilon))  = \rec(J)$,
and, moreover, every 
nonempty face $J$ of $P$ appears in this way. 
\end{definition}

\begin{lemma}\label{l:combconstant}
After possibly replacing $P$ with $P(\epsilon)$ for some $\epsilon \in \Q^{\ell}$, $\{ P(\epsilon) \}_\epsilon$ is locally combinatorially constant.
\end{lemma}
\begin{proof}
Lemma~\ref{l:refine} implies that 
we may order $\{ P(\epsilon) \}_\epsilon$ by refinement. 
Since $T$ 
is finite and $\sigma$ has finitely many faces, there exists an $\epsilon \in \R^{\ell}$ such that $Q = P(\epsilon)$ is minimal under this ordering.  
Then $\{ Q(\epsilon) \}_\epsilon$ is locally combinatorially constant. Consider $\epsilon' \in \R^{\ell}$ such that $\epsilon + \epsilon' \in \Q^{\ell}$, and let $Q' = Q(\epsilon')$. Then for $\epsilon'$ sufficiently small,
$\{ Q'(\epsilon) \}_\epsilon$ is locally combinatorially constant.
\end{proof}

Given a family of cones $\{C_k\}_{k=1}^{\infty}$, define $\limsup C_k$ to be the cone of points $u \in \mathbb{R}^n$ such that $u$ is a limit point of
a sequence of points $u_k \in C_k$, i.e., 
there exists a subsequence of $u_k$ converging to $u$.

\begin{lemma}\label{l:convergence} 
Assume that $\{ P(\epsilon) \}_\epsilon$ is locally combinatorially constant. 
Fix a 
nonempty face $J$ of $P$.
For $\epsilon$ sufficiently small, let 
$J(\epsilon)$ be the 
nonempty face of $P(\epsilon)$ such that $T(J(\epsilon))  = T(J)$ and $\rec(J(\epsilon))  = \rec(J)$. 
Consider any $\{ \epsilon_k \}_{k \in \Z_{\ge 0}}$ such that $\lim_{m \to \infty} \epsilon_k  = 0$.
Then $\limsup_k \tau_{J(\epsilon_k)} \subset \tau_J$, 
and, if we assume that $\dim P = n$, then 
$\limsup_k \tau_{J(\epsilon_k)} = \tau_J$.
\end{lemma}
\begin{proof}
We first show that $\limsup_k \tau_{J(\epsilon_k)} \subset \tau_J$.
Suppose that $u_k \in \tau_{J(\epsilon_k)}$, and that, after possibly replacing $\{ \epsilon_k \}_{k \in \Z_{\ge 0}}$ with a subsequence,  $\lim_{k \to \infty} u_k = u \in \R^n$ exists. Then $u \in \tau_{J'}^\circ$  for some 
nonempty face $J'$ of $P$. Then the calculation in \eqref{e:limit} implies that $T(J) = T(J(\epsilon_k)) \subset T(J')$, and hence $J \subset J'$ and $u \in \tau_{J'} \subset \tau_J$. 

We need to prove the converse statement.  
Assume that 
$\dim P = n$. 
Then $\tau_J$ is generated by its rays 
$\{ \gamma_m \}_{1 \le m \le p}$. 
For any $\epsilon$ sufficiently small, let 
$\{ \gamma_m(\epsilon) \}_{1 \le m \le p}$ denote the corresponding rays in the dual fan to $P(\epsilon)$. 
Consider elements $\{ u_{m,k} \in \gamma_m(\epsilon_k) : 1 \le m \le p, k \ge 0 \}$ such that $||u_{m,k} || = 1$. 
Then, after possibly replacing $\{ \epsilon_k \}_{k \in \Z_{\ge 0}}$ with a subsequence, we may assume that
$\lim_{k \to \infty} u_{m,k} = u_m$ exists for $1 \le m \le p$. 
Then $u_m \in \limsup_k \gamma_m(\epsilon_k)  \subset \gamma_m$ and $||u_m || = 1$.
Given an element $u \in \tau_J$, there exists $a_m \in \R^n_{\ge 0}$ such that $u = \sum_m a_m u_m$. Then 
$\lim_{k \to \infty} \sum_m a_m u_{m,k} = u \in
\limsup_k \tau_{J(\epsilon_k)}$,
as desired. 
\end{proof}

With these lemmas in hand, we now return to our problem. 
Fix a restricted, 
weakly $\alpha$-compatible pair $(\mathcal{Z},\mathcal{F})$, where 
$\mathcal{Z} = (Z_s)_{s \in \mathcal{S}}$ and $\mathcal{F} = (F_s)_{s \in \mathcal{S}}$.
We may apply Lemma~\ref{l:combconstant} to both $P(\epsilon) = \Newt(f)_{\mathcal{Z}(\epsilon)}$ and $P(\epsilon) = Q_{\mathcal{Z}(\epsilon)}$. 
Hence, by Lemma~\ref{l:deformZ}, we may replace 
$\mathcal{Z}$ by $\mathcal{Z}(\epsilon)$
so that $\{ Q_{\mathcal{Z}(\epsilon)}\}_\epsilon$ and $\{ \Newt(f)_{\mathcal{Z}(\epsilon)}\}_\epsilon$  are locally combinatorially constant. 
Consider a 
nonempty face $J$ of $Q_{\mathcal{Z}}$, dual to a cone $\tau_J$ in $\Sigma_{\mathcal{Z}}$. 
For $\epsilon$ sufficiently small, let
$J(\epsilon)$ denote the corresponding nonempty face  of 
$Q_{\mathcal{Z}(\epsilon)}$, dual to the cone 
$\tau_{J(\epsilon)}$ of $\Sigma_{\mathcal{Z(\epsilon})}$. 
Similarly, given a cone $C$ in $\Sigma \cap \Sigma_{\mathcal{Z}}$, we  let $C(\epsilon)$ denote the corresponding cone in $\Sigma \cap \Sigma_{\mathcal{Z(\epsilon})}$. 
If $C$ is dual to a 
face of $\Newt(f)_{\mathcal{Z}}$ of the form $K' + J'$ for some $K' \in \Gamma$ and some nonempty face $J' \subset Q_{\mathcal{Z}(\epsilon)}$, 
then $C(\epsilon) = \sigma_{K'} \cap \tau_{J'(\epsilon)}$ is 
dual to the face $K' + J'(\epsilon)$ in $\Newt(f)_{\mathcal{Z}(\epsilon)}$.

\begin{remark}\label{r:independence}
Consider a nonempty face $J$ of $Q_{\mathcal{Z}}$ such that $\sigma_M \cap \tau_J \neq \{ 0 \}$.
Since $\{ Q_{\mathcal{Z}(\epsilon)} \}_{\epsilon}$ is locally combinatorially constant,
for any $s \in \mathcal{S}$, $Z_s \in J$ if and only if 
$Z_s(\epsilon_s) \in J(\epsilon)$.
In particular, if 
$\sigma_M \cap \tau_{J(\epsilon)} \neq \{ 0 \}$,
then $F_J  = \max \{ F_s : Z_s \in J \} = \max \{ F_s : Z_s(\epsilon_s) \in J(\epsilon) \} = F_{J(\epsilon)}$.
\end{remark}

The lemma below says that a pair $(C,J)$ not being $\alpha$-critical is an open condition. 

\begin{lemma}\label{l:closed}
Let $J$ be a nonempty face of $Q_{\mathcal{Z}}$, 
and consider a cone $C \in \Sigma \cap \Sigma_{\mathcal{Z}}$.
Suppose there exists a 
sequence $\{ \epsilon_m  \}_{m \in \Z_{\ge 0}}$  
such that $\epsilon_m  \in \Q^{\mathcal{S}}$, $\lim_{m \to \infty} \epsilon_m = 0$ and
$(C(\epsilon_m),J(\epsilon_m))$ is $\alpha$-critical for all $m$. Then  
$(C,J)$ is  $\alpha$-critical.
\end{lemma}
\begin{proof}
By Lemma~\ref{l:Cinvariant} and since $\{ \Newt(f)_{\mathcal{Z}(\epsilon)}\}_\epsilon$ is locally combinatorially constant, $C(\epsilon_m) \subset \tau_{J(\epsilon_m)}$ is dual to a face of the form $K + J'(\epsilon_m)$ of $\Newt(f)_{\mathcal{Z}(\epsilon_m)}$, where
$K \in \Gamma$ and
$J'$ is a face  of $Q_{\mathcal{Z}}$ such that $J(\epsilon_m) \subset J'(\epsilon_m)$, or, equivalently, 
$J \subset J'$. Then $C$ is dual to $K + J'$. 
In particular, $C \subset \tau_{J'} \subset \tau_J$. 
By hypothesis, 
$C(\epsilon_m) \cap \sigma_M \neq \{ 0 \}$. 
It follows from Bolzano--Weierstrass and  Lemma~\ref{l:convergence} that 
$C \cap \sigma_M \neq \{ 0 \}$. Then $\sigma_M \cap \tau_J \neq \{ 0 \}$, and, by Remark~\ref{r:independence}, $F_J  = F_{J(\epsilon_m)}$ for all $m$. The condition $C(\epsilon_m) \subset \sigma_{A_{F_J}}$ implies that $\sigma_{K} \subset 
\sigma_{A_{F_J}}$, and hence $C \subset \sigma_{A_{F_J}}$. 
By Lemma~\ref{l:convergence} and since $\mathrm{H}_{\alpha A_{F_J} + \One}$ is closed,  $C = \limsup_m C(\epsilon_m)  \subset \mathrm{H}_{\alpha A_{F_J} + \One}$,  and hence 
$C$ is critical with respect to $(\alpha,A_{F_J})$. 
We conclude that $(C,J)$ is  $\alpha$-critical.
\end{proof}

We say that $\epsilon$ 
can be chosen to be arbitrarily
small 
if for any $\delta > 0$, there exists a choice of $\epsilon$ such that 
$|\epsilon_{s}| < \delta$  for all 
$s \in \mathcal{S}$.

\begin{lemma}\label{l:deformnotcritical}
Suppose that $(C,J)$ is  $\alpha$-critical and  $C \not \subset \sigma_M$. 
Then there exists an arbitrarily small choice of $\epsilon \in \Q^{\mathcal{S}} $ such that 
$(C(\epsilon),J(\epsilon))$ is not 
$\alpha$-critical.
\end{lemma}
\begin{proof}
By Lemma~\ref{l:Cinvariant},  $C$ is dual to a face  $K + J'$ of $\Newt(f)_{\mathcal{Z}}$,
where
$K \in \Gamma$ and $J'$ is a face of $Q_{\mathcal{Z}}$ 
such that $K  
\subset F_{J'}$ and  $J \subset J'$. 
Then 
$C(\epsilon) \subset \tau_{J(\epsilon)}$ is dual to the face $K + J'(\epsilon)$ of $\Newt(f)_{\mathcal{Z}(\epsilon)}$. Note that $C \subset \sigma_{A_{F_J}}$ implies that $\sigma_K \subset \sigma_{A_{F_J}}$, and hence $A_{F_J} \in K$. Since $K,M$ are faces of $\Gamma$,
$K \cap M$ is a (possibly empty) face of $M$, and $C_{K \cap M} = C_{K} \cap C_{M}$.
Let $B_K = \gens(C_{K \cap M}) \cup \mathcal{A}_M$. 

Assume that $\One \in \Span(B_{K})$. 
We can write
$
\One = \sum_{V \in B_{K}} \lambda_V V, 
$
for some  coefficients $\lambda_V$, 
with $\lambda_V = 1$ for all $V \in \mathcal{A}_M$.
Applying $\psi_{M}$ to both sides gives 
$
-\alpha = \sum_{V \in B_{K} \smallsetminus \Unb(C_{K \cap M})} \lambda_V.
$
Hence we may equivalently write 
$$
\alpha A_{F_J} + \One = \sum_{V \in B_{K} \smallsetminus \Unb(C_{K \cap M})} \lambda_V (V - A_{F_J}) + \sum_{V \in  \Unb(C_{K \cap M})} \lambda_V V.
$$
Consider $u \in C^\circ \subset \sigma_{K}^\circ  \cap \mathrm{H}_{\alpha A_{F_J} + \One}$. 
Consider $V \in \gens(C_K)$. If $V \in \Unb(C_{K})$, then $u \in \sigma_{K}^\circ$ implies that $\langle u , V \rangle = 0$. 
Since $A_{F_J} \in K$, if $V \in \ver(K)$, then 
$u \in \sigma_{K}^\circ$ implies that
$\langle u , V - A_{F_J} \rangle = 0$. 
We compute:
\begin{align*}
0 = \langle u, \alpha A_{F_J} + \One \rangle &= 
\sum_{V \in B_{K} \smallsetminus \Unb(C_{K \cap M})} \lambda_V \langle u, V - A_{F_J} \rangle + \sum_{V \in  \Unb(C_{K \cap M})} \lambda_V \langle u, V \rangle 
\\
&= \sum_{V \in B_{K} \smallsetminus C_K} \langle u, V - A_{F_J} \rangle.
\end{align*}
Since each term in the right-hand sum is positive, we deduce that the sum must be empty. It follows that $B_{K} \subset 
C_{K \cap M}$ and 
$\One \in 
\Span(C_{K \cap M})$.
Since $M$ is minimal in $\Contrib(\alpha)$, we deduce that 
$K \cap M = M$. 
Then $C \subset \sigma_{K} \subset \sigma_M$, a contradiction.

We conclude that $\One \notin \Span(B_{K})$. 
Since $A_{F_J} \in B_{K}$,  it follows that $\alpha A_{F_J} + \One \notin \Span(B_{K})$. 
Since 
$\alpha A_{F_J} + \One \in \Span(M) \cap \Q^n$ and 
$\Newt(f)$ is $\alpha$-simplicial, it follows that there exists $u' \in \Q^n$ such that
\begin{enumerate}
\item\label{i:uWWW1} $\langle u' , \alpha A_{F_J} + \One \rangle = 1$,
\item\label{i:uWWW2} $\langle u' , V \rangle = 0$ for all elements $V \in B_{K}$, and
\item\label{i:uWWW3} $\langle u' , V \rangle = 0$ for all $V \in \gens(C_{F_{J'}} \smallsetminus C_M)$.
\end{enumerate}  
Consider an element  $u \in C^\circ \cap \Q^n = \sigma_{K}^\circ \cap (\tau_{J'})^\circ \cap \Q^n  \subset \mathrm{H}_{\alpha A_{F_J} + \One}$, and 
let $\hat{u}(\lambda) = u + \lambda u'$ for some choice of $\lambda \neq 0 \in \R$.
Then property \eqref{i:uWWW1} implies that 
$\langle \hat{u}(\lambda), \alpha A_{F_J} + \One \rangle = \lambda \neq 0$, and hence 
$\hat{u}(\lambda) \notin \mathrm{H}_{\alpha A_{F_J} + \One}$. 
Properties \eqref{i:uWWW2} and \eqref{i:uWWW3} imply that  for any $V$ in $C_K$,
$
\langle \hat{u}(\lambda) , V \rangle = \langle u , V \rangle + \lambda \langle u' , V \rangle = \langle u , V \rangle. 
$
It follows that $\hat{u}(\lambda) \in \sigma_{K}^\circ$ provided $|\lambda|$ is sufficiently small.

Recall that $\tau_{J'} = \cap_{Z_s \in J'} \tau_{Z_s}$.
Consider $s \in \mathcal{S}$ such that $Z_{s} \in J'$.
We claim that for a generic choice of $\lambda \in \Q$, 
we may choose $\epsilon_{s} \in \Q$  such that
$ \langle \hat{u}(\lambda), Z_{s}(\epsilon_{s}) \rangle = \langle u , Z_{s} \rangle$. Assume this claim holds.
It follows that with this choice of $\epsilon = \{ \epsilon_s \}_{s \in \mathcal{S}}$, 
$\hat{u}(\lambda) \in (\tau_{J'}(\epsilon))^\circ$ provided $|\lambda|$ is chosen sufficiently small. 
Then $\hat{u}(\lambda) \in \sigma_{K}^\circ \cap (\tau_{J'}(\epsilon))^\circ = C(\epsilon)^\circ$ and $\hat{u}(\lambda) \notin \mathrm{H}_{\alpha A_{F_J} + \One}$.
Either $\sigma_M \cap \tau_{J(\epsilon)} = \{ 0 \}$, 
or $\sigma_M \cap \tau_{J(\epsilon)} \neq \{ 0 \}$ and, by Remark~\ref{r:independence}, 
$F_{J} = F_{J(\epsilon)}$ and $C(\epsilon) \not \subset \mathrm{H}_{\alpha A_{F_{J(\epsilon)}} + \One}$. 
In either case, $(C(\epsilon),J(\epsilon))$ is not $\alpha$-critical. 

It remains to verify the claim.  We compute: 
\begin{align*}
\langle \hat{u}(\lambda), Z_{s}(\epsilon_{s}) \rangle &=  \langle \hat{u}(\lambda), Z_{s} + \epsilon_{s} V_M \rangle \\
&= \langle u , Z_{s} \rangle + \lambda \langle u' , Z_{s} \rangle + \epsilon_{s}(\langle u, V_M \rangle + \lambda \langle u', V_M \rangle). 
\end{align*}
Suppose that  $\langle u', V_M \rangle \neq 0$. Then for $\lambda \neq -\frac{\langle u, V_M \rangle}{\langle u', V_M \rangle}$, 
we may set 
$
\epsilon_{s} = -\frac{\lambda \langle u' , Z_{s} \rangle}{\langle u, V_M \rangle + \lambda \langle u', V_M \rangle}, 
$
and the above calculation shows that $ \langle \hat{u}(\lambda), Z_{s}(\epsilon_{s}) \rangle = \langle u , Z_{s} \rangle$. 
If $\langle u', V_M \rangle = 0$, then we may set $
\epsilon_{s} = 0$.
Since $(\mathcal{Z},\mathcal{F})$  is restricted, 
$Z_s \in \Span(\{ V_M \} \cup \gens(C_{F_s} \smallsetminus C_M) \cup \mathcal{A}_M)$ for every $s \in \mathcal{S}$.
Since $Z_s \in J'$,  Definition~\ref{d:max} implies that $F_s \subset F_{J'}$. 
Then properties 
\eqref{i:uWWW2} and \eqref{i:uWWW3} imply that $\langle u', Z_{s} \rangle = 0$, and the above calculation shows that $ \langle \hat{u}(\lambda), Z_{s}(\epsilon_{s}) \rangle = \langle u , Z_{s} \rangle$.
\end{proof}

\begin{lemma}\label{l:weakimpliesstrong}
Suppose there exists a restricted, weakly 
$\alpha$-compatible pair $(\mathcal{Z},\mathcal{F})$.
Then there exists an
$\alpha$-compatible pair $(\mathcal{Z},\mathcal{F})$.
\end{lemma}
\begin{proof}
Consider the restricted, weakly 
$\alpha$-compatible pair $(\mathcal{Z},\mathcal{F})$ above. 
Suppose
$(\mathcal{Z},\mathcal{F})$ is not $\alpha$-compatible. 
That is, suppose there exists a pair  $(C,J)$ that is $\alpha$-critical and  $C \not \subset \sigma_M$.
Then Lemma~\ref{l:closed} and Lemma~\ref{l:deformnotcritical} imply that we can deform 
$(\mathcal{Z},\mathcal{F})$ and strictly increase the number of pairs 
$(C,J)$ that do not have $\alpha$-critical intersection. 
Since there are finitely many such pairs, by repeating this procedure we obtain an  $\alpha$-compatible pair.
\end{proof}

\subsection{Existence of restricted, weakly $\alpha$-compatible sets}\label{ss:exist2}

In this section, we use the existence of a locally unique labeling 
to explicitly construct a restricted, weakly 
$\alpha$-compatible pair. Recall that $\gens(C_F) = \ver(F) \cup \Unb(C_{F})$ is the set of distinguished vertices on the rays of $C_F$. Recall that because $\Newt(f)$ is $\alpha$-simplicial, 
there is a bijection between
$\{ K \in \Gamma : M \subset K \subset F \}$ and 
subsets of $\gens(C_F \smallsetminus C_M) = \gens(C_F) \smallsetminus \gens(C_M)$.

Consider an element $F \in \Contrib(\alpha)_M$. 
Given an element $V$ in 
$\gens(C_F)$, let $\zeta(V) \in F \subset \Gamma$ be defined by 
\[
\zeta(V) := \begin{cases}
V &\textrm{ if } V \in \ver(F), \\
V + V_M &\textrm{ if } V \in \Unb(C_{F}).
\end{cases}
\]

\begin{lemma}\label{l:zetabasic}
Suppose that $F,F' \in \Contrib(\alpha)_M$ and $V \in \gens(C_F)$. Then $\zeta(V) \in F'$ if and only if 
$V \in \gens(C_{F'})$. 
\end{lemma}
\begin{proof}
First, suppose that $V \in \ver(F)$. Then $\zeta(V) = V \in F'$ if and only if $V \in \ver(F')$. 
Second, suppose that $V \in \Unb(C_{F})$. Consider $u \in \sigma_{F'}^\circ$. 
Then $\langle u , V \rangle = \langle u , \zeta(V) - V_M \rangle$, so $\zeta(V) \in F'$  if and only if
$\langle u , \zeta(V) - V_M \rangle = 0$, if and only if 
$\langle u , V \rangle = 0$, if and only if 
$V \in \Unb(C_{F'})$. The result follows.
\end{proof}

Let $\mathcal{S}$ be the set of saturated chains 
of faces  in $\Gamma$ starting at $M$, i.e., a chain of faces starting at $M$ where the dimension increases by one at each step. Let $s = F_\bullet$ be an element of $\mathcal{S}$. Let 
$\ell_s$
denote the length of $F_\bullet$, i.e., the number of elements in $F_\bullet$ minus one. 
We
let $F_{\bullet,i}$ denote the $i$th element of $F_\bullet$ for $0 \le i \le \ell_s$. For example, $F_{\bullet,0} = M$. 
We write $F \in F_\bullet$ if $F = F_{\bullet,i}$ for some $0 \le i \le \ell_s$.

Define  $V_{s,0} = V_M$. 
Since $F_\bullet$ is saturated and  $\Newt(f)$ is $\alpha$-simplicial, we may define 
$V_{s,i}$ to be
the unique 
element of $\gens(C_{F_{\bullet,i}} \smallsetminus C_{F_{\bullet,i-1}})$ for $1 \le i \le \ell_s$.

\begin{definition}\label{d:explicitconstruct}
Let $\mathcal{S}$ be the set of saturated chains 
of faces  in $\Gamma$ starting at $M$. 
We define a pair $(\mathcal{Z},\mathcal{F})$, where 
$\mathcal{Z} = (Z_s)_{s \in \mathcal{S}}$ and $\mathcal{F} = (F_s)_{s \in \mathcal{S}}$ are collections of elements of $\Q^n$ and $\Contrib(\alpha)_M$ respectively, as follows:
for any element $s = F_\bullet$  of $\mathcal{S}$, let
\[
Z_s := \sum_{i = 0}^{\ell_s} b_{i,\ell_s} \zeta(V_{s,i}),
\]
where $\{ b_{i,j} \}_{0 \le i,j \le r}$, $r = n - 1 - \dim M$, and
\[
b_{i,j} = b_{i,j}(\mu) =
\begin{cases}
2^{-i} - 2i \mu,& \text{if } i = j \\   
2^{-(i+1)}  - (i + j)\mu,& \text{if } i < j \\
0,              & \text{otherwise}
\end{cases}
\]
for some $\mu \in \Q$ such that $0 < \mu \ll 1$.
Let  $F_s := F_{\bullet,\ell_s}$ be the maximal element of $F_\bullet$.
\end{definition}

For example, $b_{0,0} = 1$ and if $s = F_\bullet$, where 
$F_\bullet = \{ M \}$, then $\ell_s = 0$, $Z_s = V_M$, and $F_s = M$. 
Note that  we abuse notation above by not indicating  the dependence of $(\mathcal{Z},\mathcal{F})$ on the choice of $\mu$. Below we fix a value of $\mu$ sufficiently small. 
Our goal is to show that we can construct 
a restricted, weakly 
$\alpha$-compatible pair from $(\mathcal{Z},\mathcal{F})$. Recall that we have fixed a locally unique labeling $(A_F, e_{F}^*)$ of $\Contrib(\alpha)_M$. 

\begin{definition}\label{d:orthogoperator}
Let $s = F_\bullet \in \mathcal{S}$. 
Let $\mathcal{A}_s = 
\{ A_F : F \in F_\bullet \}$. Given an element $A$ in $\mathcal{A}_s$, we define a base direction $e^*_{s,A}$ as follows: if $A = A_F$ for some $F \in F_\bullet$, then 
$e^*_{s,A} := e_F^*$. We  define a linear function 
$\Phi_{s} \colon \R^n \to \R^n$ by 
\[
\Phi_{s}(X) = X - \sum_{A \in \mathcal{A}_s} \langle e^*_{s,A} , X \rangle (A - V_M).
\]
\end{definition}

The fact that $e^*_{s,A}$ is well-defined in Definition~\ref{d:orthogoperator}
is an immediate consequence of the  locally unique labeling condition.  Explicitly, if $A = A_{F} = A_{F'}$ for some  $F,F' \in F_\bullet$, then either $F \subset F'$ or $F' \subset F$, and $(*)$
implies that $e_F^* = e^*_{F'}$.

\begin{remark}\label{r:phiFbullet}
For $s \in \mathcal{S}$, $u \in \sigma_M$ and $X \in \R^n$, 
$\langle u , \Phi_{s}(X) \rangle = \langle u , X \rangle$.
\end{remark}

\begin{lemma}\label{l:orthogformula}
Let $s = F_\bullet \in \mathcal{S}$.
Then $\langle e_F^* , \Phi_{s}(X) \rangle = 0$ for any $F \in F_\bullet$ and any $X \in \R^n$. If $F_s \subset K$ and $\langle e_K^*, X \rangle = 0$ for some $X \in \R^n$, then $\langle e_K^*, \Phi_{s}(X) \rangle = 0$.
\end{lemma}
\begin{proof}
Recall from \eqref{e:VM} that
$\langle e_F^*, V_M \rangle = 0$ for all $F \supset M$.
Suppose that $F \in F_\bullet$. 
Then $e^*_{s,A_F} = e_F^*$, and we compute
$\langle e_F^* , \Phi_{s}(X) \rangle = 
\langle e_F^* , X - \langle e^*_{s,A_F} , X \rangle (A_F - V_M) \rangle = 0.$
Suppose that $F_s \subset K$ and $\langle e_K^*, X \rangle = 0$. If 
$\langle e_K^*, A \rangle = 0$ for  all $A \in \mathcal{A}_s$, then 
$\langle e_K^*, \Phi_s(X) \rangle = \langle e_K^*, X \rangle = 0$.
Suppose that 
$\langle e_K^*, A_F \rangle \neq 0$ for some $F \in F_\bullet$. Then $A_K = A_F$. Since $F \subset K$,  the locally unique labeling condition $(*)$ implies that $e^*_{s,A_F} = e_F^* = e_K^*$. As above, 
$\langle e_K^*, \Phi_{s}(X) \rangle = 
\langle e_K^*, X - \langle e^*_{s,A_F} , X \rangle (A_F - V_M) \rangle = 0.$
\end{proof}

\begin{lemma}\label{l:reducetooriginal}
With the notation of Definition~\ref{d:explicitconstruct}, suppose that  $(\mathcal{Z},\mathcal{F})$ satisfies the following property:

Suppose that 
$\sigma_K^\circ \cap \tau_{Z_{s}} \cap \tau_{Z_{s'}}  \ne \emptyset$, for some  $K \in \Contrib(\alpha)_M$  and $s  = F_\bullet ,s'  = F_\bullet' \in \mathcal{S}$. 
Then 
\begin{enumerate}
\item\label{i:pweak1} $K \subset F_s$, and
\item\label{i:pweak2}
either $F_s \in F_\bullet'$ or $F_{s'} \in F_\bullet$. 
\end{enumerate} 
Let $\Phi(\mathcal{Z}) := ( \Phi_s(Z_s) )_{s \in \mathcal{S}}$. Then $(\Phi(\mathcal{Z}),\mathcal{F})$ is restricted and weakly $\alpha$-compatible. 
\end{lemma}
\begin{proof}
It follows from Definition~\ref{d:explicitconstruct} and Definition~\ref{d:orthogoperator} that $(\Phi(\mathcal{Z}),\mathcal{F})$ is restricted. Suppose that 
$\sigma_K^\circ \cap \tau_{\Phi_s(Z_{s})} \cap \tau_{\Phi_{s'}(Z_{s'})}  \ne \emptyset$, for some  $K \in \Contrib(\alpha)_M$  and $s  = F_\bullet ,s'  = F_\bullet' \in \mathcal{S}$.
By Remark~\ref{r:phiFbullet}, the restriction of $\Sigma_{\mathcal{Z}}$ to $\sigma_M$ equals the 
restriction of $\Sigma_{\Phi(\mathcal{Z})}$ to $\sigma_M$. Hence $\sigma_K^\circ \cap \tau_{Z_{s}} \cap \tau_{Z_{s'}}  \ne \emptyset$. We deduce that 
$K \subset F_s$, and,
either $F_s \in F_\bullet'$ or $F_{s'} \in F_\bullet$.
The latter condition implies that either $F_s \subset F_{s'}$ or $F_{s'} \subset F_s$. 

Applying Lemma~\ref{l:orthogformula} with $s = F_\bullet$ and $F = F_s$, gives $
\langle e^*_{F_s}, \Phi_s(Z_s) \rangle = 0$. 
It remains to show that $
\langle e^*_{F_s}, \Phi_{s'}(Z_{s'}) \rangle = 0$. 
Suppose that $F_s \in F_\bullet'$. Applying
Lemma~\ref{l:orthogformula} with $s' = F_\bullet'$ and $F = F_s$, gives
$
\langle e^*_{F_s}, \Phi_{s'}(Z_{s'}) \rangle = 0$, as desired. 
Suppose that $F_{s'} \in F_\bullet$. 
Then $F_{s'} \subset F_s$. 
Applying
Lemma~\ref{l:orthogformula} with $s' = F_\bullet'$  and $K = F_s$, gives 
$
\langle e^*_{F_s}, \Phi_{s'}(Z_{s'}) \rangle = 0$, provided 
$\langle e^*_{F_s}, Z_{s'} \rangle = 0$. 
By Definition~\ref{d:explicitconstruct}, 
$Z_{s'} \in \Span(\{ V_M \} \cup \gens(C_{F_{s'}} \smallsetminus C_M))$. 
Since $F_{s'} \subset F_s$, $F_s$ is a $B_1$-face with base direction $e^*_{F_s}$, and 
$\langle e^*_{F_s}, V_M \rangle = 0$ by \eqref{e:VM}, it follows from Remark~\ref{r:apexisbounded} that $\langle e^*_{F_s}, Z_{s'} \rangle = 0$, as desired.
\end{proof}

It remains to show that  $(\mathcal{Z},\mathcal{F})$ satisfies conditions \eqref{i:pweak1} and \eqref{i:pweak2} in Lemma~\ref{l:reducetooriginal}.
We will prove this through a series of lemmas.

\begin{lemma}\label{l:lowerbound2}
There exists a constant $\lambda_M > 0$ such that for 
any $K,F,F' \in \Contrib(\alpha)_M$ that are not subfaces of  a common face in  $\Contrib(\alpha)_M$, and for any nonzero 
$u \in \sigma_K$, there exists an element 
$V \in \gens(C_F \smallsetminus C_M) \cup \gens(C_{F'} \smallsetminus C_M)$ such that 
$
\langle u , \zeta(V) \rangle / N(u) \ge 1 + \lambda_M.
$
\end{lemma}
\begin{proof}
Fix $K,F,F' \in \Contrib(\alpha)_M$ that are not subfaces of  a common face in  $\Contrib(\alpha)_M$. 
Let $\mathcal{V} = \gens(C_F \smallsetminus C_M) \cup \gens(C_{F'} \smallsetminus C_M)$, and 
consider the continuous
function 
$\phi \colon \sigma_K \smallsetminus \{ 0 \} \to \R$  
defined by 
\[
\phi(u) =  \Big( \max_{V \in \mathcal{V}} \langle u , \zeta(V) \rangle/N(u)\Big) - 1. 
\]
We claim that the image satisfies $\Ima(\phi) \subset \R_{> 0}$. Indeed, suppose $\phi(u) = 0$. 
Let $F_u$ be the face of $\partial \Newt(f)$ minimized by $u$. Then $F_u$ contains $K$ and $\{ \zeta(V) : V  \in \mathcal{V} \}$.
By Lemma~\ref{l:zetabasic}, $C_{F_u}$ contains $C_K$, 
$C_{F}$ and $C_{F'}$. Then $K,F,F'$ are common subfaces of $F_u$, a contradiction.

Note that $\phi(\eta u) = \phi(u)$ for all $\eta \in \R_{> 0}$ and 
$u \in   \sigma_K \smallsetminus \{ 0 \}$. 
Since $\sigma_K \cap S$ is compact, there exists $\lambda = \lambda(K,F,F') > 0$ such that $\phi(\sigma_K \cap S) \subset [\lambda,\infty)$.
We let $\lambda_M$ be the minimum value of $\lambda(K,F,F')$ over the finitely many choices of 
$K,F,F'$.
\end{proof}

Below we fix $\lambda_M > 0$ satisfying Lemma~\ref{l:lowerbound2}.

\begin{lemma}\label{l:bij}
Let $r = n - 1 - \dim M$. 
Assume that $\mu > 0$ is chosen sufficiently small. 
Then the coefficients 
$\{ b_{i,j} = b_{i,j}(\mu) \}_{0 \le i,j \le r}$ satisfy the following properties:
\begin{enumerate}
\item \label{i:B1} 
$b_{i,j} > 0$ for $i \le j$, 
\item\label{i:B2}  
$b_{i,j} > b_{i,j + 1}$ for $i \le j < r$, 
\item\label{i:B3}  
$b_{i,j} > b_{i + 1,j}$ for $i < j$,
\item\label{i:B4}  
$\sum_{i \ge k} b_{i,j} > \sum_{i \ge k} b_{i,j + 1}$ for any $1 \le k \le j < r$, and
\item\label{i:B5}   
$\sum_{i \ge 0} b_{i,r} +  b_{r,r}\lambda_M > 1$.
\end{enumerate}
\end{lemma}
\begin{proof}
We check the conditions hold by direct computation, for $\mu$ sufficiently small. 
Condition \eqref{i:B1} is clear.
For condition \eqref{i:B2}, we compute, for $i = j < r$,
\[ b_{i,i} = 2^{-i} - 2i\mu > 
b_{i,i+1} = 2^{-(i + 1)} - (2i + 1)\mu, \]
and, for $i < j < r$,
\[
b_{i,j} = 2^{-(i + 1)}  - (i + j)\mu >
b_{i,j + 1} = 2^{-(i + 1)}  - (i + j + 1)\mu.
\]
For condition \eqref{i:B3}, we compute, for $i + 1 < j$,
\[
b_{i,j} = 2^{-(i + 1)}  - (i + j)\mu >
b_{i+1,j} = 2^{-(i+2)}  - (i + j + 1)\mu,
\]
and, for $i + 1 = j$,
\[
b_{i,i+1} = 2^{-(i + 1)}  - (2i + 1)\mu >
b_{i+1,i+1} = 2^{-(i + 1)} - (2i + 2)\mu.
\]
For condition \eqref{i:B4}, define  $c_{k,j} = \sum_{i \ge k} b_{i,j}$ for $k \le j$. Then 
\begin{align*}
c_{k,j} &= 2^{-k} - \sum_{i = k}^j (i + j)\mu \\
&= 2^{-k} - (j - k +1)(3j + k)\mu/2.
\end{align*}
For $j < r$, $c_{k,j} > c_{k,j+1}$, as desired.
For condition \eqref{i:B5}, we compute
\begin{align*}
\sum_{i \ge 0} b_{i,r} + b_{r,r} \lambda_M  - 1 
&= c_{0,r} + b_{r,r} \lambda_M  - 1
\\
&= - 3r(r + 1)\mu/2 + b_{r,r} \lambda_M  
\\
&= - 3r(r + 1)\mu/2 +  (2^{-r} - 2r\mu)\lambda_M
\\
&=  2^{-r}\lambda_M - \mu r(3(r + 1)/2 + 2\lambda_M).
\end{align*}
The latter expression is positive for $\mu$ sufficiently small.
\end{proof}

\begin{lemma}\label{l:Vstructure2}
Let $s = F_\bullet \in \mathcal{S}$ and suppose
$u \in \sigma_M \cap \tau_{Z_s}$ is nonzero.
Then
$\langle u , \zeta(V_{s,i}) \rangle \le \langle u , \zeta(V_{s,j}) \rangle$ for any $0 \le i \le j \le \ell_s$. 
Moreover, if $F_s \subset F \in \Gamma$, then there 
exists a constant $0 \le m < \lambda_M$ 
such that $\gens(C_{F_s} \smallsetminus C_M) = \{ V \in \gens(C_F \smallsetminus C_M) : \langle u , \zeta(V) \rangle/N(u) \le 1 + m \}$. 

\end{lemma}
\begin{proof}
Since $u \in \sigma_M$, $\zeta(V_{s,0}) = V_M \in M$, and $\zeta(V_{s,j}) \in \Newt(f)$, it follows that 
$\langle u , \zeta(V_{s,0}) \rangle \le \langle u , \zeta(V_{s,j}) \rangle$ for any $0 \le j \le \ell_s$. Suppose that 
$\langle u , \zeta(V_{s,i}) \rangle > \langle u , \zeta(V_{s,j}) \rangle$ for some  $0 < i < j \le \ell_s$.
Let $\pi = (i,j) \in \Sym_{\ell_s}$ be the permutation of $[\ell_s]$ switching $i$ and $j$. 
Let
$\pi(s)$ be the unique element in $\mathcal{S}$ such that $\ell_{\pi(s)} = \ell_s$ and $V_{\pi(s),i} = V_{s,\pi(i)}$ for $1 \le i \le \ell_s$. 
Using \eqref{i:B3} in Lemma~\ref{l:bij}, we compute:
\begin{align*}
\langle u , Z_s - Z_{\pi(s)} \rangle 
&= b_{i,\ell_s} \langle u ,  \zeta(V_{s,i}) - \zeta(V_{\pi(s),i}) \rangle + 
b_{j,\ell_s} \langle u ,  \zeta(V_{s,j}) - \zeta(V_{\pi(s),j}) \rangle 
\\
&= 
(b_{i,\ell_s} - b_{j,\ell_s})\langle u , \zeta(V_{s,i}) - \zeta(V_{s,j}) \rangle > 0.
\end{align*}
The latter contradicts the assumption that $u \in \tau_{Z_s}$. This completes the proof of the first statement.

Since $M$ is interior and $u \in \sigma_M$, $N(u) > 0$. 
Let $m 
= (\langle u , \zeta(V_{s,\ell_s}) \rangle/N(u)) - 1 \ge 0$. 
Assume that $m \ge \lambda_M$. 
Using all the 
statements 
of Lemma~\ref{l:bij},
we compute
\begin{align*}
\langle u , Z_{s} \rangle / N(u) &=  \sum_{i = 0}^{\ell_s} b_{i,\ell_s} \langle u , \zeta(V_{s,i}) \rangle / N(u)
\\
&\ge  \sum_{i = 0}^{\ell_s} b_{i,\ell_s} + b_{\ell_s,\ell_s} \lambda_M
\\
&\ge  \sum_{i = 0}^{r} b_{i,r} + b_{r,r} \lambda_M > 1.
\end{align*}
On the other hand, if $\tilde{s}$ is the unique element 
in $\mathcal{S}$ with $\ell_{\tilde{s}} = 0$, then, since $u \in \sigma_M$, 
\[
\langle u , Z_{\tilde{s}} \rangle/N(u) = \langle u , V_M \rangle / N(u) = 1 < \langle u , Z_{s} \rangle / N(u).
\]
This contradicts the assumption that $u \in
\tau_{Z_s}$. We conclude that $m < \lambda_M$.

Since $\langle u , \zeta(V_{s,\ell_s}) \rangle  = \max_{0 \le i \le \ell_s} \langle u, \zeta(V_{s,i}) \rangle$, 
we have 
\[
\gens(C_{F_s} \smallsetminus C_M) \subset \{ V \in \gens(C_F \smallsetminus C_M) : \langle u , \zeta(V) \rangle/N(u) \le 1 + m \}.
\]
It remains to prove the reverse inclusion. 
Suppose that 
$V \in \gens(C_F \smallsetminus C_M)$ and $\langle u , \zeta(V) \rangle/N(u) \le 1 + m$. Equivalently, we assume that 
$\langle u , \zeta(V_{s,\ell_s}) - \zeta(V) \rangle \ge 0$.
We argue by contradiction. Assume that
$V \notin C_{F_s}$. Let $s'$ be the unique element in $\mathcal{S}$ such that $\ell_{s'} = \ell_s + 1$, $V_{s',i} = V_{s,i}$ for $0 \le i \le \ell_s$, and $V_{s',\ell_{s'}} = V$. If we let $V_{s,-1} = V_{s',-1} = 0$, then we can write\[
\langle u , Z_s \rangle = \sum_{k = 0}^{\ell_s} b_{k,\ell_s} \langle u ,  \zeta(V_{s,k}) \rangle = \sum_{k = 0}^{\ell_s} (\sum_{i = k}^{\ell_s} b_{i,\ell_s})\langle u , \zeta(V_{s,k}) - \zeta(V_{s,k-1}) \rangle, 
\]
\[
\langle u , Z_{s'} \rangle = \sum_{k = 0}^{\ell_s + 1} b_{k,\ell_s + 1} \langle u ,  \zeta(V_{s',k}) \rangle = \sum_{k = 0}^{\ell_s + 1} (\sum_{i = k}^{\ell_s + 1} b_{i,\ell_s + 1})\langle u , \zeta(V_{s',k}) - \zeta(V_{s',k-1}) \rangle, \text{ and} 
\]
\begin{align*}
\langle u , Z_s - Z_{s'} \rangle &= 
\sum_{k = 0}^{\ell_s} (\sum_{i = k}^{\ell_s} b_{i,\ell_s} - \sum_{i = k}^{\ell_s + 1} b_{i,\ell_s + 1})\langle u , \zeta(V_{s,k}) - \zeta(V_{s,k-1}) \rangle + b_{\ell_s + 1,\ell_s + 1}  \langle u, \zeta(V_{s,\ell}) - \zeta(V)  \rangle. 
\end{align*}
Since $u \in \sigma_M$ and $V_{s,0} = V_M \in M$, we have $\langle u , \zeta(V_{s,0}) \rangle = N(u) > 0$. 
By conditions \eqref{i:B1} and \eqref{i:B4} in Lemma~\ref{l:bij}, 
it follows that all terms above are nonnegative, and 
at least one term is positive.
This contradicts the assumption that $u \in
\tau_{Z_s}$.
\end{proof}

The following lemma completes our proof. 

\begin{lemma}\label{l:original}
With the notation of Definition~\ref{d:explicitconstruct}, $(\mathcal{Z},\mathcal{F})$ satisfies the 
the following property:

Suppose that 
$\sigma_K^\circ \cap \tau_{Z_{s}} \cap \tau_{Z_{s'}}  \ne \emptyset$, for some  $K \in \Contrib(\alpha)_M$  and $s  = F_\bullet ,s'  = F_\bullet' \in \mathcal{S}$. 
Then 
\begin{enumerate}
\item\label{i:ppweak1} $K \subset F_s$, and
\item\label{i:ppweak2}
either $F_s \in F_\bullet'$ or $F_{s'} \in F_\bullet$. 
\end{enumerate}  
\end{lemma}
\begin{proof}
Fix $u \in   \sigma_K^\circ \cap \tau_{Z_s} \cap \tau_{Z_{s'}}$. 
By Lemma~\ref{l:Vstructure2}, $\langle u , \zeta(V) \rangle/N(u) < 1 + \lambda_M$ for all $V \in \gens(C_{F_s} \smallsetminus C_M) \cup \gens(C_{F_{s'}} \smallsetminus C_M)$.
By Lemma~\ref{l:lowerbound2}, there exists a face $F \in \Contrib(\alpha)$ such that $K,F_s,F_{s'} \subset F$. 
By Lemma~\ref{l:Vstructure2}, there exists $m,m' \ge 0$ such that 
$$\gens(C_{F_s} \smallsetminus C_M) = \{ V \in \gens(C_F \smallsetminus C_M) : \langle u , \zeta(V) \rangle/N(u) \le 1 + m \}, \text{ and}$$ 
$$\gens(C_{F_{s'}} \smallsetminus C_M) = \{ V \in \gens(C_F \smallsetminus C_M) : \langle u , \zeta(V) \rangle/N(u) \le 1 + m' \}.$$

Since $u \in \sigma_K^\circ$, it follows from Lemma~\ref{l:zetabasic} that $\gens(C_K \smallsetminus C_M) = \{ V \in \gens(C_F \smallsetminus C_M) : \langle u , \zeta(V) \rangle / N(u) = 1 \}$, and hence 
$K \subset F_s$. 
It remains to establish 
\eqref{i:ppweak2}. 
Without loss of generality, we may assume  that $m \le m'$.  Then $\gens(C_{F_s} \smallsetminus C_M) = \{ V \in \gens(C_{F_{s'}} \smallsetminus C_M) : \langle u , \zeta(V) \rangle/N(u) \le 1 + m \}$.
By Lemma~\ref{l:Vstructure2}, $\langle u , \zeta(V_{s',i}) \rangle \le \langle u , 
\zeta(V_{s',j}) \rangle$ for $0 \le i \le j \le \ell_{s'}$. We deduce that $F_s \in F_{\bullet}'$. 
\end{proof}

A corollary of the proof of Lemma~\ref{l:original} is that $K \in F_\bullet$.

\begin{proof}[Proof of Theorem~\ref{thm:existence}]
Let $\alpha \not \in \Z$, and assume that all faces of $\Contrib(\alpha)$ are $\UB$ and $\Newt(f)$ is $\alpha$-simplicial. 
Let $M$ be a minimal face of $\Contrib(\alpha)$. Then Lemma~\ref{l:reducetooriginal} and Lemma~\ref{l:original} imply that there 
is a
restricted, weakly $\alpha$-compatible pair. Lemma~\ref{l:weakimpliesstrong} then 
implies that there is 
an $\alpha$-compatible pair. 
\end{proof}

\section{Beyond the simplicial case}\label{sec:beyondsimplicial}

Our techniques are capable of proving the local motivic monodromy conjecture for certain nondegenerate singularities whose Newton polyhedra are not simplicial.
In this section, we prove our strongest result on the local motivic monodromy conjecture, explain the remaining cases needed to prove the local motivic monodromy conjecture for nondegenerate singularities, and prove the local motivic monodromy conjecture for $3$-dimensional nondegenerate singularities.

\subsection{Local motivic monodromy conjecture}

We first state our strongest theorem on the local motivic monodromy conjecture for nondegenerate singularities. Recall that given $\beta \in \Q$,  $D(\beta) \in \Z_{> 0}$ is the denominator of $\beta$, 
written as a reduced fraction.

\begin{theorem}\label{thm:optimalclass}
Suppose $f$ is nondegenerate, and suppose that, for every $\alpha \in \mathbb{Q} \smallsetminus \Z$, either:
\begin{enumerate}
 \item\label{item:goodproj} $\Newt(f)$ has $D(\alpha)$-good projection and there is a face in $\Contrib(\alpha)$ that is not  pseudo-$\UB$,

 \item\label{item:nopole} $\Newt(f)$ is $\alpha$-simplicial and every face of $\Contrib(\alpha)$ is $\UB$, or

 \item\label{item:largeunbounded} there is $\beta \in \Q$ with $D(\alpha)$ dividing $D(\beta)$ and  a face $F$ of $\Contrib(\beta)$ with $|\Unb(C_F)| = n-1$. 

\end{enumerate}  
Then there is a set of candidate poles $\mathcal{P} \subset \mathbb{Q}$ for $Z_{\mot}(T)$ such that for all $\alpha \in \mathcal{P}$, $\exp(2 \pi i \alpha)$ is a nearby eigenvalue of monodromy. 
\end{theorem}

To complete the proof of Theorem~\ref{thm:optimalclass} we need the following lemma.

\begin{lemma}\label{lem:largeunbounded}
Let $\alpha \in \Q$.
Suppose there is $\beta \in \Q$ with $D(\alpha)$ dividing $D(\beta)$ and a face $F \in \Contrib(\beta)$ with 
$|\Unb(C_F)| = n - 1$. 
Then $\exp(2 \pi i \alpha)$ is a nearby eigenvalue of monodromy. 
\end{lemma}

\begin{proof}
Let $F$ be a face in $\Contrib(\beta)$ with $|\Unb(C_F)| = n-1$. 
Recall that we may write $\langle \Unb(C_F) \rangle = \R^{I_F}_{\ge 0}$ for some 
$I_F \subset [n]$, and that $\overline{F}$ denotes the image of $F$ under the projection $\R^n \to \R^n/\langle \Unb(C_F) \rangle$. 
Then $\overline{F} = \{\rho_F\} \subset \mathbb{R}$, where $\rho_F$ is the lattice distance of $F$ to the origin.  Observe that $\beta = -1/\rho_F$ and $D(\beta) = \rho_F$.
Let $x = x_{I_F}$ be a general point in $\A^{I_F} \subset X_f$. 
By either Varchenko's theorem (see \eqref{e:varchenko}) or  
 Theorem~\ref{t:nonnegativeVarchenko}, 
we compute
\[
\widetilde{E}(\cF_{x}) + 1  = \sum_{i = 0}^{\rho_F - 1} [i/D(\beta)]. \qedhere
\]
\end{proof}

\begin{proof}[Proof of Theorem~\ref{thm:optimalclass}]
The result is an immediate consequence of Theorem~\ref{thm:eigenvalue}, Theorem~\ref{thm:nopole} and Lemma~\ref{lem:largeunbounded}.
\end{proof}

On the other hand, there are three major 
classes of Newton polyhedra that Theorem~\ref{thm:optimalclass} does not cover. 

\begin{enumerate}

\item\label{i:missing1} All faces of $\Contrib(\alpha)$ are $\UB$, but $\Newt(f)$ is not $\alpha$-simplicial.

\item\label{i:missing2} Every face of $\Contrib(\alpha)$ is pseudo-$\UB$,
    and at least one face  of $\Contrib(\alpha)$ is not $\UB$.
    
\item\label{i:missing3} There is a face of $\Contrib(\alpha)$ that is not $\UB$, and $\Newt(f)$ does not have $D(\alpha)$-good projection. 

\end{enumerate}

For \eqref{i:missing1}, see \cite[Theorem 4.3]{ELT} and \cite[Theorem A]{Quek22} for results that produce “fake poles” of the topological and naive motivic zeta function under certain conditions but without an $\alpha$-simplicial assumption. 
For \eqref{i:missing2}, a $B_2$-facet in the sense of \cite[Definition 3.9]{ELT} is pseudo-$\UB$ but not $\UB$. It is known that $B_2$-facets sometimes do not give rise to poles of the local topological zeta function, see, e.g., \cite[Proposition 3.11]{ELT}. 
 For \eqref{i:missing3},  see \cite{Esterov21} for one approach to proving that $\exp(2 \pi i \alpha)$ is an eigenvalue of monodromy in this situation. 
See Example~\ref{ex:ELTprop} and Example~\ref{e:ELT} for explicit examples. Note that \eqref{i:missing3} does not occur when $\Newt(f)$ is convenient.

\subsection{Dimension 3 case}

We now use Theorem~\ref{thm:optimalclass} to deduce the local motivic monodromy conjecture for nondegenerate singularities when $n=3$ and prove the theorem below. See Section~\ref{ssec:history} for a history of prior results on monodromy conjectures for nondegenerate singularities when $n=3$.

\begin{theorem}\label{thm:n=3}
Suppose that $f$ is a nondegenerate polynomial in three variables. Then there is a set of candidate poles $\mathcal{P} \subset \mathbb{Q}$ for $Z_{\mot}(T)$ such that for all $\alpha \in \mathcal{P}$, $\exp(2 \pi i \alpha)$ is a nearby eigenvalue of monodromy. 
\end{theorem}

\begin{lemma}\label{lem:B1simplicial}
Suppose $n = 3$ and  $F$ is a $B_1$-face. Then $C_F$ is simplicial.
\end{lemma}
\begin{proof}
Let $A$ be an apex with base direction $e_\ell^*$. Then $C_F \cap \{ e_\ell^* = 0 \}$ is a polyhedral cone of dimension at most $2$, and hence is simplicial. Since $C_F$ is spanned by $C_F \cap \{ e_\ell^* = 0 \}$
and the ray through $A$, it follows that $C_F$ is simplicial.
\end{proof}

\begin{proof}[Proof of Theorem~\ref{thm:n=3}]
Let $\alpha \in \Q$ be a candidate pole. If $\alpha \in \Z$, then $1$ is an eigenvalue of monodromy for $H^0(\cF_0, \mathbb{C})$. Hence, we may assume that $\alpha \notin \Z$.  
Similarly, we may assume that $X_f$ is not smooth at the origin, else $\{ -1 \}$ is a set of candidate poles for $Z_{\mot}(T)$. 
We show that if $\Newt(f)$ does not satisfy condition (\ref{item:goodproj}) or (\ref{item:largeunbounded}) of Theorem~\ref{thm:optimalclass}, then it satisfies condition (\ref{item:nopole}). 
By Lemma~\ref{lem:largeunbounded}, we may therefore assume that, for all $\beta$ with $D(\alpha)$ dividing $D(\beta)$ and all $F \in \Contrib(\beta)$, $|\Unb(C_F)| \le 1$. Then $\Newt(f)$ has $D(\alpha)$-good projection, so we may assume that for every face $F$ of $\Contrib(\alpha)$, 
$F$ is pseudo-$\UB$.

We now argue that every face $F$ of $\Contrib(\alpha)$ is $\UB$. Then by Lemma~\ref{lem:B1simplicial}, $\Newt(f)$ is $\alpha$-simplicial and we have verified that condition (\ref{item:nopole}) of Theorem~\ref{thm:optimalclass} is satisfied. 

If $\dim F \le 1$, then $F$ is simplicial and hence is $\UB$. First suppose that $\dim F = 2$ and $F$ is compact. Choose two vertices $w_1$ and $w_2$ that lie on a $2$-dimensional face of $\R^3_{\ge 0}$
(if they do not exist, then no triangulation contains any $\UB$-facet), say $\{e_1^* = 0\}$. 
Observe that for $j \in \{ 2,3 \}$, either 
$\langle e_j^*, w_1 \rangle > 0$ or $\langle e_j^*, w_2 \rangle > 0$, else 
 one of $w_1$ or $w_2$ would be in the upper convex hull of the other.

If $F$ is not $\UB$, there are at least two other vertices, $w_3$ and $w_4$. We now consider two cases. 
\begin{enumerate}
\item First consider the case when $w_3$ and $w_4$ both have $\langle e_1^*, w_i \rangle = 1$. 
Consider the $2$-dimensional $\UB$ lattice simplex with vertices 
$w_1, w_3, w_4$. 
We may assume that there is an apex with base direction $e_2^*$. 
If the apex is $w_1$, then $w_3$ and $w_4$ must be of the form $(1, 0, c)$ for some $c \in \mathbb{N}$. But then one of $w_3$ or $w_4$ is in the upper convex hull of the other. 
Hence we may assume that the apex is $w_3$. 
Then, $w_1 = (0, 0, a)$ and $w_3 = (1, 1, b)$ for some $a, b \in \mathbb{N}$. 
Note that $\langle e_1^*, w_2 \rangle = 0$
implies that $\langle e_2^*, w_2 \rangle > 0$. 
Now consider the $2$-dimensional $\UB$ lattice simplex with vertices 
$w_2, w_3, w_4$. This has an apex at height $1$ with base direction $e_3^*$. As above, the apex can not be $w_2$, else  one of $w_3$ or $w_4$ is in the upper convex hull of the other. It also can not be $w_3$, else $b = 1$ and $\alpha = -1 \in \Z$. Hence $b = 0$ and the apex is $w_4$. 
Since $w_1,w_2,w_3,w_4$ all lie on $F$, we deduce that 
\[
w_1 = (0,0,a), w_2 = (0,a,0), w_3 = (1, 1, 0), w_4 = (1,0,1),
\]
for some $a \in \Z_{>0}$. Note that $a > 1$, else $X_f$ is smooth at the origin. Then $w = ((a - 1)/a)w_1 + (1/a)w_2 = (0,1,a - 1)$ is a lattice point in $F$, and the $2$-dimensional lattice simplex with vertices $w$,$w_3$,$w_4$ is not $\UB$, a contradiction.

\item Now suppose that there is some vertex $w_3$ 
such that $\langle e_1^*, w_3 \rangle > 1$. 
As $w_1, w_2, w_3$ span a $2$-dimensional $\UB$ lattice simplex and either $\langle e_2^*, w_1 \rangle > 0$ or   $\langle e_2^*, w_2 \rangle > 0$, we may assume that
$w_1$ is an apex with base direction $e_2^*$ and $\langle e_2^*, w_1 \rangle = 1$.
 The fourth vertex $w_4$ must have 
$\langle e_2^*, w_4 \rangle > 1$, 
 as otherwise we would be in the previous case.
Consider the $2$-dimensional $\UB$ lattice simplex
with vertices $w_2, w_3, w_4$. It cannot 
have an apex at height $1$ in either of the directions $e_1^*$ or $e_2^*$. 
Then we must have $w_2 = e_3$, which implies that $X_f$ is smooth at the origin, a contradiction.
\end{enumerate}
Now suppose that $\dim F = 2$ and $|\Unb(C_F)| = 1$. Then $C_F$ has good projection and $\overline{F}$ is $\UB$. 
By Remark~\ref{r:goodprojection}, $F$ is $\UB$. 
\end{proof}

\bibliographystyle{amsalpha}
\bibliography{Monodromy-Dec2025.bib}

\end{document}